\documentclass[a4paper,12pt]{article}
\usepackage{amsmath}
\usepackage{amssymb}
\usepackage{amsthm}

\newcommand{\C}{{\mathbb C}}
\newcommand{\R}{{\mathbb R}}
\newcommand{\Q}{{\mathbb Q}}
\newcommand{\Z}{{\mathbb Z}}

\newcommand{\g}{{\frak g}}

\newcommand{\p}{{\frak p}}
\newcommand{\fk}{{\frak k}}

\newcommand{\lv}{{\mathfrak l}}
\newcommand{\n}{{\mathfrak n}}
\newcommand{\m}{{\mathfrak m}}

\newcommand{\e}{{\bf e}}

\newcommand{\bA}{{\mathbb A}}

\newcommand{\Ind}{\operatorname{Ind}}
\newcommand{\Sym}{\operatorname{Sym}}

\newcommand{\Hom}{\operatorname{Hom}}
\newcommand{\Tr}{\operatorname{Tr}}
\newcommand{\tr}{\operatorname{tr}}
\newcommand{\vol}{\operatorname{vol}}

\newcommand{\tl}{\tau_{\Lambda}}

\newcommand{\cC}{{\cal C}}
\newcommand{\cD}{{\cal D}}
\newcommand{\cF}{{\cal F}}

\newcommand{\cH}{{\cal H}}

\newcommand{\cJ}{{\cal J}}

\newcommand{\cP}{{\cal P}}
\newcommand{\cS}{{\cal S}}

\newcommand{\cU}{{\cal U}}
\newcommand{\cW}{{\cal W}}

\newcommand{\diag}{\operatorname{diag}}
\numberwithin{equation}{section}
\theoremstyle{plain}
 \newtheorem{thm}{Theorem}[section]
 \newtheorem{prop}[thm]{Proposition}
 \newtheorem{lem}[thm]{Lemma}
 \newtheorem{cor}[thm]{Corollary}

 \newtheorem*{thmEZ}{Theorem}
 \newtheorem*{mainth}{Main theorem}
 
\theoremstyle{definition}
 \newtheorem{defn}[thm]{Definition}
 \newtheorem{rem}[thm]{Remark}
 
 \newtheorem*{asm}{Working Assumption}

\begin{document}
\begin{title}
{\Large\bf Fourier-Jacobi expansion of cusp forms on $Sp(2,\R)$}
\end{title}
\author{Hiro-aki Narita}
\maketitle
\begin{abstract}
This paper develops a general theory of the Fourier-Jacobi expansion of cusp forms on the real symplectic group $Sp(2,\R)$ of degree two including generic cusp forms. An explicit description of such expansion is available for cusp forms generating discrete series representations, generalized principal series representations induced from the Jacobi parabolic subgroup and principal series representations, where we note that the latter two cases include non-spherical representations. 

As the archimedean local ingredients we need Fourier-Jacobi type spherical functions and Whittaker functions, whose explicit formulas are obtained by Hirano and by Oda, Miyazaki-Oda, Niwa and Ishii respectively. To realize these spherical functions in the Fourier-Jacobi expansion we use the spectral theory for the Jacobi group by Berndt-B{\"o}cherer and Berndt-Schmidt, which can be referred to as the global ingredient of our study. Based on the theory by Berndt-B{\"o}cherer we generalize the classical Eichler-Zagier correspondence in the representation theoretic context. 
\end{abstract}

\tableofcontents
\clearpage

\section{Introduction}
\subsection{Background and the aim of the paper}
The Fourier-Jacobi expansion has been a powerful tool to study holomorphic Siegel modular forms. In fact, without the frequent use of the Fourier-Jacobi expansion, the current advancement of the studies on holomorphic Siegel modular forms would be impossible. The studies of holomorphic Siegel modular forms by means of the Fourier-Jacobi expansion has a long and rich story. It is thus difficult to review relevant works thoroughly. We cite B{\"o}cherer \cite{Boe}, Eichler-Zagier \cite{E-Z}, 
Ikeda \cite{Ik-1},~\cite{Ik-2} for instance. See also Piatetskii-Shapiro \cite{Ps-1}, which explains the notion of the Fourier-Jacobi expansion for the case of general symmetric domains.

On the other hand, the Fourier-Jacobi expansion can be defined for general automorphic forms not restricted to holomorphic ones. However, there seems to be quite a few trials to study such expansion in detail for non-holomorphic real analytic automorphic forms. We  cite a recent progress by Pollack et al. for automorphic forms generating quaternionic discrete series~(cf.~\cite{Po},~\cite{Nr-4}) and also cite Ishikawa \cite{Isy}. For the case of Siegel modular forms or automorphic forms on the real symplectic group $Sp(n,\R)$ we cay say that there is completely no such trial, even for $Sp(2,\R)$ of degree two. 
Now recall that the symplectic group of degree two has two maximal parabolic subgroups, up to conjugation. The Fourier-Jacobi expansion respects the maximal parabolic subgroup called the Jacobi parabolic subgroup~(or Klingen parabolic subgroup). The Fourier expansion along another maximal parabolic subgroup called the Siegel parabolic subgroup, has been a classical tool to study holomorphic Siegel modular forms. We remark that Moriyama \cite{Mo-2} discussed this expansion adelically for automorphic forms on the similitude group $GSp(2)$ in a rather abstract manner but in a general setting. The aim of this paper is to develop the notion of the Fourier-Jacobi expansion for general cusp forms on $Sp(2,\R)$ including those generating irreducible admissible representations admitting Whittaker models, which we call generic cusp forms.

There have been active representation theoretic studies on generic cuspidal representations, namely cuspidal automorphic representations admitting (global) Whittaker models. For instance the Langlands-Shahidi method on automorphic $L$-functions works well for general treatment of $L$-functions of generic automorphic representations. As fundamental references on this topics we cite Langlands \cite{La} and Shahidi \cite{Sha}. Though the representation theoretic studies of automorphic forms have produced many general and qualitative results, we have had only very little understanding of generic automorphic forms as functions. The author hopes that this paper will serve as a first step toward a detailed study of general cusp forms on $Sp(2,\R)$ including generic ones in terms of their functional aspects. 

We choose the Siegel modular group $Sp(2,\Z)$ to define cusp forms on $Sp(2,\R)$. We express the Fourier-Jacobi expansion of  a cusp form $F$ with respect to $Sp(2,\Z)$ as
\[
F(g)=\sum_{m\in\Z}F_m(g),~F_m(g):=\int_{\R/\Z}F(n(0,u_1,0)g)\exp(-2\pi\sqrt{-1}mu_1)du_1,
\]
where $\{n(0,u_1,0)\mid u_1\in\R\}$ is the center of the unipotent radical $N_J$ of the Jacobi parabolic subgroup $P_J$~(see Section \ref{Groups} for $P_J$ and $N_J$). The group $N_J$ is well known as a Heisenberg group. To establish such Fourier-Jacobi expansion the notion of the Whitaker functions is necessary to understand both of the $F_0$-term and $F_m$-terms with $m\not=0$. To complete the description of $F_m$ for a non-zero $m$ we further need the Fourier-Jacobi type spherical functions studied by Hirano. 
\subsection{Fourier-Jacobi type spherical functions and Whittaker functions}
One indispensable step toward an explicit Fourier-Jacobi expansion is to determine what special functions are necessary to describe the expansion. For the case of holomorphic automorphic forms such special functions are characterized by differential equations of rank one arising from the Cauchy-Riemann condition. Such functions are described by elementary functions such as the exponential function. The functions of this sort are formulated as generalized spherical functions for holomorphic discrete series representations. 

However, for the case of non-holomorphic forms, it is not settled in such an elementary manner to understand the spherical functions appropriate for the Fourier expansion. For non-holomotphic automorphic forms on $Sp(2,\R)$ the notion of Fourier-Jacobi type spherical functions were introduced and studied in detail by Hirano \cite{Hi-1}, \cite{Hi-2} and \cite{Hi-3}~(cf.~Section \ref{Result-Hirano}). These works studied such spherical functions for a wide class of admissible representations; not only holomorphic discrete series representation but also large discrete series representations, $P_J$-principal series representations~(induced from the Jacobi parabolic subgroup $P_J$) and principal series representations~(induced from the minimal parabolic subgroup). According to \cite{Hi-1}, \cite{Hi-2} and \cite{Hi-3} the Meijer $G$-functions are useful to know the Fourier-Jacobi type spherical functions explicitly.  As we have remarked above we also need the Whittaker functions~(cf.~Section \ref{ReviewWhittaker}), whose explicit formulas are studied by Oda \cite{Od}, Miyazaki-Oda \cite{M-Od-1},~\cite{M-Od-2},~\cite{M-Od-3}, Ishii \cite{Is-2} and Niwa \cite{Nw} for the aforementioned three admissible representations of $Sp(2,\R)$ other than holomorphic discrete series. For this  we note the well known fact that the holomorphic discrete series representations are not generic. We remark that the Whittaker functions are viewed as the Fourier-Jacobi type spherical functions attached to non-trivial characters of $N_J$ as was pointed out by Hirano \cite[Section 8]{Hi-1}. 

Let us now note that we are left with the case of generalized principal series representations induced from the Siegel parabolic subgroup. For this case the Fourier-Jacobi type spherical functions have not been studied yet. However, an explicit formula for the Whittaker function has been recently obtained by Ishii \cite{Is-3}. 
The spherical functions we have mentioned explain the archimedean local aspect of the Fourier-Jacobi expansion.
\subsection{Generalized Eichler-Zagier correspondence}
Toward the Fourier-Jacobi expansion we need an appropriate global theory to realize the spherical functions as above in the expansion. One of our key ideas is the notion of ``generalized Eichler-Zagier correspondence''. To explain this in some detail, recall that we have let $P_J$ be the Jacobi parabolic subgroup of $Sp(2,\R)$ with the unipotent radical $N_J$ known as a Heisenberg group. We then introduce the Jacobi group $G_J:=N_J\rtimes SL_2(\R)$, which is a subgroup of  $P_J$. Let $\pi_1$ be an irreducible unitary genuine representation of the non-split two fold cover $\widetilde{SL}_2(\R)$ of $SL_2(\R)$, where recall that a representation of $\widetilde{SL}_2(\R)$ is called genuine if it does not factor through $SL_2(\R)$. Let $\omega_m$ be the Weil representation of $\widetilde{SL}_2(\R)$ extended from the Stone von-Neumann representation $\nu_m$ with the central character indexed by $m\not=0$. According to Berndt-Schmidt \cite[Theorem 2.6.1]{Be-Sc} and Satake \cite[Appendix I,~Proposition 2]{Sa} every irreducible unitary representation is of the form 
\[
\rho_{\pi_1,m}:=\pi_1\otimes\omega_m
\]
with some genuine representation $\pi_1$ of $\widetilde{SL}_2(\R)$. For this we note that $\rho_{\pi_1,m}|_{\widetilde{SL}_2(\R)}$ factors though $SL_2(\R)$ and is thus well-defined as a representation of $G_J:=N_J\rtimes SL_2(\R)$. For a non-zero integer $m$ and an irreducible genuine unitary representation $\pi_1$ of $\widetilde{SL}_2(\R)$ we then define the space of Jacobi cusp forms of type $\pi_1$ and index $m$ by 
\[
\Hom_{G_J}(\rho_{\pi_1,m},\cH_m^0)
\]
with the cuspidal subspace $\cH_m^0$~(cf.~Section \ref{EZcorrespondence}) of 
\[
\cH_m:=\{\phi\in L^2(G_J(\Z)\backslash G_J)\mid~\text{$\phi|_{N_J}$ has the central character indexed by $m$}\},
\]
where $G_J(\Z):=G_J\cap Sp(2,\Z)$. On the other hand, letting $\Phi_m:=\Hom_{N_J}(\nu_m,L^2(N_J(\Z)\backslash N_J))$ 
with $N_J(\Z):=N_J\cap Sp(2,\Z)$ for a non-zero integer $m$, we introduce the space 
\[
\cS_{\pi_1}(\widetilde{SL}_2(\Z),\Phi_m)
\]
of $\Phi_m$-valued cusp forms on $\widetilde{SL}_2(\R)$ with respect to the double cover $\widetilde{SL}_2(\Z)$ of $SL_2(\Z)$~(cf.~Definition \ref{Vector-valuedCuspfm}), whose coefficient functions generate $\pi_1$. We then have a representation theoretic generalization of the Eichler-Zagier correspondence~(cf.~Theorem \ref{Eichler-ZagierCorresp}) as follows:
\begin{thmEZ}[Generalized Eichler-Zagier correspondence]
For a non-zero integer $m$ and an irreducible genuine unitary representation $\pi_1$ of $\widetilde{SL}_2(\R)$ we have an isomorphism
\[
\Hom_{G_J}(\rho_{\pi_1,m},\cH_m^0)\simeq \cS_{\pi_1}(\widetilde{SL}_2(\Z),\Phi_m).
\]
\end{thmEZ}
\noindent
When $\pi_1$ is a holomorphic discrete series representation this is the usual Eichler-Zagier correspondence between the vector-valued elliptic cusp forms of half-integral weights and holomorphic Jacobi cusp forms~(cf.~\cite[Theorem 5.1]{E-Z}). When $\pi_1$ is a unitary principal series representation $\Hom_{G_J}(\rho_{\pi_1,m},\cH_m^0)$ is the equivalent notion of Maass Jacobi cusp forms. 
This theorem is inspired by the notion of the theta decomposition of Jacobi forms~(cf.~\cite[Section 5]{E-Z}). 
There seem many related works on this decomposition other than \cite{E-Z}. For instance we cite Pitale \cite{Pi} and Bringmann-Raum-Richter \cite{B-R-R}. As another relevant work we refer to Section 11 of Takase \cite{Ta}, which takes up some representation theoretic treatment of the Eichler-Zagier correspondence for the holomorphic case.
\subsection{Fourier-Jacobi expansion}\label{Intro-FJ}
We are now in a position to discuss the Fourier-Jacobi expansion of cusp forms on $Sp(2,\R)$. It is well known that the cuspidal spectrum of a reductive group decomposes into a discrete sum of irreducible admissible representations with finite multiplicities~(cf.~\cite{G-G-P},~\cite{Gd}). 
In view of this  we assume that the admissible representations which cusp forms generate are irreducible~(cf.~Definition \ref{Def-cuspforms}). 

Let $\pi$ be an irreducible admissible representation of $Sp(2,\R)$. Recall that an irreducible representation $\tau$ of the maximal compact subgroup $K$ of $Sp(2,\R)$ is called a $K$-type of $\pi$ if $\tau$ occurs in the restriction of $\pi$ to $K$. We make the working assumption on the space $\cJ_{\rho,\pi}(\tau^*)^{00}$~(respectively~$W_{\psi,\pi}(\tau^*)^0$) of the rapidly decreasing Fourier-Jacobi type spherical functions~(respectively~rapidly deceasing Whittaker functions), where $\tau^*$~(respectively~$\psi$) denotes the contragredient of $\tau$~(respectively~a unitary character of the maximal unipotent subgroup $N_0$ introduced in Section \ref{Groups}). We have to assume the multiplicity one property of the $K$-type $\tau$ in $\pi$ to have the well-defined notion of the spherical functions above. 
For the definitions of $W_{\psi,\pi}(\tau^*)^0$ and $\cJ_{\rho,\pi}(\tau^*)^{00}$ see Sections \ref{ReviewWhittaker},~\ref{Result-Hirano}. 
In addition we note that the notation $\cJ_{\rho,\pi}(\tau^*)^{00}$ is to avoid the confusion with Hirano's notation $\cJ_{\rho,\pi}(\tau^*)^{0}$ for the space of moderate growth Fourier-Jacobi type spherical functions~(cf.~\cite{Hi-1},~\cite{Hi-2} and \cite{Hi-3}).
\begin{asm}
There is a multiplicity one $K$-type $\tau$ of $\pi$ such that
\begin{itemize}
\item there is no rapidly decreasing element in $W_{\psi,\pi}(\tau^*)^0$ when $\psi$ is degenerate~(cf.~Section \ref{Deg-Whittaker-sec}),
\item $\dim\cJ_{\rho,\pi}(\tau^*)^{00}\le 1$ holds for any irreducible unitary representations $\rho$ of $G_J$ with the non-trivial central character.
\end{itemize}
\end{asm}
We remark that discrete series representations, irreducible $P_J$-principal series representations and irreducible principal series representations with the conditions (\ref{PS-Whittaker}) are verified to satisfy this assumption~(see Sections \ref{ReviewSpherical} and \ref{GCusp-WA}). We remark that the condition $\dim W_{\psi,\pi}(\tau^*)^0\le 1$ holds for general irreducible admissible representations of real quasi split groups and for non-degenerate $\psi$ in view of Wallach \cite[Theorem 8.8]{W-1}. As for the assumption $\dim\cJ_{\rho,\pi}(\tl^*)^{00}\le 1$ we cite Liu-Sun \cite{Li-Su}, which leads us to expect that the assumption would have  generality. 

Toward the statement of the main theorem~(cf.~Theorem \ref{F-J-exp-mainthm}) we prepare a couple of ingredients for the Fourier-Jacobi expansion of a cusp form $F$ generating an irreducible admissible representation $\pi$.
\begin{itemize}
\item {\bf Whittaker functions $F_{\xi_0,\xi_3}$ of $F$.}\\
We first let 
\[
F_{\xi_0,\xi_3}(g):=\displaystyle\int_{N_0(\Z)\backslash N_0}F(n(u_0,u_1,u_2,u_3)g)\psi_{\xi_0,\xi_3}(n(u_0,u_1,u_2,u_3))^{-1}dn\quad(g\in G)
\]
for the unitary character 
\[
\psi_{\xi_0,\xi_3}:N_0\ni n(u_0,u_1,u_2,u_3)\mapsto\exp(2\pi\sqrt{-1}(\xi_0u_0+\xi_3u_3))
\]
of $N_0$ with $(\xi_0,\xi_3)\in\Z^2$, which is invariant with respect to $N_0(\Z):=N_0\cap Sp(2,\Z)$. Here  $dn$ denotes the invariant measure of $N_0(\Z)\backslash N_0$ normalized so that ${\rm vol}(N_0(\Z)\backslash N_0)=1$. Due to $\dim W_{\psi,\pi}(\tau^*)^0\le 1$ remarked above, $F_{\xi_0,\xi_3}$ is a constant multiple of the Whittaker function when $\xi_0\xi_3\not=0$.
\item {\bf Eisenstein-Poincar{\'e} series $E_{\alpha}(F_{S_{\alpha,m},n}(*g))(r)$ on $G_J$.}\\
For a non-zero integer $m$, $1\le \alpha\le 2|m|$ with $\frac{\alpha^2}{4m}\in\Z$, and a non-zero element $n$ in the $\Z$-submodule  $\widehat{L_{\alpha,m}}\subset\Q$~(for $\widehat{L_{\alpha,m}}$ see Section \ref{FJ-expansion}), we denote by 
\[
E_{\alpha}(F_{S_{\alpha,m},n}(*g))(r)
\]
with a fixed $g\in Sp(2,\R)$ the Eisenstein-Poincar{\'e} series on $G_J$ with the test function $F_{S_{\alpha,m},n}(*g)$~(cf.~Section \ref{FJ-expansion}). Here $F_{S_{\alpha,m},n}$ is proved to be the left translation of $F_{m,n}$~($n\in\widehat{L_{\alpha,m}}\setminus\{0\}$) by some element of $Sp(2,\Q)$ in which a representative of cusps of the Jacobi group~(cf.~\cite[Section 4.2]{Be-Sc},~Section \ref{FJ-expansion}) is involved.
\item {\bf Fourier-Jacobi type spherical functions.}\\
With the notation $\{w_l\}_{l\in L},~\{u_j^m\}_{j\in J},~\{v_k^*\}_{0\le k\le d_{\Lambda}}$ for a basis of the representation space of $ \pi_1$,~$\nu_m$~(or $\omega_m$),~$\tau^*$ respectively~(cf.~Sections \ref{RepSL2},~\ref{UnitaryRepJacobi},~\ref{RepMaxCpt}), the restriction of $W\in\cJ_{\rho,\pi}(\tau^*)^{00}$ to $A_J$~(the torus part of the Langlands decomposition of $P_J$) is written as
\[
W(a_J)=\sum_{
\begin{subarray}{c}
j\in J,~0\le k\le d_{\Lambda}\\
\text{s.t.}~l=l(j,k)\in L
\end{subarray}}c_{j,k}^{(\pi_1)}(a_J)w_l\otimes u_j^m\otimes v_{k}^*\quad(a_J\in A_J)
\]
with coefficient functions $c_{j,k}^{(\pi_1)}$. 
Here $l(j,k)=-j+k+\Lambda_2$ when the highest weight of $\tau^*$ is $(-\Lambda_2,-\Lambda_1)$. The Fourier-Jacobi type spherical functions are realized in the Fourier-Jacobi expansion by a basis of $\Hom_{G_J}(\rho_{m,\pi_1},\cH_m^0)$. They are denoted by 
\[
F_{m,i}^{(\pi_1)}\quad(1\le i\le \dim \Hom_{G_J}(\rho_{m,\pi_1},\cH_m^0))
\]
as in the following theorem.
\end{itemize}
\begin{mainth}[Fourier-Jacobi expansion]
Let $\pi$ be an irreducible admissible representation of $Sp(2,\R)$ with the multiplicity one $K$-type $\tau$ satisfying the working assumption and let $F$ be a cusp form of weight $\tau^*$ with respect to $Sp(2,\Z)$ generating $\pi$~(cf.~Definition \ref{Def-cuspforms}). 

Each term $F_m$ of the Fourier-Jacobi expansion $\sum_{m\in\Z}F_m$ of $F$ is expressed as 
\[
F_m(ra_J)=\begin{cases}
\sum_{(\xi_0,\xi_3)\in\Z^2,~\xi_0\xi_3\not=0}
\sum_{
\begin{pmatrix}
a & b\\
c & d
\end{pmatrix}\in SL_2(\Z)_{\infty}\backslash SL_2(\Z)}F_{\xi_0,\xi_3}(
\begin{pmatrix}
1 & 0 & 0 & 0\\
0 & a & 0 & b\\
0 & 0 & 1 & 0\\
0 & c & 0 & d
\end{pmatrix}ra_J)&(m=0)\\
\sum_{
\begin{subarray}{c}
1\le\alpha\le 2|m|\\
\text{s.t.~$\alpha^2/4m\in\Z$}
\end{subarray}}\sum_{n\in\widehat{L_{\alpha,m}}\setminus\{0\}}b_{m,\alpha}(F)E_{\alpha}(F_{S_{\alpha,m},n}(*a_J))(r)+\\
\underset{\pi_1\in\widehat{\widetilde{SL}_2(\R)},m(\pi_1)\not=0}{\sum}\sum_{i=1}^{m(\pi_1)}b_{m,i}^{(\pi_1)}(F)F_{m,i}^{(\pi_1)}(ra_J)&(m\not=0)
\end{cases}
\]
for $(a_J,r)\in A_J\times G_J$, with
\[
F_{m,i}^{(\pi_1)}(ra_J):=
\underset{\begin{subarray}{c}
j\in J,~0\le k\le d_{\Lambda}\\
\text{s.t.}~l=l(j,k)\in L
\end{subarray}}{\sum}c_{j,k}^{(\pi_1)}(a_J)\phi_{\pi_1}^{(i)}(w_l\otimes u_j^m)(r)\otimes v_k^*.
\]
Here
\begin{itemize}
\item $SL_2(\Z)_{\infty}:=\left\{\left.
\begin{pmatrix}
1 & n\\
0 & 1
\end{pmatrix}~\right|~n\in\Z\right\}$,
\item$\m(\pi_1):=\dim\Hom_{G_J}(\rho_{m,\pi_1},\cH_m^0)$ and $\{\phi^{(i)}_{\pi_1}\}$ is a basis of $\Hom_{G_J}(\rho_{m,\pi_1},\cH_m^0)$,
\item $b_{m,i}^{(\pi_1)}(F)$~(respectively~$b_{m,\alpha}(F)$) is a constant depending on the normalization of the Fourier-Jacobi type spherical function for $\rho=\rho_{m,\pi_1}$ and on a choice of a basis $\{\phi_{\pi_1}^{(i)}\}$ (respectively~on $m$ and $\alpha$).
\end{itemize}
\end{mainth}
We note that $F_m(rg)$ with a fixed $g\in G$ belongs to $\cH_m$ as a function in $r\in G_J$ for $m\not=0$. This $L^2$-space $\cH_m$ decomposes into the continuous spectrum ${\cH}^{\rm c}_m$ and the cuspidal spectrum ${\cH}^0_m$. According to this decomposition we have 
\[
F_m=F^{\rm c}_m+F^0_m
\]
with $F^{\rm c}_m\in\cH^{\rm c}_m$ and $F^0_m\in{\cH^0_m}$. In general, we cannot deny the contribution to the Fourier-Jacobi expansion by $F^{\rm c}_m$,  which is given as the sum of the Eisenstein-Poincar{\'e} series $E_{\alpha}(F_{S_{\alpha,m},n}(*a_J))(r)$s. The cuspidal part $F^0_m$ is written as the sum of images $F_{m,i}^{(\pi_1)}$ of the Fourier-Jacobi type spherical functions by the intertwining operators $\phi_{\pi_1}^{(i)}$~(cf.~Section \ref{ProofThm}).

As a further remark, we have $F_0\equiv 0$ and $F^{\rm c}_m\equiv 0$ for $m\not=0$ when $F$ generates a holomorphic or anti-holomorphic discrete series representation (i.e.~$F$ is a holomorphic or anti-holomorphic cusp form), for which recall that neither holomorphic discrete series nor anti-holomorphic discrete series are generic. When $F$ is holomorphic or anti-holomorphic we should furthermore note that $m(\pi_1)\not=0$ only if $\pi_1$ is a holomorphic or anti-holomorphic discrete series and that the summation over $J$~(and $L$) is reduced to a finite sum~(cf.~\cite[Theorems 6.3-6.6]{Hi-1},~Corollary \ref{FJ-exp-fourcases}~(1)). However, this is no longer true in general, which was pointed out by Ikeda \cite{Ik-1} for the case of Eisenstein series.
\subsection{Outline of the paper}
Let us explain the outline of the paper. Sections \ref{BN} and \ref{Rep-real} are to prepare the fundamental notion necessary for the subsequent argument. Section \ref{BN} collects basic notations for real groups and discrete groups. Section \ref{Rep-real} is devoted to introducing the representations of the real groups we need. In Section \ref{ReviewSpherical} we review results on the explicit formulas for the Whittaker functions and the Fourier-Jacobi type spherical functions. This section includes all the known explicit formulas relevant to the Fourier-Jacobi expansion and also takes up Whittaker functions attached to degenerate characters of the maximal unipotent subgroup of $Sp(2,\R)$, the latter of which have not drawn attention of experts. This section is indispensable for understanding the Fourier-Jacobi expansion explicitly. 
As well as reviewing the known results we remark that some idea of changes of variables leads to a reduction of the problems for the explicit formulas, which has not been pointed out in the literature. In Section \ref{EZcorrespondence} the notion of the generalized Eichler-Zagier correspondence is introduced and studied in detail. 
In Section \ref{FJ-exp} we are then able to prove the main theorem on the Fourier-Jacobi expansion. In Corollary \ref{FJ-exp-fourcases} we write down the Fourier-Jacobi expansion in a more specific manner for discrete series, $P_J$-principal series and principal series. In this section we also provide some explicit description of $F_m$-terms with $m\not=0$ of the Fourier-Jacobi expansion~(Section \ref{Explicit-FJ}). More precisely we explicitly describe Jacobi cusp forms $\phi_{\pi_1}^{(i)}(w_l\otimes u_j^m)$ for some specified $(l,j)$ and see how they contribute to such $F_m$s with the explicit Fourier-Jacobi type spherical functions obtained by Hirano. It is now worthwhile to make remarks as follows:
\begin{itemize}
\item Skew Jacobi forms introduced by Skoruppa \cite{Sk} have no contribution to the Fourier expansion of holomorphic or anti-holomorphic Siegel modular forms. They contribute to the Fourier-Jacobi expansions for the non-holomorphic cases. We provide such examples for the case of the large discrete series representations~(cf.~Section \ref{Explicit-FJ} (II-2),~(II-3)). In addition to this we refer to interesting relevant works by Miyazaki \cite{Mi-2} and Raum-Richter \cite{Ra-Ri} though they do not deal with cusp forms.
\item The Fourier-Jacobi expansion of the holomorphic case in terms of the Fourier-Jacobi type spherical functions~(cf.~Corollary \ref{FJ-exp-fourcases} (1)) needs only a finitely many $\Hom_{G_J}(\rho_{m,\cD_{n_1}},\cH_m^0)$ of discrete series $\cD_{n_1}^+$ in order to describe each $F_m$-term, as has been pointed out  above. However, reviewing the classical Fourier expansion in terms of the Fourier expansion along the minimal parabolic subgroup \cite{Nr-1}, it turns out  that only a single $\Hom_{G_J}(\rho_{m,\cD_{n_1}},\cH_m^0)$ is enough. In Section \ref{Explicit-FJ} (I-2) this is explained by $\cS_{\cD_{n_1}}(\widetilde{SL}_2(\Z),\Phi_m)$ instead of $\Hom_{G_J}(\rho_{m,\cD_{n_1}},\cH_m^0)$. We note that this remark is valid for vector valued holomorphic cusp forms as well as scalar valued ones.

As an additional remark on this there is no result on the Fourier expansion along the minimal parabolic subgroup for non-holomorphic automorphic forms on $Sp(2,\R)$. Such Fourier expansion needs the notion of generalized Whittaker functions for infinite dimensional representations of the maximal unipotent subgroup $N_0$. It seems that it is  difficult to determine them explicitly especially for non-holomorphic cases, and indeed there is no such result in the literature. 
In the current research status we cannot therefore carry out the observation mentioned above for non-holomorphic cases. 
\end{itemize}
\subsection*{Acknowledgement}
The author would like to thank Prof. Taku Ishii for his significant help on detailed studies about Whittaker function attached degenerate characters of the maximal unipotent subgroup. He would also like to thank Dr. Shuji Horinaga for pointing out a misleading about the proof of Proposition \ref{DegnerateWhittaker} for an old version of the paper. The author's gratitude is also due to Prof. Miki Hirano, Prof. Tadashi Miyazaki and Prof. Masao Tsuzuki for their fruitful comments on the Fourier-Jacobi type spherical functions and the Whittaker functions. 
\section{Basic Notation}\label{BN}
\subsection{Basic notation for real groups and discrete groups}\label{Groups}
\subsection*{(1)~Real groups}
Let $G=Sp(2,\R)$ be the real symplectic group of degree two defined by
\[
\{g\in GL_4(\R)\mid {}^tgJ_4g=J_4\},
\]
where $J_4=
\begin{pmatrix}
0_2 & 1_2\\
-1_2 & 0_2
\end{pmatrix}$. This has a maximal parabolic subgroup $P_J$ of $G$ given by 
the Levi decomposition $N_J\rtimes L_J$. Here $N_J$ is the nilpotent Lie group defined by
\[
\left\{\left. n(u_0,u_1,u_2):=
\begin{pmatrix}
1 & 0 & u_1 & u_2\\
0 & 1 & u_2 & 0\\
0 & 0 & 1 & 0\\
0 & 0 & 0 & 1
\end{pmatrix}
\begin{pmatrix}
1 & u_0 & 0 & 0\\
0 & 1 & 0 & 0\\
0 & 0 & 1 & 0\\
0 & 0 & -u_0 & 1
\end{pmatrix}~\right|~u_0,~u_1,~u_2\in\R\right\}
\]
and the Levi part $L_J$ is the subgroup of $G$ given by
\[
\left\{\left.
\begin{pmatrix}
\alpha & & & \\
& a & & b\\
& & \alpha^{-1} & \\
& c & & d
\end{pmatrix}~\right|~\alpha\in\R^{\times},~
\begin{pmatrix}
a & b\\
c & d
\end{pmatrix}\in SL_2(\R)\right\}.
\]
We call this the Jacobi parabolic subgroup (also called the Klingen parabolic subgroup). 

The unipotent radical $N_J$ of $P_J$ is nothing but the Heisenberg group with the center  
\[
Z_J:=\{n(0,u_1,0)\mid u_1\in\R\}.
\] 
We introduce the non-reductive real group $G_J$ called the Jacobi group. This is defined by the semi-direct product 
\[
N_J\rtimes SL_2(\R),
\]
which is given by replacing $L_J$ with the special linear group $SL_2(\R)$ in $P_J$. More precisely, $SL_2(\R)$ is viewed as a subgroup of $G_J$~(or $L_J$) by putting $\alpha=1$ in $L_J$. The group $G_J$ is the centralizer of $Z_J$ in $P_J$. 
The parabolic subgroup $P_J$ has the Langlands decomposition $P_J=N_JA_JM_J$ with 
\[
A_J:=\left\{\left.a_J=
\begin{pmatrix}
a_1 & 0 & 0 & 0\\
0 & 1 & 0 & 0\\
0 & 0 & a_1^{-1} & 0\\
0 & 0 & 0 & 1
\end{pmatrix}~\right|~a\in\R^{\times}_+\right\},~
M_J:=\left\{\left.
\begin{pmatrix}
\epsilon & 0 & 0 & 0\\
0 & a & 0 & b\\
0 & 0 & \epsilon & 0\\
0 & c & 0 & d
\end{pmatrix}~\right|~
\begin{array}{c}
\begin{pmatrix}
a & b\\
c & d
\end{pmatrix}\in SL_2(\R)\\
\epsilon\in\{\pm 1\}
\end{array}\right\}.
\]

We also need the minimal parabolic subgroup $P_0$ of $G$ with the unipotent radical $N_0$, where $N_0$ is defined by 
\[
\left\{\left.n(u_0,u_1,u_2,u_3)\in
\begin{pmatrix}
1 & 0 & u_1 & u_2\\
0 & 1 & u_2 & u_3\\
0 & 0 & 1 & 0 \\
0 & 0 & 0 & 1
\end{pmatrix}
\begin{pmatrix}
1 & u_0 & 0 & 0\\
0 & 1 & 0 & 0\\
0 & 0 & 1 & 0\\
0 & 0 & -u_0 & 1
\end{pmatrix}~\right|~n_i\in\R~(0\le i\le 3)\right\}. 
\]
We also review the Langlands decomposition $P_0=N_0A_0M_0$ of the minimal parabolic subgroup $P_0$, where
\[
A_0:=\{a_0=\diag(a_1,a_2,a_1^{-1},a_2^{-1})\mid a_1,~a_2\in\R_{>0}\},~M_0:=\{\diag(\epsilon_1,\epsilon_2,\epsilon_1,\epsilon_2)\mid \epsilon_1,~\epsilon_2\in\{\pm 1\}\}.
\]
The group $N_0$ admits the semi-direct product decomposition $N_0=N_S\rtimes N_L$ with the subgroups $N_S$ and $N_L$ defined by
\[
N_S=\{n(u_0,u_1,u_2,u_3)\in N_0\mid u_0=0\},\quad N_L:=\{n(u_0,0,0,0)\mid u_0\in\R\}.
\]
We remark that $N_S$ is well known as the unipotent radical of the Siegel parabolic subgroup. 

Let us introduce the Cartan involution $\theta$ of $G$ defined by $\theta(g):={}^tg^{-1}$ for $g\in G$. Then 
\[
K:=\{g\in G\mid\theta(g)=g\}=\left\{\left.
\begin{pmatrix}
A & B\\
-B & A
\end{pmatrix}\in G~\right|~A,B\in M_2(\R)\right\}
\]
is a maximal compact subgroup of $G$. This is isomorphic to the unitary group $U(2)$ of degree two  by the map
\[
K\ni
\begin{pmatrix}
A & B\\
-B & A
\end{pmatrix}\mapsto A+\sqrt{-1}B\in U(2).
\]
We should remark that $G$ has an Iwasawa decomposition $G=N_0A_0K$ with the notation above.

We furthermore introduce the real group $\widetilde{SL_2}(\R)$ characterized by the non-split exact sequence
\[
1\rightarrow\{\pm 1\}\rightarrow\widetilde{SL_2}(\R)\rightarrow SL_2(\R)\rightarrow 1,
\]
namely, the non-split double cover of $SL_2(\R)$. This is realized by means of the unique non-trivial element of the second cohomology $H^2(SL_2(\R),\{\pm1\})$ called the Kubota cocycle~(cf.~\cite{Ku}).
The group $SL_2(\R)$ is known to have the special orthogonal group $SO_2(\R)$ as a maximal compact subgroup and $\widetilde{SL}_2(\R)$ has a maximal compact subgroup $\widetilde{SO_2}(\R)$, the non-split two fold cover of $SO_2(\R)$.  The group $\widetilde{SO_2}(\R)$ is given as the inverse image of $SO_2(\R)$ by the covering map $\widetilde{SL_2}(\R)\rightarrow SL_2(\R)$.
\subsection*{(2)~Discrete subgroups}
We explain the notation for the discrete subgroups of the real groups above. As the most fundamental notation we introduce the Siegel modular group 
\[
Sp(2,\Z):=G\cap GL_4(\Z).
\]
In addition to this we will need \[
G_J(\Z):=G_J\cap Sp(2,\Z),
\]
$N_J(\Z):=N_J\cap Sp(2,\Z)$,~$Z_J(\Z):=Z_J\cap Sp(2,\Z)$, $N_S(\Z):=N_S\cap Sp(2,\Z)$ etc. We remark that $G_J(\Z)=N_J(\Z)\rtimes SL_2(\Z)$, where $SL(2,\Z)$ is viewed as a subgroup of $G_J(\Z)$ by 
\[
\left\{\left.
\begin{pmatrix}
1 & 0 & 0 & 0\\
0 & a & 0 & b\\
0 & 0 & 1 & 0\\
0 & c & 0 & d
\end{pmatrix}~\right|~
\begin{pmatrix}
a & b\\
c & d
\end{pmatrix}\in SL_2(\Z)\right\}.
\]
\subsection{Lie algebras and root system}\label{Lie-gp-alg}
Following the standard manner of the notation, we denote the Lie algebras of real Lie groups by the corresponding German letter~(Fraktur). For a real Lie algebra $\lv$ we denote its complexification by $\lv_{\C}$. 

The Lie algebra $\g$ of $G$ is given by $\{X\in M_4(\R)\mid {}^tXJ_4+J_4X=0_4\}$. 
The Cartan involution of $\g$, denoted also by $\theta$, is defined by $\theta(X)=-{}^tX$ for $X\in\g$. Then $\g$ has the eigen-space decomposition $\g=\fk+\p$ with 
\begin{align*}
\fk&=\{X\in\g\in\mid \theta(X)=X\}=\left\{\left.
\begin{pmatrix}
A & B\\
-B & A
\end{pmatrix}~\right|~A,~B\in M_2(\R),~{}^tA=-A,~{}^tB=B\right\},\\
\p&=\{X\in\g\in\mid \theta(X)=-X\}=\left\{\left.
\begin{pmatrix}
A & B\\
B & -A
\end{pmatrix}~\right|~A,~B\in M_2(\R),~{}^tA=A,~{}^tB=B\right\}.
\end{align*}
The former is nothing but the Lie algebra of $K$. 

We consider the root space decomposition of $\g_{\C}$ with respect to the complexification $\frak{t}_{\C}$ of the compact Cartan subalgebra $\frak{t}=\R T_1\oplus\R T_2$~(in $\fk$), where 
\[
T_1:=
\begin{pmatrix}
0 & 0 & 1 & 0\\
0 & 0 & 0 & 0\\
-1 & 0 & 0 & 0\\
0 & 0 & 0 & 0 
\end{pmatrix},\quad 
T_2:=
\begin{pmatrix}
0 & 0 & 0 & 0\\
0 & 0 & 0 & 1\\
0 & 0 & 0 & 0\\
0 & -1 & 0 & 0
\end{pmatrix}.
\]
The dual space ${\frak t}_{\C}^*$ of ${\frak t}_{\C}$ has a basis $\{\beta_1,\beta_2\}$ given by
\[
\beta_i(T_j)=\sqrt{-1}\delta_{ij}.
\]
We denote $\beta\in{\frak t}_{\C}^*$ by $(a,b)$ if $\beta=a\beta_1+b\beta_2$. 
With this notation the set of roots for the root space decomposition $(\g_{\C},{\frak t}_{\C})$ is given by $\Delta:=\{\pm(2,0),~\pm(0,2),~\pm(1,1),~\pm(1,-1)\}$. This set has the standard choice of positive roots given by $\Delta^+:=\{(2,0),~(0,2),~(1,1),~(1,-1)\}$. 
The roots $\{\pm(1,-1)\}$ forms the set of compact roots, whose roots vectors are in $\fk_{\C}$. 
Each root in $\{\pm(2,0),~\pm(0,2),~\pm(1,1)\}$ is called a non-compact root, whose root vector is in the complexification $\p_{\C}$ of $\p$.
\section{Representations of real groups}\label{Rep-real}
This section collects the representations of the real groups $\widetilde{SL}_2(\R),~N_J,~G_J$ and $G=Sp(2,\R)$ necessary for us. 
Sections 3.1 and 3.2 take up representations of $\widetilde{SL}_2(\R),~N_J$ and $G_J$. The rest of the sections, Sections 3.3--3.6, is devoted to an explanation of the admissible representations of $G$ in our concern. More precisely, after reviewing some fundamental facts on representations of the maximal compact subgroup $K$ in Section 3.3, we explain discrete series representations,~$P_J$-principal series representations and principal series representation~(induced from the minimal parabolic subgroup) in Sections 3.4--3.6.
 
From now on we often use the notation $\e(x):=\exp(2\pi\sqrt{-1}(x))$ for $x\in\R$. 
\subsection{Irreducible unitary representations of $\widetilde{SL}_2(\R)$}\label{RepSL2}
From \cite[Prposition 2.3]{Hi-1}~(see also \cite[Lemmas 4.1,~4.2]{Ge}) we review the classification of irreducible unitary representations of $\widetilde{SL}_2(\R)$.
\begin{lem}\label{GenuineRep}
Irreducible unitary representations of $\widetilde{SL_2}(\R)$ are classified as follows:\\
(1)~(unitary principal series)~$\cP_s^{\tau}$ for $(s,\tau)\in\sqrt{-1}\R\times\{0,1,\pm\frac{1}{2}\}\setminus\{(0,1)\}$,\\
(2)~(complementary series)~$\cC_s^{\tau}$ for $0<s<1$ and $\tau\in\{0,1\}$ and for $0<s<\frac{1}{2}$ and $\tau\in\{\pm\frac{1}{2}\}$,\\
(3)~((limit of) discrete series representations)~$\cD^{\pm}_n$ for $n\in\frac{1}{2}\Z_{\ge 2}$,\\
(4)~(quotient representation)~$\cD^+_{\frac{1}{2}},~\cD^-_{\frac{1}{2}}$,\\
(5)~The trivial representation $1_{\widetilde{SL}_2(\R)}$.\\
Here only the representations $\cP_s^{\tau},~\cC_s^{\tau}$ with $\tau=0,~1$,~$\cD^{\pm}_n$ with $n\in\Z_{\ge 1}$ and $1_{\widetilde{SL}_2(\R)}$ give those of $SL_2(\R)$.
\end{lem}
Let $\g_1$ be the Lie algebra of $\widetilde{SL}_2(\R)$, which coincides with the Lie algebra of $SL_2(\R)$. 
The comlexification $\g_{1,\C}$ of the Lie algebra $\g_1$ has a basis $\{U,V_{\pm}\}$ as follows:
\[
U:=-\sqrt{-1}
\begin{pmatrix}
0 & 1\\
-1 & 0
\end{pmatrix},\quad V_{\pm}:=\frac{1}{2}
\begin{pmatrix}
1 & \pm\sqrt{-1}\\
\pm\sqrt{-1} & -1
\end{pmatrix}.
\] 
Following \cite{Hi-1} we introduce the index set $L$ for an irreducible unitary representation $\pi_1$ as follows:
\begin{equation}\label{index-genuine}
L=L_{\pi_1}:=
\begin{cases}
\tau+2\Z&(\pi_1=\cP_s^{\tau},~\cC_s^{\tau}),\\
\pm n\pm2\Z_{\ge 0}&(\pi_1=\cD_n^{\pm}).
\end{cases}
\end{equation}
The representation space $\cW_{\pi_1}$ of an irreducible unitary representation $\pi_1$ has a basis $\{w_l\}_{l\in L}$ satisfying the following formula for the infinitesimal action of $\pi$:
\begin{equation}\label{Infinitesimal-MaassOp}
\pi_1(U)w_l=lw_l,\quad\pi_1(V^{\pm})w_l=\frac{1}{2}(z_1+1\pm l)w_{l\pm2},
\end{equation}
where we understand that $w_l=0$ for $l\not\in L$ and that $z_1=s$~(respectively~$n-1$) for $\pi_1=\cP_s^{\tau}$~(respectively~$\pi_1=\cD_n^{\pm}$). Recall that $\widetilde{SO_2}(\R)$ is a maximal compact subgroup of $\widetilde{SL_2}(\R)$, given by the inverse image of $SO_2(\R)$ by the covering map $\widetilde{SL_2}(\R)$ to $SL_2(\R)$~(cf.~Section \ref{Groups} (1)). For each irreducible unitary representation $\pi_1$ of $\widetilde{SL_2}(\R)$ the irreducible representation of $\widetilde{SO_2}(\R)$ indexed by  $l\in L_{\pi_1}$ with the least absolute value is called the minimal $\widetilde{SO_2}(\R)$-type of $\pi_1$. We denote this by $\tau_{\rm min}$.  
\subsection{Irreducible unitary representations of $N_J$ and $G_J$}\label{UnitaryRepJacobi}
The group $N_J$ is a two step nilpotent Lie group, known as a Heisenberg group. The central characters 
\[
\chi_m:Z_J\ni n(0,u_1,0)\mapsto\e(mu_1)\in\C^{\times}
\]
with $m\in\R$ are essential to know the classification of irreducible unitary representations of $N_J$, where recall that $Z_J$ denotes the center of $N_J$~(cf.~Section \ref{Groups}). In fact, irreducible unitary representations of $N_J$ with the trivial central character are unitary characters of $N_J$, and irreducible unitary representations of $N_J$ with non-trivial central characters are infinite dimensional representations. The latter is classified in terms of the central characters, up to unitary equivalence. This fact is verified easily, e.g. by means of the orbit method of nilpotent Lie groups~(cf.~Corwin-Greenleaf \cite{Co-G-2}). 

Let us review the classification of irreducible unitary representations of $N_J$ in detail. We can consider the trivial extension of the character $\chi_m$ to a character of the polarization subgroup $M:=\{n(0,u_1,u_2)\mid u_1,~u_2\in\R\}$ of $N_J$, which is also denoted by $\chi_m$. The group $M$ is the image of the exponential map of the polarization subalgebra for the Lie algebra $\n_J$ of $N_J$~(for a definition see \cite[pp.27,~28]{Co-G-2}) with respect to a non-zero linear form of $\n_J$ whose restriction to the center is non-trivial. We then  introduce a unitary representation $(\nu_m,\cU_m)$ of $N_J$ by the $L^2$-induction 
\[
L^2\text{-}{\rm Ind}_M^{N_J}(\chi_m)
\]
from the character $\chi_m$ of $M$ with the representation space $\cU_m$. Here $N_J$ acts on $\cU_m$ by the right translation, which we denote by $\nu_m$. Since we have the identification $M\backslash N_J\simeq\R$, we view $\cU_m$ as $L^2(\R)$. 

The space $\cU_m$ has a basis $\{u_j^m\}_{j\in J}$ with $J:=\operatorname{sign}(m)(\frac{1}{2}+\Z_{\ge 0})$ satisfying \cite[(2.5)]{Hi-1}~(see also \cite[(3.2),~(3.3)]{Be-Sc}). For instance we can choose Hermite functions $\{h_j\}_{j\in\Z_{\ge 0}}$ as such a basis~(cf.~Section \ref{Explicit-FJ}). To be precise we understand that, when $m>0$~(respectively~$m<0$), $u_j=h_{j-\frac{1}{2}}$~(respectively $u_j=h_{-(j-\frac{1}{2})}$) for $j\in J$.
\begin{prop}\label{UnitaryDual-N_J}
(1)~Up to unitary equivalence, irreducible unitary representations of $N_J$ are exhausted by
\begin{enumerate}
\item unitary characters $\eta_{m_0,m_2}:N_J\ni n(u_0,u_1,u_2)\mapsto \e(m_0 u_0+m_2 u_2)\in\C^{\times}$ with $(m_0,m_2)\in\R^2$,
\item Infinite dimensional representations $(\nu_m,\cU_m)$ with $m\in\R\setminus\{0\}$.
\end{enumerate}
(2)~For $m,~m'\in\R\setminus\{0\}$, $\nu_m\simeq\nu_{m'}$ if and only if $m=m'$.
\end{prop}

We next describe irreducible unitary representations of $G_J$. We need to recall that, for a non-zero $m\in\R$,  $(\nu_m,\cU_m)$ is extended to the unitary representation $(\tilde{\nu}_m,\cU_m)$ of $N_J\times\widetilde{SL}_2(\R)$ by 
$\tilde{\nu}_m(n\cdot\tilde{g})=\omega_m(\tilde{g})\nu_m(n)$ for $(n,\tilde{g})\in N_J\times \widetilde{SL}_2(\R)$, where $\omega_m$ denotes the Weil representation of $\widetilde{SL}_2(\R)$. For this representation we remark that $G_J$ has $Z_J$ as the center and it is possible to consider the notion of the central characters also for irreducible representations of $G_J$.
\begin{prop}\label{UnitaryDual-G_J}
(1)~Let $(\tilde{\nu}_m,\cU_m)$ be as above. For an irreducible genuine representation $(\pi_1,\cW_{\pi_1})$ of $\widetilde{SL}_2(\R)$~(i.e. representation of $\widetilde{SL}_2(\R)$ not factoring through $SL_2(\R)$) we put the representation $\rho_{m,\pi_1}$ of $G_J$ by 
\[
\rho_{m,\pi_1}(n\cdot g):=\pi_1(g)\otimes\tilde{\nu}_m(n\cdot g)
\]
for $(n,g)\in N_J\times SL_2(\R)$, where note that $\rho_{m,\pi_1}|_{\widetilde{SL}_2(\R)}$ factors through $SL_2(\R)$. Then $(\rho_{m,\pi_1},\cW_{\pi_1}\otimes\cU_m)$ is an irreducible unitary representation of $G_J$. All irreducible unitary representations of $G_J$ with non-trivial central characters are obtained in this manner.\\
(2)~Let $(\rho,\cF_{\rho})$ be an irreducible unitary representation of $G_J$ with the trivial central character. 
Then $\rho$ is unitarily equivalent to one of the following representations:
\begin{enumerate}
\item $\rho_{\pi_1,0}(n\cdot g)=\pi_1(g)\tilde{\nu}_0(n\cdot g)$ for $(n,g)\in N_J\times SL_2(\R)$, where $\pi_1$~(respectively~$\tilde{\nu}_0$) is an irreducible unitary representation of $SL_2(\R)$~(respectively~the trivial representation of $G_J$), 
\item representations of $G_J$ induced from unitary characters of $N_0$ given by 
\[
N_0\ni n(u_0,u_1,u_2,u_3)\mapsto\e(u_0+ru_3)
\]
for $r\in\R$~(for $N_0$ see Section \ref{Groups}).
\end{enumerate}
\end{prop}
\begin{proof}
For this proposition see \cite[Propositions 2.5,~2.6]{Hi-1}. We also cite Berndt-Schmidt \cite[Theorem 2.6.1]{Be-Sc} and Satake \cite[Appendix 1,~Proposition 2]{Sa} for the assertion (1). 
For this we note that irreducible genuine representations $\sigma$ of $\widetilde{SL}_2(\R)$ has the same multiplier with $\tilde{\nu}_m|_{\widetilde{SL}_2(\R)}$, which is due to the uniqueness of the non-trivial element of $H^2(SL_2(\R),\{\pm 1\})$~(cf.~Section \ref{Groups}). Thus $\rho_{m,\pi_1}|_{\widetilde{SL}_2(\R)}$ factors through $SL_2(\R)$ and $\rho_{m,\pi_1}$ is thus a well-defined representation of $G_J$. As is remarked just before \cite[Proposition 2.6]{Hi-1} the second assertion is deduced from Mackey's method for representations of semi-direct product groups, so called Mackey machine.
\end{proof}
\subsection{Representations of the maximal compact subgroup $K$}\label{RepMaxCpt}
It is a fundamental fact that irreducible finite dimensional representations of a compact linear connected  real  reductive group are parametrized by dominant weights. For the maximal compact subgroup $K$ of $G$ the set of equivalence classes of irreducible finite dimensional representations of $K$ are in bijection with the set of the dominant weights $\{(\Lambda_1,\Lambda_2)\in\Z^2\mid\Lambda_1\ge\Lambda_2\}$, where recall that $(\Lambda_1,\Lambda_2)$ denotes the root $\Lambda_1\beta_1+\Lambda_2\beta_2$~(cf.~Section \ref{Lie-gp-alg}). The dominance respects the compact positive root $(1,-1)$. This is noting but the fact that the set of equivalence classes of irreducible finite dimensional representations of $U(2)\simeq K$ is parametrized by this set of dominant weights. For each dominant weight $\Lambda=(\Lambda_1,\Lambda_2)$, let $\tl$ be the irreducible representation of $K$ parametrized by $\Lambda$, which is the pullback of the irreducible representation $\det^{2\Lambda_2}\Sym^{\Lambda_1-\Lambda_2}$ of $U(2)$ via the isomorphism $K\simeq U(2)$. Here $\Sym^{\Lambda_1-\Lambda_2}$ denotes the $\Lambda_1-\Lambda_2$-th symmetric tensor representation of the 2-dimensional standard representation of $U(2)$. We denote the infinitesimal action of $\tl$ also by $\tl$.

We now introduce a basis $\{Z,H,Y,Y'\}$ of $\fk_{\C}$ as follows:
\[
Z:=
\begin{pmatrix}
0_2 & -\sqrt{-1}1_2\\
\sqrt{-1}1_2 & 0_2
\end{pmatrix},~H:=\frac{1}{\sqrt{-1}}(T_1-T_2),~Y:=
\begin{pmatrix}
J_2 & 0_2\\
0_2 & J_2
\end{pmatrix},~Y':=
\begin{pmatrix}
0_2 & J'_2\\
-J'_2 & 0_2
\end{pmatrix}
\]
with $J_2:=
\begin{pmatrix}
0 & 1\\
-1 & 0
\end{pmatrix}$ and $J'_2=
\begin{pmatrix}
0 & 1\\
1 & 0
\end{pmatrix}$. 
Then $\{H,X:=\frac{1}{2}(Y-\sqrt{-1}Y'),\bar{X}:=\frac{1}{2}(-Y-\sqrt{-1}Y')\}$ forms an ${\frak sl}_2$-triple. In fact, $[H,X]=2X,~[H,\bar{X}]=-2\bar{X},~[X,\bar{X}]=H$ holds. The element $Z$ belongs to the center. The representation space $V_{\Lambda}$ of $\tl$ has the dimension $d_{\Lambda}+1$ with $d_{\Lambda}=\Lambda_1-\Lambda_2$ and has a basis $\{v_k\}_{0\le k\le d_{\Lambda}}$.satisfying
\begin{align*}
\tl(Z)v_k&=(\Lambda_1+\Lambda_2)v_k,\quad\tl(X)v_k=(k+1)v_{k+1},\\
\tl(H)v_k&=(2k-d_{\Lambda})v_k,\quad\tl(\bar{X})v_k=(d_{\Lambda}+1-k)v_{k-1}
\end{align*}
for $0\le k\le d_{\Lambda}$. Later we will often need the contragredient representation $(\tl^*,V_{\Lambda}^*)$ with the representation space $V_{\Lambda}^*$. Its highest weight is $(-\Lambda_2,-\Lambda_1)$. 
The space $V_{\Lambda}^*$ has a basis $\{v_k^*\}_{0\le k\le d_{\Lambda}}$ satisfying
\begin{align*}
\tl(Z)v_k^*&=(-\Lambda_1-\Lambda_2)v_k^*,\quad\tl(X)v_k^*=(k+1)v_{k+1}^*,\\
\tl(H)v_k^*&=(2k-d_{\Lambda})v_k^*,\quad\tl(\bar{X})v_k=(d_{\Lambda}+1-k)v_{k-1}^*
\end{align*}
for $0\le k\le d_{\Lambda}$. This is nothing but what is obtained by replacing $(\Lambda_1,\Lambda_2)$ with $(-\Lambda_2,-\Lambda_1)$. Note that $d_{\Lambda}$ remains the same under such replacement of the highest weights. 
\subsection{Discrete series representations of $G$}\label{DS-rep}
We provide Harish Chandra's parametrization of discrete series representations for $G$, namely irreducible unitary representations of $G$ whose matrix coefficients belong to $L^2(G)$. Our discussion is based on \cite[Theorem 16]{Hc-1} and \cite[Theorem 9.20,~Theorem 12.21]{Kn} with the help of \cite{Ko} and \cite{V}. 

To parametrize such representations we need regular dominant analytically integral weights, namely regular dominant weights coming from the derivatives of unitary characters of the compact Cartan subgroup $\exp({\frak t})$, where see Section \ref{Lie-gp-alg} for the notation ${\frak t}$. Now note that unitary characters of $\exp({\frak t})$ are parametrized by 
\[
\{a\beta_1+b\beta_2\mid (a,b)\in\Z^2\}\simeq\Z^2.
\]
We then see that the set of the regular analytically integral weights is in bijection with \[
\Xi':=\{\lambda=(\lambda_1,\lambda_2)\in\Z^2\mid \langle\lambda,\beta\rangle\not=0~\forall\beta\in\Delta^+\}.
\] 
Here $\langle*,*\rangle$ denotes the inner product induced by the Killing form of ${\frak g}$, which can be regarded as the standard inner product of the two dimensional Euclidean space $\R^2$ in terms of the bijection $\{a\beta_1+b\beta_2\mid (a,b)\in\R^2\}\simeq\R^2$. 

By Harish Chandra's parametrization we mean the classification of discrete series representations in terms of the infinitesimal equivalence~(which means the unitary equivalence of discrete series). Let $\pi_{\lambda}$ be the discrete series representation of $G$ parametrized by $\lambda\in\Xi'$. It is known that $\pi_{\lambda}$ and $\pi_{\lambda'}$ are infinitesimally equivalent if and only if $\lambda$ and $\lambda'$ are conjugate by the Weyl group of $\fk$. 
As a result we see that the equivalence classes of discrete series representations of $G$ are in bijection with 
\[
\Xi:=\{\lambda\in\Xi'\mid \langle\lambda,(1,-1)\rangle>0\}=\{(\lambda_1,\lambda_2)\in\Xi'\mid \lambda_1>\lambda_2\},
\]
each of which is called a Harish Chandra parameter~(cf.~\cite[Terminology after Theorem 9.20]{Kn}). 

Now we introduce the following four sets of positive roots system including the set $\Delta_{c}^+:=\{(1,-1)\}$ of the compact positive root:
\begin{align*}
\Delta_{I}^+&:=\{(2,0),~(0,2),~(1,1),~(1,-1)\},\\
\Delta_{II}^+&:=\{((2,0),~(0,-2),~(1,1),~(1,-1)\},\\
\Delta_{III}^+&:=\{(2,0),~(0,-2),~(-1,-1),~(1,-1)\},\\
\Delta_{IV}^+&:=\{(-2,0),~(0,-2),~(-1,-1),~(1,-1)\}.
\end{align*}
Corresponding to the above four sets we introduce the four sets of regular dominant weights as follows:
\begin{align*}
\Xi_{I}&:=\{(\lambda_1,\lambda_2)\in\Z_{>0}^2\mid\lambda_1>\lambda_2\},\\
\Xi_{II}&:=\{(\lambda_1,\lambda_2)\in\Z_{>0}\times\Z_{<0}\mid\lambda_1>-\lambda_2\},\\
\Xi_{III}&:=\{(\lambda_1,\lambda_2)\in\Z_{>0}\times\Z_{<0}\mid\lambda_1<-\lambda_2\},\\
\Xi_{IV}&:=\{(\lambda_1,\lambda_2)\in\Z_{<0}^2\mid\lambda_1>\lambda_2\}.
\end{align*}
Discrete series representations of $G$ parametrized by $\Xi_{I}$~(respectively~$\Xi_{IV}$) are called holomorphic discrete series representations~(respectively~anti-holomorphic discrete series representations). Discrete series representations of $G$ parametrized by $\Xi_{II}\cup\Xi_{III}$ are called large in the sense of Vogan \cite[Section 6]{V}. 
The large discrete series representations are characterized by the maximality of their Gelfand Kirillov dimensions~'(by having the minimal primitive ideals) among the discrete series representations. They are known to admit Whittaker models~(cf.~\cite[Theorem 6.8.1]{Ko}), and are therefore also called generic discrete series representations. In fact, we will see that they admits  rapidly decreasing Whittaker models~(cf.~Theorem \ref{LargeDSWhittaker}).

For $(\lambda_1,\lambda_2)\in\Xi$ the discrete series representation parametrized by $(-\lambda_2,-\lambda_1)$ is contragredient to that parametrized by $(\lambda_1,\lambda_2)$. We then see that the discrete series representations parametrized by $\Xi_{II}$ and $\Xi_{III}$ are in bijection with  contragredeint representations of those parametrized by $\Xi_{III}$ and $\Xi_{II}$ respectively. 

For each $\lambda\in\Xi$ we recall that 
\[
\Lambda:=\lambda+\rho_n-\rho_c
\]
with the half sum $\rho_n$~(respectively~$\rho_c$) of non-compact positive roots~(respectively~compact positive roots) is called the Blattner parameter~(cf.~\cite[Terminology after Theorem 9.20]{Kn}), which provides the highest weight of the minimal $K$-type~(cf.~\cite[p626]{Kn}) of the discrete series representation $\pi_{\lambda}$. 
For $\lambda=(\lambda_1,\lambda_2)\in\Xi_{I}$ or $\Xi_{IV}$ we have $\Lambda=(\lambda_1+1,\lambda_2+2)$ or $(\lambda_1-2,\lambda_2-1)$ respectively. 
When $\lambda=(\lambda_1,\lambda_2)$ is in $\Xi_{II}$ or $\Xi_{III}$, we have $\Lambda=(\lambda_1+1,\lambda_2)$ or $(\lambda_1,\lambda_2-1)$ respectively.
\subsection{$P_J$-principal series representations of $G$}\label{P_JPS-rep}
We introduce the generalized principal series representations induced from the Jacobi parabolic subgroup $P_J$, which we call the $P_J$-principal series representations. 

Let $\sigma:=(D,\epsilon)$ be a representation of $M_J\simeq SL_2(\R)\times\{\pm 1_2\}$ with the sign character $\epsilon$ defined by $\epsilon(\pm 1_{2})=\pm 1$ and a (limit of) discrete series representation $D=\cD_{n}^{\pm}$ with $n\in\Z_{\ge 1}$~(cf.~Lemma \ref{GenuineRep}), where note that $\cD^{\pm}_n$ with an integer $n$ is not genuine. Given $z\in\C$ we define the quasi character 
\[
\nu_{z}:A_J\ni a_J=\diag(a_1,1,a_1^{-1},1)\mapsto a_1^z.
\] 
We then introduce the $P_J$-principal series representation $\Ind_{P_J}^G(1_{N_J}\otimes\nu_z\otimes\sigma)$ by the standard manner of the normalized parabolic induction.

By the Frobenius reciprocity of compact groups we can study the branching rule for the restriction of a $P_J$-principal series representation to $K$. Among the $K$-types $\tl$ of $\Ind_{P_J}^G(1_{N_J}\otimes\nu_z\otimes\sigma)$ the followings occur with multiplicity one~(cf.~\cite[Proposition 2.1]{M-Od-2},~\cite{M-Od-3}):
\begin{enumerate}
\item $\Lambda=(l,l)~(l\ge n,~l\equiv n\mod 2)$ and $\Lambda=(n,l)~(l\le n,~l\equiv n\mod 2)$ for $\epsilon(-1_2)=(-1)^n$ and $D=\cD_n^+$.
\item $\Lambda=(l,l-1)~(l\ge n)$ and $\Lambda=(n,l-1)~(l\le n,~l\equiv n\mod 2)$ for $\epsilon(-1_2)=-(-1)^n$ and $D=\cD_n^+$.
\item $\Lambda=(l,l)~(l\le-n,~l\equiv n\mod 2$ and $\lambda=(l,-n)~(l\ge -n,~l\equiv n\mod 2)$ for $\epsilon(-1_2)=(-1)^n$ and $D=\cD_n^-$.
\item $\Lambda=(l+1,l)~(l\le -n)$ and $\lambda=(l+1,-n)~(l\ge -n,~l\equiv n\mod 2)$ for $\epsilon(-1_2)=-(-1)^n$ and $\cD_n^-$.
\end{enumerate}
We call the $K$-type $\tl$ with $\Lambda=(n,n)$~(respectively~$(n,n-1),~(-n,-n)$ and $(-n+1,-n)$) the corner $K$-type of $\Ind_{P_J}^G(1_{N_J}\otimes\nu_z\otimes\sigma)$ for the case 1~(respectively~2,~3 and 4). 
For the cases 1 and 3~(respectively~the cases 2 and 4) the dimension of the corner $K$-type is one~(respectively~two). Following \cite{M-Od-2} we call the $P_J$-principal series enumerated as $1$ and $3$~(respectively~$2$ and $4$) even~(respectively~odd). 
\subsection{Principal series representations of $G$~(induced from the minimal parabolic subgroup)}\label{PS-rep}
Recall that $P_0$ denotes the minimal parabolic subgroup with the Langlands decomposition $P_0=N_0A_0M_0$ with 
\[
A_0:=\{a_0=\diag(a_1,a_2,a_1^{-1},a_2^{-1})\mid a_1,~a_2\in\R_{>0}\},\quad M_0:=\{\diag(\epsilon_1,\epsilon_2,\epsilon_1,\epsilon_2)\mid \epsilon_1,~\epsilon_2\in\{\pm 1\}\}
\]
(cf.~Section \ref{Groups} (1)).
Representations $\sigma$ of $M_0$ are the sign characters determined by
\[
\sigma_1:=\sigma(\diag(-1,1-1,1))\in\{\pm 1\},\quad\sigma_2:=\sigma(\diag(1,-1,1,-1))\in\{\pm 1\}.
\]
We introduce quasi characters of $A_0$ by
\[
\nu_z:A_0\ni a_0=\diag(a_1,a_2,a_1^{-1},a_2^{-1})\mapsto a_1^{z_1}a_2^{z_2}
\]
for $z=(z_1,z_2)\in\C^2$. With $\sigma$ and $\nu_z$ above we define the principal series representation by the normalized parabolic induction $\Ind_{P_0}^G(1_{N_0}\otimes\nu_z\otimes\sigma)$. 

From \cite[Proposition 3.2]{M-Od-1} we have the list of the $K$-types of $\Ind_{P_0}^G(1_{N_0}\otimes\nu_z\otimes\sigma)$ with multiplicity one as follows:
\begin{enumerate}
\item $\Lambda=(l,l)$ with $l\in\Z$ such that 
$(-1)^l=\sigma_1=\sigma_2$ 
for $(\sigma_1,\sigma_2)=\pm(1,1)$.
\item $\Lambda=(l+1,l)$ with $l\in\Z$ for $(\sigma_1,\sigma_2)=\pm(1,-1)$. 
\end{enumerate}
We call the principal series representations even~(respectively~odd) for the first case~(respectively~the second case). 
We note that the minimal $K$-type(s) of the principal series representation is/are given by $\tau_{(0,0)}$~(respectively~$\{\tau_{(1,1)},~\tau_{(-1,-1)}\}$, and $\{\tau_{(1,0)},~\tau_{(0,-1)}\}$) if $(\sigma_1,\sigma_2)=(1,1)$~(respectively~$(-1,-1)$ and $(\pm 1,\mp 1)$). For the definition of the minimal $K$-type we cite \cite[p626]{Kn}, as we also do so in Section \ref{DS-rep}. 
\section{A review on Whittaker functions and Fourier-Jacobi type spherical functions for $Sp(2,\R)$}\label{ReviewSpherical}
\subsection{A review on Whittaker functions}\label{ReviewWhittaker}
Let $\pi$ be an admissible representation  of $G=Sp(2,\R)$ with a multiplicity one $K$-type $\tau$ and let $\psi$ be a unitary character of the maximal unipotent subgroup $N_0$ of $G$. The Whittaker functions on $G$ are defined as the restriction of the intertwining operators in 
\[
\Hom_{(\g,K)}(\pi, C^{\infty}_{\psi}(N_0\backslash G))
\]
to the $K$-type $\tau$, where 
\[
C^{\infty}_{\psi}(N_0\backslash G):=\{\phi\in C^{\infty}(G)\mid \phi(ng)=\eta(n)\phi(g)\quad\forall(n,g)\in N_0\times G\}.
\]
The image of the restriction map is contained in $C_{\psi,\tau^*}^{\infty}(N_0\backslash G/K):=$\\
\[
\{\text{$C^{\infty}$-function}~W:G\rightarrow V^*~\mid W(ngk)=\psi(n)\tau^*(k)^{-1}W(g)~\forall(n,g,k)\in N_0\times G\times K\}, 
\]
where recall that $(\tau^*,V^*)$ is the notation for the contragredient representation of $(\tau,V)$.  
Now let us note that we have to impose the multiplicity one property on the $K$-type $\tau$ of $\pi$. 
Then the $K$-embedding $\tau\hookrightarrow\pi$ is unique and the notion of the Whittaker functions is well-defined. 
By $W_{\psi,\pi}(\tau^*)$ we denote the image of the restriction map. We call an element
in $W_{\psi,\pi}(\tau^*)$ a Whittaker function for $\pi$ and $K$-type $\tau^*$. 
We will need 
\begin{align*}
W_{\psi,\pi}(\tau^*)^0&:=\{w\in W_{\psi,\pi}(\tau^*)\mid\text{$w$ is rapidly decreasing}\},\\
\cS_{\psi}(N_0\backslash G)&:=\{\phi\in C^{\infty}_{\psi}(N_0\backslash G)\mid\text{$\phi$ is rapidly decreasing}\}
\end{align*}
since we are motivated by the Fourier expansion of cusp forms, which are rapidly decreasing~(cf.~\cite[Chapter 1,~Section 4]{Hc-2},~\cite{Gd}).

Let $\dim V^*=d+1$ and $\{v_k^*\}_{k=0}^d$ be a basis of $V^*$ consisting of weight vectors with highest weight vector $v_d^*$ as in Section \ref{RepMaxCpt}. When the highest weight of $\tau$ is $(\Lambda_1,\Lambda_2)$, we have $d=d_{\Lambda}=\Lambda_1-\Lambda_2$~(cf.~Section \ref{RepMaxCpt}). We can express each Whittaker function $W\in W_{\psi,\pi}(\tau^*)$ as 
\[
W(g)=\sum_{k=0}^{d}c_k(g)v_k^*
\]
with coefficient functions $c_k(g)$. 
Now recall that $G$ admits the Iwasawa decomposition of $G=N_0A_0K$~(cf.~Section \ref{Groups}), where $A_0:=
\left\{
a_0=\diag(a_1,a_2,a_1^{-1},a_2^{-1})\mid a_1,~a_2\in\R_{>0}\right\}$. We then see that the Whittaker function $W$ is determined by the restriction to $A_0$.
\subsection*{The case of large discrete series and $P_J$-principal series}
These cases are settled by \cite{Od} and \cite{M-Od-2}~(see also \cite{M-Od-3}). We first review the following result by Oda \cite{Od} for the case of large discrete series representations.  For this note that $\pi_{\lambda}$ is replaced by $\pi_{\lambda}^*$ with $\lambda\in\Xi_{II}$ in \cite[Theorem 9.1]{Od}.
\begin{thm}[Oda]\label{LargeDSWhittaker}
Assume that $\psi$ is a non-degenerate unitary character of $N_0$, namely, $\psi(n(u_0,u_1,u_2,u_3))=\exp(2\pi\sqrt{-1}(m_0u_0+m_3u_3))\in\C^{\times}$ with $m_0m_3\not=0$.\\
(1)~(cf.~\cite[Theorem 9.1]{Od})~Let $\pi_{\lambda}$ be a discrete series representation with Harish Chandra parameter $\lambda\in\Xi_{III}$ and $\tl$ be the minimal $K$-type of $\pi_{\lambda}$ with highest weight $\Lambda=(\Lambda_1,\Lambda_2)$. We have
\[
\dim W_{\psi,\pi}(\tau^*)^0=
\begin{cases}
1&(m_3<0),\\
0&(m_3>0).
\end{cases}
\]
The coefficient function $c_{d_{\Lambda}}(a_0)$ of the restriction $W|_{A_0}$ of the Whittaker function $W$ to $A_0$ is given explicitly by $a_1^{-\Lambda_2+1-d_{\Lambda}}a_2^{-\Lambda_2}\left(\displaystyle\frac{a_1}{a_2}\right)^{d_{\Lambda}}e^{2\pi m_3a_2^2}h_{d_{\Lambda}}(a_0)$ with
\[
h_{d_{\Lambda}}(a_0)=\displaystyle\int_0^{\infty}t^{-\Lambda_1-3/2}W_{0,\Lambda_1}(t)\exp\left(\dfrac{t^2}{64\pi m_3a_2^2}+\dfrac{64\pi^3m_0^2m_3a_1^2}{t^2}\right)\dfrac{dt}{t},
\]
up to constant multiples. 
Here, for $y>0$, $W_{\kappa,\mu}(y)$ 
denotes a unique rapidly decreasing solution of the confluent hypergeometric equation by Whittaker~(cf.~\cite[Chapter 16]{W-W}):
\[
\dfrac{d^2}{dy^2}W+\left\{-\dfrac{1}{4}+\dfrac{\kappa}{y}+\dfrac{1/4-\mu^2}{y^2}\right\}W=0.
\]
The other coefficients $c_k$ for $0\le k\le d_{\Lambda}-1$ are obtained from $c_{d_{\Lambda}}$ by the recurrence relation $({\rm E})_k$ in \cite[Section 8]{Od}.\\
(2)~For a discrete series representation $\pi_{\lambda}$ with $\lambda=(\lambda_1,\lambda_2)\in\Xi_{II}$ we have
\[
\dim W_{\psi,\pi}(\tau^*)^0=
\begin{cases}
1&(m_3>0),\\
0&(m_3<0).
\end{cases}
\]
When $\dim W_{\psi,\pi}(\tau^*)^0\not=0$ the Whittaker function $W^*$ for this case is uniquely given by
\[
W^*(g)=W(\delta g\delta^{-1}\xi)\quad(g\in G)
\]
with 
\[
\delta=
\begin{pmatrix}
-I_2 & 0_2\\
0_2 & I_2
\end{pmatrix},\quad \xi=
\begin{pmatrix}
J'_2 & 0_2\\
0_2 & J'_2
\end{pmatrix}
\]
(see Section \ref{RepMaxCpt} for the notation $J'_2=:
\begin{pmatrix}
0 & 1\\
1 & 0
\end{pmatrix}$), up to constant multiples, 
where $W$ is the Whittaker function for the discrete series representation with Harish Chandra parameter $(-\lambda_2,-\lambda_1)\in\Xi_{III}$.
\end{thm}
\begin{proof}
The idea to obtain the explicit formula in (1) is explained in \cite[Section 8.1]{M-Od-2}. From the calculation of the explicit formula  we see that the condition $m_3<0$ is necessary for the Whittaker function to be rapidly decreasing, for which note that the characterizing differential equation for the case of large discrete series is quite similar to that for the case of $P_J$-principal series as in \cite{M-Od-2}. 

We first check the rapidly decreasing property of the explicit Whittaker functions for $m_3<0$. 
Though this is already pointed out in the proof of \cite[Theorem 9.1]{Od} we give a more detailed explanation. 
In the integrand of the integral expression of $h_{d_{\Lambda}}(a_0)$, $W_{0,\Lambda_1}(t)$ is rapidly decreasing and thus bounded. 
We therefore see  that $|h_{d_{\Lambda}}(a_0)|$ is bounded by a constant multiple of 
\[
\displaystyle\int_0^{\infty}t^{-\Lambda_1-3/2}\exp\left(\dfrac{t^2}{64\pi m_3a_2^2}+\dfrac{64\pi^3m_0^2m_3a_1^2}{t^2}\right)\dfrac{dt}{t}.
\]
It suffices to verify the rapidly decreasing property of this integral. 
By the change of variables $t\rightarrow x:=\left(\displaystyle\frac{t}{8\pi\sqrt{|m_0|(-m_3)a_1a_2}}\right)^2$ this can be understood by the integral expression of the $K$-Bessel function
\[
K_{\nu}(y)=\displaystyle\int_0^{\infty}\exp\left(-y\left(x+\dfrac{1}{x}\right)\right)x^{\nu}\dfrac{dx}{x}.
\]
It is well known that this is rapidly decreasing. This implies the rapidly decreasing property of $c_{d_{\Lambda}}(a_0)$. 
As for the other $c_k$ with $0\le k\le d_{\Lambda}-1$ the rapidly decreasing property also holds since the recurrence relation $({\rm E})_k$ in \cite[Section 8]{Od} preserves such property. 

We next explain how to get $\dim W_{\psi,\pi}(\tau^*)^0$. In \cite[Theorem 9.1]{Od} the result on $\dim W_{\psi,\pi}(\tau^*)^0$ is formulated as
\[
\dim\Hom_{(g,K)}(\pi_{\lambda},\cS_{\psi}(N_0\backslash G))=
\begin{cases}
1&(m_3<0)\\
0&(m_3>0)
\end{cases}
\]
for $\lambda\in\Xi_{III}$. 
Here note that the obtained Whittaker function is rapidly decreasing as we have seen. 
We can say that the formula for $\dim W_{\psi,\pi}(\tau^*)^0$ in the first assertion is due to the result on the explicit formula for the Whittaker functions and Yamashita's characterization theorem \cite[Theorem 2.4]{Ya} of some general intertwining operators for discrete series. For \cite[Theorem 2.4]{Ya} we note that large discrete series representations satisfy the condition ``far from the wall''~(for the definition see \cite[Definition 1.7]{Ya}), which is necessary to use \cite[Theorem 2.4]{Ya}. 

To see the formula for $\dim W_{\psi,\pi}(\tau^*)^0$ in the second assertion we have two remarks. 
As the first remark we note that the Whittaker function $W^*$ satisfies the left-equivariance with respect to $N_0\ni n(u_0,u_1,u_2,u_3)\mapsto\psi(\delta n(u_0,u_1,u_2,u_3)\delta^{-1})$ and the right equivariance with respect to $\tl$~(not $\tl^*$) when $W$ is associated with  the unitary character $\psi$ of $N_0$ and $\tl$ is the minimal $K$-type of $\pi_{\lambda}$. In addition we note that $\delta g\delta^{-1}\xi\in G$ for $g\in G$ and that the unitary character $N_0\ni n(u_0,u_1,u_2,u_3)\mapsto\psi(\delta n(u_0,u_1,u_2,u_3)\delta^{-1})$ is parametrized by $(m_0,-m_3)$ instead of $(m_0,m_3)$.

As the second remark let us note that the differential equations characterizing the Whittaker functions are induced by the infinitesimal action of the Schmid operator~(cf.~\cite[Section 5]{Od}). We should further note that the Schmid operator for $\pi_{\lambda}$ with $\lambda\in\Xi_{II}$ is conjugate to that for $\pi_{\lambda}$ with $\lambda\in\Xi_{III}$ by $\delta^{-1}\xi$. In fact, the Schmidt operator for $\lambda\in\Xi_{II}$ is associated with non-compact roots $(1,1),~(2,0),~(0,-2)$~(for the notation of the roots see Section \ref{Lie-gp-alg}). These three roots and their root vectors are conjugate to $(-1,-1),~(0,-2),~(2,0)$ and their root vectors by $\delta^{-1}\xi$ respectively. 
The latter roots define the Schmid operator for $\lambda\in\Xi_{II}$. 
The Whittaker function $W^*$ is therefore annihilated by the Schmid operator for $\lambda\in\Xi_{II}$. The Whittaker function $W^*$ is a rapidly decreasing for $m_3>0$ as $W$ is for $m_3<0$. 

The change of the variables $G\ni g\mapsto \delta g\delta^{-1}\xi\in G$ therefore reduces the problem to the first assertion. 
We can thus say that the second assertion on  $\dim W_{\psi,\pi}(\tau^*)^0$ follows also from the first assertion.  
\end{proof}
We next review the results by Miyazaki-Oda \cite{M-Od-2}~(see also \cite{M-Od-3}) on the Whittaker functions for $P_J$-principal series representations. Their results are given only for the $P_J$-principal series representations ${\rm Ind}_{P_J}^G(1_{N_J}\otimes\nu_z\otimes\sigma)$ with $\sigma=(\cD_n^-,\epsilon)$. However, the Whittaker functions for these $P_J$-principal series are related to those for $P_J$-principal series with $\sigma=(\cD_n^+,\epsilon)$ by the change of variables $G\ni g\mapsto\delta g\delta^{-1}\xi\in G$. 
The results can be stated for $P_J$-principal series representations with both choices of $\cD_n^{\pm}$.  
According to \cite[Theorem 14.15]{Kn} we know that the $P_J$-principal series just mentioned are irreducible when the parameters $z$ are non-zero and purely imaginary. Though such irreducibility is not assumed for the following theorem, we will assume it in the discussion of the Fourier-Jacobi expansion. 
Recall that the $P_J$-principal series enumerated as $1$ and $3$~(respectively~$2$ and $4$) are called  even~(respectively~odd)~(cf.~Section \ref{P_JPS-rep}).
 
\begin{thm}[Miyazaki-Oda]\label{PJWhittaker}
Let $\pi$ be a $P_J$-principal series representation 
and assume that the unitary character $\psi\in\hat{N_0}$ is non-degenerate, i.e. $m_0m_3\not=0$~(cf.~Theorem \ref{LargeDSWhittaker}).\\
(1)
\begin{enumerate}
\item Let $\pi$ be even and associated with $\sigma=(\cD_n^-,\epsilon)$, and let $W$ be the rapidly decreasing Whittaker function for $\pi$ and $\tl^*$, where $\tl$ is the corner $K$-type of $\pi$ with $\Lambda=(-n,-n)$. When $m_3<0$ the restriction of $W|_{A_0}$ to $A_0$ of the Whittaker function $W$ is uniquely given by
\[
a_1^{n+1}a_2^n\exp(2\pi m_3a_2^2)\displaystyle\int_0^{\infty}t^{-n+(1/2)}W_{0,z}(t)\exp(\frac{t^2}{64\pi m_3a_2^2}+\frac{64\pi^3m_0^2m_3a_1^2}{t^2})\dfrac{dt}{t}
\]
up to constant multiples.

If $m_3>0$ there is no such Whittaker function.

When $\pi$ is associated with $(\cD_n^+,\epsilon)$ the corner $K$-type is replaced by $\tl$ with $\Lambda=(n,n)$ and the explicit formula is obtained by replacing $(m_0,m_3)$ with $(m_0,-m_3)$.
\item Let $\pi$ be odd and associated with $\sigma=(\cD_n^-,\epsilon)$, and let $W(g)=c_0(g)v_0^*+c_1(g)v_1^*$ be the the rapidly decreasing Whittaker function for $\pi$ and $\tl^*$, where $\tl$ is the corner $K$-type of $\pi$ with $\Lambda=(-n+1,-n)$. When $m_3>0$ the restriction $W|_{A_0}$ of $W$ to $A_0$ is uniquely given by
\begin{align*}
c_0(a)&=a_1^{n+2}a_2^n\exp(2\pi m_3a_2^2)\displaystyle\int_0^{\infty}t^{-1/2-n}W_{0,z}(t)\exp(\frac{t^2}{64\pi m_3a_2^2}+\frac{64\pi^3m_0^2m_3a_1^2}{t^2})\dfrac{dt}{t},\\
c_1(a)&=a_1^{n+1}a_2^{n-1}\exp(2\pi m_3a_2^2)\displaystyle\int_0^{\infty}t^{3/2-n}W_{0,z}(t)\exp(\frac{t^2}{64\pi m_3a_2^2}+\frac{64\pi^3m_0^2m_3a_1^2}{t^2})\dfrac{dt}{t},
\end{align*}
up to constant multiples. 

If $m_3<0$ there is no such Whittaker function.

When $\pi$ is associated with $(\cD_n^+,\epsilon)$ the corner $K$-type is replaced by $\tl$ with $\Lambda=(n,n-1)$ and the explicit formula is obtained by replacing $(m_0,m_3)$ with $(m_0,-m_3)$.
\end{enumerate}
(2)~Let $\pi$ be associated with $\sigma=(\cD^-_n,\epsilon)$ and the $K$-type $\tl$ above. We have
\[
\dim W_{\psi,\pi}(\tl^*)^0=
\begin{cases}
1&(m_3<0)\\
0&(m_3>0)
\end{cases}.
\]
On the other hand, for $\pi$ with $\sigma=(\cD^+_n,\epsilon)$ and the $K$-type $\tl$ as above, 
\[
\dim W_{\psi,\pi}(\tau^*)^0=
\begin{cases}
1&(m_3>0)\\
0&(m_3<0)
\end{cases}.
\]
\end{thm}
\begin{proof}
First of all, we remark that how to check the rapidly decreasing property of the explicit Whittaker functions is quite similar to the case of large discrete series representations. We do not thus go into detail for this case.

The assertions for the case of $\sigma=(\cD_n^-,\epsilon)$ are stated as \cite[Theorems 8.1 and 8.2]{M-Od-2}~(see also \cite{M-Od-3}). 
We can deduce the dimension formula for $W_{\psi,\pi}(\tau^*)^0$ also from the fact that the multiplicity of the moderate growth Whittaker model for $\pi$ coincides with that for $\cD^{-}_n$, which is due to \cite[Theorem 39]{W-3}~(see also \cite[Theorem 15.6.7]{W-4}). In fact we have now known that both Whittaker functions are rapidly decreasing, where note that the Whittaker function for $\cD^-_n$ with respect to the additive character indexed by $m_3<0$ is given uniquely up to scalars in terms of the exponential function and proved to be rapidly decreasing. 

A Whittaker function for the case of $\sigma=(\cD_n^+,\epsilon)$ is written as
 $W^*(g)=W(\delta g\delta^{-1}\xi)~(g\in G)$ with the Whittaker function $W$ for the case of $\sigma=(\cD^-,\epsilon)$. This reduction of the problem is quite similar to the case of large discrete series representations. 
To see this we now note that the Whitaker functions for $P_J$-principal series are characterized by the differential equations arising from the Casimir operators and the shift operators~(for the definition of the shift operators see \cite[Section 8,~Definition (8.1),~Section 9]{M-Od-1}). The shift operators for $P_J$-principal series with $\sigma=(\cD_n^+,\epsilon)$ and $(\cD_n^-,\epsilon)$ are related each other by $\delta^{-1}\xi$-conjugate. The Casimir operator remains unchanged by $G$-conjugate. We thereby see that $W^*$ satisfies the characterizing differential equations for the case of $P_J$-principal series with $\sigma=(\cD_n^+,\epsilon)$. 
\end{proof}
\subsection*{The case of the principal series representations~(induced from the minimal parabolic subgroup)}
We review the formula of this case from Ishii \cite[Theorems 3.2,~3.3,~3.4]{Is-2} and Niwa \cite{Nw}. We assume that the principal series representations $\pi:=\Ind_{P_0}^G(1_{N_0}\otimes\nu_z\otimes\sigma)$ satisfy
\begin{equation}\label{PS-assumption}
\text{neither $z_1,~z_2$ or $z_1\pm z_2$ is an integer,}
\end{equation}
which is assumed in \cite[(1.2)]{Is-2}. If we assume in addition that $z_1$ and $z_2$ are purely imaginary this implies the irreducibility of $\pi$ by \cite[Theorem 14.15]{Kn} as in the case of the $P_J$-principal series.
\begin{thm}[Ishii,~Niwa]\label{PS-Whittaker}
Let $\pi$ be a principal series representation with the assumption (\ref{PS-assumption}). 
In what follows, we put $y_1:=a_1/a_2$ and $y_2:=a_2^2$ and suppose that $\psi\in\hat{N_0}$ is parametrized by $m_0=m_1=1$.
\begin{enumerate}
\item Let $\pi$ be even and recall that $\tl$ with $\Lambda=(l,l)=(0,0)$ or $\pm(1,1)$ is a multiplicity one $K$-type of $\pi$~(cf.~Section \ref{PS-rep}). Up to constant multiples, the restriction $W|_{A_0}$ of the Whittaker function $W$ for $\pi$ with $K$-type $\tl^*$ is uniquely given by
\begin{align*}
&y_1^2y_2^{3/2}\displaystyle\int_0^{\infty}\displaystyle\int_0^{\infty}K_{(z_1-z_2)/2}(2\pi t_1/t_2)K_{(z_1+z_2)/2}(2\pi t_1t_2)\\
&\times\exp\left(-\pi\left(\frac{y_1^2y_2}{t_1^2}+\frac{t_1^2}{y_2}+\frac{y_2}{t_2^2}+y_2t_2^2\right)\right)\frac{dt_1}{t_1}\frac{dt_2}{t_2}
\end{align*}
for $\Lambda=(0,0)$, and 
\begin{align*}
&y_1^{5/2}y_2^{2}\displaystyle\int_0^{\infty}\displaystyle\int_0^{\infty}K_{(z_1-z_2)/2}(2\pi t_1/t_2)K_{(z_1+z_2)/2}(2\pi t_1t_2)\\
&\times\left(\frac{1}{t_1t_2}-l\frac{t_2}{t_1}\right)\exp\left(-\pi\left(\frac{y_1^2y_2}{t_1^2}+\frac{t_1^2}{y_2}+\frac{y_2}{t_2^2}+y_2t_2^2\right)\right)\frac{dt_1}{t_1}\frac{dt_2}{t_2}
\end{align*}
for $\Lambda=(l,l)=\pm(1,1)$.
\item Let $\pi$ be odd with $(\sigma_1,\sigma_2)=(1,-1)$ and recall that $\tl$ with $\Lambda=(1,0)$ or $(0,-1)$ is a multiplicity one $K$-type of  $\pi$. Furthermore let $y:=(y_1,y_2)$. Up to constant multiples the restriction $W|_{A_0}$ of the Whittaker function $W$ of $\pi$ with $K$-type $\tl^*$ is uniquely given as
\[
\begin{cases}
y_1^2y_2^{3/2}\{(P_1^{(z_1,z_2)}(y)+P_2^{(z_1,z_2)}(y))v_0^*+(Q_1^{(z_1,z_2)}(y)+Q_2^{(z_1,z_2)}(y))v_1^*\}&(\Lambda=(1,0)),\\
y_1^2y_2^{3/2}\{(Q_1^{(z_1,z_2)}(y)-Q_2^{(z_1,z_2)}(y))v_0^*+(-P_1^{(z_1,z_2)}(y)+P_2^{(z_1,z_2)}(y))v_1^*\}&(\Lambda=(0,-1)).
\end{cases}
\]
\end{enumerate}
Here
\begin{align*}
P_1^{(z_1,z_2)}(y)=&(\pi y_2)^{1/2}\displaystyle\int_0^{\infty}\int_0^{\infty}\left(\frac{z_2+1}{2}-\frac{\pi y_1^2y_2}{t_1^2}\right)K_{(z_1-z_2-1)/2}(2\pi t_1/t_2)K_{(z_1+z_2+1)/2}(2\pi t_1t_2)\\
&\times\exp\left(-\pi\left(\frac{y_1^2y_2}{t_1^2}+\frac{t_1^2}{y_2}+\frac{y_2}{t_2^2}+y_2t_2^2\right)\right)\frac{dt_1}{t_1}\frac{dt_2}{t_2},\\
P_2^{(z_1,z_2)}(y)=&-P_1^{(z_1,-z_2)}(y),\\
Q_1^{(z_1,z_2)}(y)=&(\pi y_1)(\pi y_2)^{1/2}\displaystyle\int_0^{\infty}\displaystyle\int_0^{\infty}K_{(z_1-z_2-1)/2}(2\pi t_1/t_2)K_{(z_1+z_2+1)/2}(2\pi t_1t_2)\\
&\times\exp\left(-\pi\left(\frac{y_1^2y_2}{t_1^2}+\frac{t_1^2}{y_2}+\frac{y_2}{t_2^2}+y_2t_2^2\right)\right)\frac{dt_1}{t_1}\frac{dt_2}{t_2},\\
Q_2^{(z_1,z_2)}(y)=&-Q_1^{(z_1,-z_2)}(y).
\end{align*}
For an odd principal series representation $\pi$ with $(\sigma_1,\sigma_2)=(-1,1)$ the explicit formula for the Whittaker function is obtained by just  exchanging $z_1$ and $z_2$.
\end{thm}
Toward the Fourier expansion we need the Whittaker functions attached to unitary characters of $N_0$ for arbitrary $(m_0,m_3)\in\Z^2$ with $m_0m_3\not=0$. From this theorem we deduce the following:
\begin{cor}\label{PSWhittaker-Cor}
We keep the notation and the assumption as in Theorem \ref{PS-Whittaker}.\\
(1)~ Let $\pi$ be an even principal series representation. For any non-degenerate character $\psi$ with $(m_0,m_3)\in\R^2$~(that is, $m_0m_3\not=0$), up to constant multiples, the Whittaker functions restricted to $A_0$ can be expressed uniquely as
\begin{align*}
&(|m_0|y_1)^2(|m_3|y_2)^{3/2}\displaystyle\int_0^{\infty}\displaystyle\int_0^{\infty}K_{(z_1-z_2)/2}(2\pi t_1/t_2)K_{(z_1+z_2)/2}(2\pi t_1t_2)\\
&\times\exp\left(-\pi\left(\frac{m_0^2|m_3|y_1^2y_2}{t_1^2}+\frac{t_1^2}{|m_3|y_2}+\frac{|m_3|y_2}{t_2^2}+|m_3|y_2t_2^2\right)\right)\frac{dt_1}{t_1}\frac{dt_2}{t_2}
\end{align*}
for $\Lambda=(0,0)$, and 
\begin{align*}
&(|m_0|y_1)^{5/2}(|m_3|y_2)^{2}\displaystyle\int_0^{\infty}\displaystyle\int_0^{\infty}K_{(z_1-z_2)/2}(2\pi t_1/t_2)K_{(z_1+z_2)/2}(2\pi t_1t_2)\\
&\times\left(\frac{1}{t_1t_2}-l\frac{t_2}{t_1}\right)\exp\left(-\pi\left(\frac{m_0^2|m_3|y_1^2y_2}{t_1^2}+\frac{t_1^2}{|m_3|y_2}+\frac{|m_3|y_2}{t_2^2}+|m_3|y_2t_2^2\right)\right)\frac{dt_1}{t_1}\frac{dt_2}{t_2}
\end{align*}
for $\Lambda=(l,l)=\pm(1,1)$. 

When $\pi$ is an odd principal series representations with $(\sigma_1,\sigma_2)=(1,-1)$, the explicit formula is 
uniquely given as
\[
\begin{cases}
(|m_0|y_1)^2(|m_3|y_2)^{3/2}\{(P_1^{(z_1,z_2)}((|m_0|y_1,|m_3|y_2))+P_2^{(z_1,z_2)}((|m_0|y_1,|m_3|y_2)))v_0^*\\
+(Q_1^{(z_1,z_2)}((|m_0|y_1,|m_3|y_2))+Q_2^{(z_1,z_2)}((|m_0|y_1,|m_3|y_2))v_1^*\}&(\Lambda=(1,0)),\\
((|m_0|y_1)^2(|m_3|y_2)^{3/2}\{(Q_1^{(z_1,z_2)}((|m_0|y_1,|m_3|y))-Q_2^{(z_1,z_2)}((|m_0|y_1,|m_3|y_2))v_0^*\\
+(-P_1^{(z_1,z_2)}((|m_0|y_1,|m_3|y_2))+P_2^{(z_1,z_2)}((|m_0|y_1,|m_3|y_2))v_1^*\}&(\Lambda=(0,-1)),
\end{cases}
\]
up to constant multiples. 
As for odd principal series representations with $(\sigma_1,\sigma_2)=(-1,1)$ the explicit formula is given similarly by exchanging $z_1$ and $z_2$.\\
(2)~We have $\dim W_{\psi,\pi}(\tau_{\Lambda}^*)^0=1$ for $\psi$ and $(\pi,\tau_{\Lambda})$s above. 
\end{cor}
\begin{proof}
Putting 
\[
a_{m_0,m_3}:=\diag(m_0\sqrt{|m_3|},\sqrt{|m_3|},(m_0\sqrt{|m_3|})^{-1},\sqrt{|m_3|}^{-1}),
\]
the explicit formulas for the Whittaker functions above are obtained by
\[
\begin{cases}
W(a_{m_0,m_3}g)&(m_3>0),\\
W(\delta a_{m_0,m_3}g\delta^{-1}\xi)&(m_3<0),
\end{cases}
\]
for $g\in G$ with the Whittaker functions $W$ attached to the character of $N_0$ with $(m_0,m_3)=(1,1)$, 
where recall that $\delta:=
\begin{pmatrix}
-I_2 & 0_2\\
0_2 & I_2
\end{pmatrix}$ and $\xi:=
\begin{pmatrix}
J'_2 & 0_2\\
0_2 & J'_2
\end{pmatrix}$ with $J'_2=
\begin{pmatrix}
0 & 1\\
1 & 0
\end{pmatrix}$~(cf.~Theorem \ref{LargeDSWhittaker} (2)). As we have remarked in the proof of Theorem \ref{LargeDSWhittaker} (2) we note that $\delta a_{m_0,m_3}g\delta^{-1}\xi\in G$ for $g\in G$. 
The differential equations characterizing the Whittaker functions are induced by the infinitesimal actions of the two generators of the center of the universal enveloping algebra for $\g$, one of which is the Casimir operator and another of which coincides with some composite of the shift operators~(cf.~\cite[Remark 3]{Is-2}). For the definition of the shift operators see \cite[Section 8,~Definition (8.1),~Section 9]{M-Od-1}. 

We now justify the above remark on the explicit formula, whose argument is similar to the case of $P_J$-principal series representations~(cf.~Theorem \ref{PJWhittaker}). 
The formula in the assertion is $W(a_{|m_0|,m_3}a)$~(or $W(a_{|m_0|,m_3}a\xi)$) for $a\in A_0$. 
Now let us note that $a_{|m_0|,m_3}=a_{m_0,m_3}{\rm diag}(\epsilon,1,\epsilon,1)$ with some $\epsilon\in\{\pm 1\}$.  Taking into account the right equivariance of the Whittaker functions with respect to $\tl^*$, we can then verify that the difference between $W(a_{|m_0|,|m_3|}a)$ and $W(a_{m_0,m_3}a)$~(or $W(a_{|m_0|,|m_3|}a\xi)$ and $W(a_{m_0,m_3}a\xi)$) is given at most as the multiples of the $K$-type vectors by $\{\pm 1\}$. 
We remark that the change of variables $g\mapsto\delta g\delta^{-1}\xi$ leads to switching between the explicit formulas for even principal series with the multiplicity one $K$-types $\tau_{(1,1)}$ and $\tau_{(-1,-1)}$, and also to that between the formulas for odd principal series with the multiplicity one $K$-types $\tau_{(1,0)}$ and $\tau_{(0,-1)}$. For this note further that the change of variables $g\mapsto(\delta^{-1}\xi)^{-1}g(\delta^{-1}\xi)$ induces the conjugation of the shift operators and the Casimir operator as in the case of $P_J$-principal series.

Regarding (2) we verify the assertion by the argument similar to that for the cases of large discrete series and $P_J$-principal series. 
Note that the $K$-Bessel function $K_{\nu}(y)$ on $\R_{>0}$ 
is rapidly decreasing as is pointed out in the proof of Theorem \ref{LargeDSWhittaker}.  The result of $\dim W_{\psi,\pi}(\tau^*)^0$ follows from \cite[Theorem 39]{W-3}~(also from \cite[Theorem 15.6.7]{W-4}), together with the rapidly decreasing property just mentioned.
\end{proof}
\begin{rem}
There are another explicit integral expressions of the Whittaker functions called ``Mellin-Barnes type integrals'', which are more suitable to the archimdean local theory of automorphic $L$-functions. For this we cite Moriyama  \cite{Mo-1} and Ishii \cite{Is-2},~\cite{Is-3}.
\end{rem}
\subsection{Whittaker functions attached to degenerate characters of $N_0$}\label{Deg-Whittaker-sec}
The Whittaker functions we have studied are proved to be rapidly decreasing. They are attached to non-degenerate characters of $N_0$.  
However, towards the Fourier Jacobi expansion, we have to also study the Whitatker functions attached to degenerate characters of $N_0$, i.e. unitary characters $\psi$ of $N_0$ parametrized by $(m_0,m_3)$ with $m_0m_3=0$. 
For the following proposition we remark that the case of $m_3=0$ for principal series is due to Taku Ishii.    
\begin{prop}\label{DegnerateWhittaker}
Let $\pi$ be a large discrete series representation, a $P_J$-principal series representations  or a principal series representation.  Let $\tau$ be the minimal $K$-type of $\pi$ when $\pi$ is a large discrete series representation. 
When $\pi$ is a $P_J$-principal series representation~(respectively~a principal series representation) let $\tau$ be the corner $K$-type of $\pi$ in the sense of Section \ref{P_JPS-rep}~(respectively~the minimal $K$-type in  Section \ref{PS-rep}) . 

For degenerate characters $\psi$ of $N_0$ parametrized by $(m_0,m_3)\not=(0,0)$ such that $m_0m_3=0$ the Whittaker functions for $\pi$ above with respect to $\tau$ is not rapidly decreasing. 
\end{prop}
\begin{proof}
We omit the case of $P_J$-principal series representations. In fact, following the coming argument for the case of large discrete series representations, the case of $P_J$ principal series representations is settled similarly by the differential equations in \cite[Propositions 7.1,~7.3]{M-Od-2}~(see also \cite{M-Od-3}). 

Let $\pi$ be a large discrete series representation and $\Lambda=(\Lambda_1,\Lambda_2)$ be the highest weight of the minimal $K$-type $\tau$ of $\pi$. We think only of $\pi$ with Harish Chandra parameter in $\Xi_{III}$ since the argument for the case of $\Xi_{II}$ is reduced to this by the reasoning to prove Theorem \ref{LargeDSWhittaker} (2). 
We put $d_{\Lambda}:=\Lambda_1-\Lambda_2$ as in Section \ref{RepMaxCpt}. Recall that the Whittaker function $W$ is written as $W(g)=\sum_{k=0}^{d_{\Lambda}}c_k(g)v_k^*$. Here note that $W$ satisfies the right $K$-equivariance with respect to $\tau^*$, whose highest weight is $(-\Lambda_2,-\Lambda_1)$. Following \cite[Section 8]{Od}~(see also Theorem \ref{LargeDSWhittaker}) we put 
\[
c_k(a)=a_1^{-\Lambda_2+1-d_{\Lambda}}a_2^{-\Lambda_2}\left(\frac{a_1}{a_2}\right)^{k}e^{2\pi m_3a_2^2}h_k(a)
\]
for $a\in A_0$. 
According to \cite[Section 8,~(G-1),~Lemma 8.1]{Od} $h_{d_{\Lambda}}(a)$ satisfies the following differential equations:
\begin{align}
&\partial_1h_{d_{\Lambda}}(a)+2\pi\sqrt{-1}\frac{a_1}{a_2}m_0 h_{d_{\Lambda}-1}(a)=0,\\
&(\partial_1\partial_2+4\pi^2\frac{a_1^2}{a_2^2}m_0^2)h_{d_{\Lambda}}(a)=0,\\
&((\partial_1+\partial_2)^2+(-2\Lambda_1-2)(\partial_1+\partial_2)+(2\Lambda_1+1)-4\pi\sqrt{-1}a_2^2m_3\partial_2)h_{d_{\Lambda}}(a)=0,
\end{align}
where $\partial_i=a_i\dfrac{\partial}{\partial a_i}$ for $i=1,2$. Here we note that the condition $m_0\not=0$ is necessary to obtain (4.3) and (4.4).

Suppose first that $m_0=0$. Then, due to $(4.2)$, $h_{d_{\Lambda}}(a)$ is constant with respect to $a_1$, which implies that $c_{d_{\Lambda}}(a)$ is not rapidly decreasing. We next assume that $m_3=0$. Then the equation (4.4) admits the following decomposition:
\[
(\partial_1+\partial_2-1)(\partial_1+\partial_2-2\Lambda_1-1)h_{d_{\Lambda}}(a)=0
\]
namely $h_{d_{\Lambda}}(a)$ is the eigen-function with respect to $\partial_1+\partial_2$ with the eigenvalues $1$ or $2\Lambda_1+1$. This implies that the equation $(4.3)$ can be rewritten as
\[
\left(\dfrac{\partial^2}{\partial a_1^2}+\dfrac{1-\lambda}{a_1}\dfrac{\partial}{\partial a_1}-\dfrac{4\pi^2m_0^2}{a_2^2}\right)h_{d_{\Lambda}}(a)=0~\text{or}~\left(\dfrac{\partial^2}{\partial a_2^2}+\dfrac{1-\lambda}{a_2}\dfrac{\partial}{\partial a_2}-\dfrac{4\pi^2a_1^2m_0^2}{a_2^4}\right)h_{d_{\Lambda}}(a)=0
\]
with $\lambda=1$ or $2\Lambda_1+1$. This can be understood by the following Bessel-type differential equation~(cf.~\cite[(4.5.9)]{AAR})
\[
\dfrac{d^2}{dx^2}u+\dfrac{1-2\alpha}{x}\dfrac{d}{dx}u+\left((\beta\gamma x^{\gamma-1})^2+\dfrac{\alpha^2-\nu^2\gamma^2}{x^2}\right)u=0.
\]
For our situation the case of $\nu=\pm\dfrac{\lambda}{2},~\alpha=\dfrac{\lambda}{2},~\beta=\dfrac{2\pi\sqrt{-1}m_0}{a_2}$ and $\gamma=1$~(respectively~$\beta=2\pi m_0\sqrt{-1}a_1$ and $\gamma=-1$) is valid when we choose $a_1$~(respectively~$a_2$) as the variable. The solutions are given by linear combinations of  
\[
(a_1a_2)^{\frac{\lambda}{2}}J_{\pm\frac{\lambda}{2}}(2\pi m_0\sqrt{-1}\dfrac{a_1}{a_2})
\]
with the Bessel function $J_{\nu}$ parametrized by $\nu$. By the asymptotic expansion of $J_{\nu}$ along purely imaginary numbers~(e.g. see \cite[Section 4.8]{AAR}) we know that $c_{d_{\Lambda}}(a)$ is not rapidly decreasing with respect to $a_2$. 

We are now left with the case of principal series representations. Recall that the differential equations characterizing the Whittaker function $W$ are given in \cite[Theorems 10.1,~11.3]{M-Od-1}~(see also \cite[Theorems  1.5,~1.6]{Is-2}). Furthermore recall that we have used the coordinate $(y_1,y_2):=(a_1/a_2,a_2^2)$ for the case of principal series. With the notation $\partial_i=y_i\displaystyle\frac{\partial}{\partial y_i}$ for $i=1,~2$ the characterizing differential equations are written for 
\[
\psi(a)=y_1^{-3/2}y_2^{-2}W(a).
\]
The differential equations just mentioned in \cite{M-Od-1} and \cite{Is-2} are those for the case of $(m_0,m_3)=(1,1)$. The characterizing differential equation for a general $(m_0,m_3)$ is obtained by replacing $2\pi y_1$ and $2\pi y_2$ with $2\pi m_0y_1$ and $2\pi m_3y_2$ respectively. 

For the case of an even principal series representation we deduce 
\[
\begin{cases}
(\partial_1^2-\nu_1^2)(\partial_1^2-\nu_2^2)\psi=0&(m_0=0),\\
(\partial_2^2-\displaystyle\frac{(\nu_1+\nu_2)^2}{4})(\partial_2^2-\displaystyle\frac{(\nu_1-\nu_2)^2}{4})\psi=0&(m_3=0).
\end{cases}
\]
from \cite[Theorem 1.5]{Is-2} or \cite[Theorem 10.1]{M-Od-1}. 
For the case of an odd principal series representation we write $\psi_0(a)v_0^*+\psi_1(a)v_1^*$ for $\psi(a)$. We deduce from \cite[Theorem 1.6]{Is-2} or \cite[Theorem 11.3]{M-Od-1} that
\begin{equation}\label{Diffeq-DefWh}
\begin{cases}
(\partial_1^2-\nu^2)\psi_0=0,~(\partial_1^2-(\nu_1^2+\nu_2^2-\nu^2))\psi_1=0&(m_0=0),\\
(\partial_2^2-\displaystyle\frac{(\nu_1+\nu_2)^2}{4})(\partial_2^2-\displaystyle\frac{(\nu_1-\nu_2)^2}{4})\psi_i=0~\text{for $i=0,~1$}&(m_3=0).
\end{cases}
\end{equation}
As a result we see that $W$ is an eigenfunction of $\partial_1$ or $\partial_2$ for both of even and odd principal series representations. From this we deduce that $W$ is of polynomial order but not of rapid decay with respect to $y_1$ or $y_2$ when $m_0=0$ or $m_3=0$ respectively. 

We now explain how to verify this. The case of $m_0=0$ is not difficult to prove. We write down the outline of the proof for the case of an odd principal series representation with $m_3=0$, which is more difficult to verify than the case of an even principal series with $m_3=0$. We put 
\[
f:=\psi_0+\sqrt{-1}\psi_1,~g=\psi_0-\sqrt{-1}\psi_1.
\]
With the notations
\[
D_1=\partial_1^2+2\partial_2^2-2\partial_1\partial_2-\dfrac{\nu_1^2+\nu_2^2}{2},~D_2=\partial_1^2-\nu^2,~D_3=(\partial_1-2\partial_2)^2-\nu^2,~Y_1=2\pi m_0y_1
\]
the characterizing differential equations in \cite[Theorem 1.6]{Is-2} or \cite[Theorem 11.3]{M-Od-1} are rewritten as
\begin{align*}
&(D_1-Y_1^2-Y_1)f=0,\\
&(D_1-Y_1^2+Y_1)g=0,\\
&(D_2+D_3-2Y_1^2-2Y_1)f+(D_2-D_3+4Y_1\partial_2)g=0,\\
&(D_2-D_3-4Y_1\partial_2)f+(D_2+D_3-2Y_1^2+2Y_1)g=0.
\end{align*}
From these differential equations we see that
\begin{align*}
&(D_2+D_3-2D_1)f+(D_2-D_3+4Y_1\partial_2)g=0,\\
&(D_2-D_3-4Y_1\partial_2)f+(D_2+D_3-2D_1)g=0,
\end{align*}
which leads to
\begin{align*}
(\nu_1^2+\nu_2^2-2\nu^2)f+4(\partial_1\partial_2-\partial_2^2+Y_1\partial_2)g=0,\\
4(\partial_1\partial_2-\partial_1^2-Y_1\partial_2)f+(\nu_1^2+\nu_2^2-2\nu^2)g=0.
\end{align*}
From these and  the relation $\partial_1\partial_2(Y_1\partial_2)=Y_1(\partial_1+1)\partial_2^2$ we deduce that $f$ and $g$ satisfy the same differential equation as (\ref{Diffeq-DefWh}). 
As a consequence we have the desired differential equations for $\psi_0$ and $\psi_1$ as in (\ref{Diffeq-DefWh}).

\end{proof}
\begin{rem}
According to \cite[Theorem 8.2 (2)]{Hi-1} and \cite[Theorem 7.2]{Hi-2} the multiplicities of the moderate growth Whittaker functions for non-trivial degenerate characters with $m_3=0$ are two or at most two for the cases of large discrete series or $P_J$-principal series respectively. The proof of this proposition shows that such Whittaker function of moderate growth is uniquely up to scalars for each of the two eigenvalues $\lambda$. Indeed, note that the $K$-Bessel function $K_{\frac{\lambda}{2}}$ is a linear combination of $J_{\pm\frac{\lambda}{2}}$ when the eigenvalue $\lambda$ is not equal to $1$. The coefficient function  $h_{d_{\Lambda}}(a)$ of the moderate growth solution is a constant multiple of  
\[
(a_1a_2)^{\frac{\lambda}{2}}K_{\frac{\lambda}{2}}(2\pi |m_0|\frac{a_1}{a_2})
\]
for both of the two representations when $\lambda\not=1$. When $\lambda=1$, which occurs only for large discrete series, such a solution is written in terms of the exponential function. 
The moderate growth Whittaker functions just mentioned are not of rapidly decreasing as in the statement of the proposition. 
\end{rem}
\subsection{Fourier-Jacobi type spherical functions by Hirano}\label{Result-Hirano}
We review Hirano's explicit formula for Fourier-Jacobi type spherical functions in \cite{Hi-1},~\cite{Hi-2} and \cite{Hi-3}. 
Given an admissible representation $\pi$ of $G=Sp(2,\R)$ with the multiplicity one $K$-type $\tl$ and an irreducible unitary representation $\rho$ of $G_J$, the Fourier-Jacobi type spherical functions are defined as the restriction of the intertwining operators in
\[
\Hom_{(\g,K)}(\pi,C^{\infty}\Ind_{G_J}^G\rho)
\]
to the $K$-type $\tl$. We note that, for the spherical functions to be well defined, we have to impose the multiplicity one property of $\tau$ as in the case of the Whittaker functions. 
Such restricted intertwining operators are contained in $C_{\rho,\tl^*}^{\infty}(G_J\backslash G/K):=$
\[
\{\text{$C^{\infty}$-function}~W:G\rightarrow H_{\rho}\boxtimes V_{\Lambda}^*\mid W(rgk)=\rho(r)\boxtimes\tl^*(k)^{-1}W(g)~\forall (r,g,k)\in G_J\times G\times K\},
\]
where $H_{\rho}$ and $V_{\Lambda}^*$ denote the representation spaces of $\rho$ and the contragredient $\tl^*$ of $\tl$ respectively. Following Hirano \cite{Hi-1},~\cite{Hi-2} and \cite{Hi-3} we denote the image of the restriction map by $\cJ_{\rho,\pi}(\tl^*)$. We call this the space of the Fourier-Jacobi type spherical functions of type $(\rho,\pi;\tl^*)$. In terms of the theory of the Fourier expansion of cusp forms we are interested in
\[
\cJ_{\rho,\pi}(\tl^*)^{00}:=\{W\in\cJ_{\rho,\pi}(\tl^*)\mid \text{$W$ is rapidly decreasing with respect to $A_J$}\}.
\]
This notation is to avoid the confusion with $\cJ_{\rho,\pi}^{0}(*)$ in \cite{Hi-1},~\cite{Hi-2} and \cite{Hi-3}~(as we have remarked in the introduction). 

Now recall that irreducible unitary representations of $G_J$ with the central character parametrized by $m\not=0$ is of the form $\rho_{m,\pi_1}:=\pi_1\boxtimes\tilde{\nu}_m$~(cf.~Proposition \ref{UnitaryDual-G_J}) and that we have introduced the notation $\{w_l\}_{l\in L},~\{u_j^m\}_{j\in J},~\{v_k^*\}_{0\le k\le d_{\Lambda}}$ for a basis of $W_{\pi_1},~\cU_m,~V_{\Lambda}^*$ respectively~(cf.~Sections \ref{RepSL2},~\ref{UnitaryRepJacobi},~\ref{RepMaxCpt}). 
As is pointed out in Hirano's papers~(~e.g.~\cite[Lemma 4.4]{Hi-1}), the restriction of $W\in\cJ_{\rho,\pi}(\tl^*)^{00}$ to $A_J$ is written as
\[
W(a_J)=\sum_{
\begin{subarray}{c}
j\in J,~0\le k\le d_{\Lambda}\\
\text{s.t.}~l=l(j,k)\in L
\end{subarray}}c_{j,k}(a_J)w_l\otimes u_j^m\otimes v_{k}^*\quad(a_J\in A_J)
\]
with coefficient functions $c_{j,k}(a_J)$. 
Here 
\[
l(j,k)=-j+k+\Lambda_2,
\]
for which note that the highest weight of $\tl^*$ is $(-\Lambda_2,-\Lambda_1)$.  This is deduced from the compatibility of $\tl^*$ and the $SO_2(\R)$-types of $\rho_{\pi_1,m}$ with respect to the $K\cap SL_2(\R)$-action~(note that $K\cap SL_2(\R)=SO_2(\R)$). 
We should note that we need the Meijer $G$-function $G_{p,q}^{m,n}$ to review Hirano's explicit formulas for the Fourier-Jacobi type spherical functions. For the detail on the Meijer $G$-function $G_{p,q}^{m,n}$ we cite \cite[Appendix]{Hi-1} and the references by Meijer cited therein. 
In what follows, we put 
\[
x:=4\pi m a_1^2,\quad x':=-4\pi m a_1^2
\]
in order to describe the explicit formula for the Fourier-Jacobi type spherical functions. 
In advance of the review we remark that the coefficient functions of the Fourier-Jacobi type spherical functions are given explicitly by using
\[
e^{\frac{1}{2}x}G_{p,q}^{q,0}\left(x\left|
\begin{array}{c}
a_1,\cdots,a_p\\
b_1,\cdots,b_q
\end{array}\right.\right)~\text{or}~e^{\frac{1}{2}x'}G_{p,q}^{q,0}\left(x'\left|
\begin{array}{c}
a_1,\cdots,a_p\\
b_1,\cdots,b_q
\end{array}\right.\right)
\]
for $(p,q)=(1,2),~(2,3)$ or $(3,4)$ and that these are rapidly decreasing~(cf.~\cite[Lemma A.2 (2)]{Hi-1}).
\subsection*{(I)~The case of holomorphic or anti-holomorphic discrete series representations}
This case can be said to be simplest among the cases we will take up. In fact, the Fourier-Jacobi type spherical functions for holomorphic discrete series with scalar minimal $K$-types are reproduced from the classical theory such as Eichler-Zagier~\cite{E-Z}. However, we cannot say that many experts are familiar with these spherical functions for the cases of non-scalar minimal $K$-types as much as those for the cases of scalar minimal $K$-types. Hirano \cite[Theorems 6.3,~6.4,~6.5,~6.6,~6.7 and 6.8]{Hi-1} deals with the cases of all holomorphic and anti-holomorphic discrete series with any minimal $K$-types, which cover all vector valued holomorphic Siegel modular forms of degree two.
\begin{thm}\label{FJ-holom-antiholom-DS}
\begin{enumerate} 
\item Let $\pi_{\lambda}$ be the holomorphic discrete series representation with Harish Chandra parameter $\lambda=(\lambda_1,\lambda_2)\in\Xi_{I}$, and $\tl$ be the minimal $K$-type of $\pi_{\lambda}$. We have 
\[
\dim \cJ_{\rho,\pi_{\lambda}}(\tl^*)^{00}=
\begin{cases}
1&(\rho=\rho_{\cD_{n_1}^+,m}~\text{with $\lambda_2+\frac{3}{2}\le n_1\le\lambda_1+\frac{1}{2}$ and $m>0$}),\\
0&(\text{otherwise}).
\end{cases}
\]
If $\dim \cJ_{\rho,\pi_{\lambda}}(\tl^*)^{00}\not=0$, the coefficient function $c_{\frac{1}{2},k_0}(a)$ of the restriction $W\in\cJ_{\rho,\pi_{\lambda}}(\tl^*)^{00}$ to $A_J$ with $k_0=n_1-\lambda_2-\frac{3}{2}$ is 
\[
x^{\frac{1}{2}(\lambda_1+1-k_0)}e^{-\frac{1}{2}x},
\]
up to scalars. The other coefficients are given by the formulas \cite[(6.4),~(6.5)]{Hi-1} inductively from $c_{\frac{1}{2},k_0}(a)$.  
\item Let $\pi_{\lambda}$ be the anti-holomorphic discrete series representation with Harish Chandra parameter $\lambda=(\lambda_1,\lambda_2)\in\Xi_{IV}$, and $\tl$ be the minimal $K$-type of $\pi_{\lambda}$. 
We have
\[
\dim \cJ_{\rho,\pi_{\lambda}}(\tl^*)^{00}=
\begin{cases}
1&(\rho=\rho_{\cD_{n_1}^-,m}~\text{with $-\lambda_1+\frac{3}{2}\le n_1\le-\lambda_2+\frac{1}{2}$ and $m<0$}),\\
0&(\text{otherwise}).
\end{cases}
\]
If $\dim \cJ_{\rho,\pi_{\lambda}}(\tl^*)^{00}\not=0$, the coefficient function $c_{-\frac{1}{2},k_0}(a)$ of the restriction $W\in\cJ_{\rho,\pi_{\lambda}}(\tl^*)^{00}$ to $A_J$ with $k_0=-n_1-\lambda_2+\frac{1}{2}$ is 
\[
{x'}^{\frac{1}{2}(-\lambda_1+2+k_0)}e^{-\frac{1}{2}x'},
\]
up to scalars. The other coefficients are given by the recurrence formulas in \cite[Theorem 6.5]{Hi-1} inductively from $c_{-\frac{1}{2},k_0}(a)$.
\item Let $m=0$ and $\lambda\in\Xi_{I}\cup\Xi_{IV}$. We have $\dim \cJ_{\rho,\pi_{\lambda}}(\tl^*)^{00}=0$ for any $\rho$.
\end{enumerate}
\end{thm}
For this theorem we remark that the last assertion is due to \cite[Theorems 6.7 and 6.8]{Hi-1}. 
We can deduce the second assertion from the first by the same argument as the proof of Theorem \ref{LargeDSWhittaker} (2). 
However, the explicit formula is relatively simple, compared with non-holomorphic cases. 
We thereby write down the formulas explicitly for both of holomorphic and antiholomorphic cases as above.  
\subsection*{(II)~The case of large discrete series representations}
All the theorems for this case are included in Hirano \cite{Hi-1}. Let $\pi_{\lambda}$ be the large discrete series representation with Harish Chandra parameter $\lambda$ and $\tl$ be the minimal $K$-type of $\pi_{\lambda}$ with Blattner parameter $\Lambda$~(cf.~Section~\ref{DS-rep}). 
For this case we have the following coincidence 
\[
\cJ_{\rho,\pi_{\lambda}}(\tl^*)=\{W\in C_{\rho,\tl^*}(G_J\backslash G/K)\mid \cD_{\tl^*,\rho}\cdot W=0\},
\]
where $\cD_{\tl^*,\rho}$ is the (generalized) Schmid operator as in Yamashita \cite{Ya}. According to \cite[Theorem 2.4]{Ya} we have the following isomorphism
\[
\Hom_{(\g,K)}(\pi_{\lambda}, C^{\infty}\Ind_{G_J}^G\rho)\simeq \cJ_{\rho,\pi_{\lambda}}(\tl^*).
\]
Here, as in the proof of Theorem \ref{LargeDSWhittaker}, we note again that large discrete series representations satisfy the condition ``far from the wall''~(for the definition see \cite[Definition 1.7]{Ya}). 
For the expression of a Fourier-Jacobi type spherical function $W$ we note that 
\[
l(j,k)=
\begin{cases}
-j+k+\lambda_2&(\lambda\in\Xi_{II})\\
-j+k+\lambda_2-1&(\lambda\in\Xi_{III}))
\end{cases}. 
\]
We first review Hirano's results on the case of $\pi_{\lambda}$ for $\lambda\in\Xi_{II}$~(cf.~\cite[Theorems 7.5,~7.6]{Hi-1}).
\begin{thm}\label{explicit-FJ-I}
Let $\pi_{\lambda}$ be the large discrete series representation of $G$ with Harish Chandra parameter $\lambda=(\lambda_1,\lambda_2)\in\Xi_{II}$ and suppose that $m>0$.
\begin{enumerate}
\item If $\pi_1=\cP_s^{\tau}~(\tau=\pm\frac{1}{2},~s\in\sqrt{-1}\R)$ or $C_s^{\tau}~(\tau=\pm\frac{1}{2},~0<s<\frac{1}{2})$, we have 
\[
\dim \cJ_{\rho_{m,\pi_1},\pi_{\lambda}}(\tl^*)^{00}=1.
\]
The coefficient function $c_{\frac{1}{2},k_0}(a)$ of the restriction of $W\in \cJ_{\rho_{m,\pi_1},\pi_{\lambda}}(\tl^*)^{00}$ to $A_J$ for $0\le k_0\le d_{\Lambda}$ with $l(\frac{1}{2},k_0)\in L$ is  
\[
e^{\frac{1}{2}x}G_{2,3}^{3,0}\left(x\left|
\begin{array}{c}
\dfrac{2s+5+2d_{\Lambda}}{4},\dfrac{-2s+5+2d_{\Lambda}}{4}\\
\dfrac{\lambda_1+2}{2},\dfrac{\lambda_1+3}{2},\dfrac{-k_0+d_{\Lambda}-\lambda_2+2}{2}
\end{array}\right.\right),
\]
up to constants, and the other coefficient functions are obtained inductively from $c_{\frac{1}{2},k_0}$ by the recurrence relations \cite[(7.1),~(7.3) in Proposition 7.1 and (7.4),~(7.6) in Lemma 7.2]{Hi-1}.
\item If $\pi_1=\cD_{n_1}^-$ with $n_1\in\frac{1}{2}\Z_{\ge 3}\setminus\Z$ and $n_1\le -\lambda_2+\frac{1}{2}$ we have 
\[
\dim \cJ_{\rho_{m,\pi_1},\pi_{\lambda}}(\tl^*)^{00}=1.
\]
The coefficient function $c_{\frac{1}{2},k_0}(a)$ of the restriction of $W\in \cJ_{\rho_{m,\pi_1},\pi_{\lambda}}(\tl^*)^{00}$ to $A_J$ with $k_0=-n_1+\frac{1}{2}-\lambda_2$ is
\[
e^{\frac{1}{2}x}G_{1,2}^{2,0}\left(x\left|
\begin{array}{c}
\dfrac{\lambda_1+4+k_0}{2}\\
\dfrac{\lambda_1+2}{2},\dfrac{\lambda_1+3}{2}
\end{array}\right.\right)=x^{\frac{3+2\lambda_1}{4}}W_{\kappa,\mu}(x)
\]
up to constants, where $\kappa=-\frac{1}{4}(1+2k_0),~\mu=-\frac{1}{4}$. The other coefficients are obtained inductively from $c_{\frac{1}{2},k_0}$ by the recurrence relation \cite[(7.1),~(7.3) in Proposition 7.1]{Hi-1}. 
\item If  $\pi_1=\cD^-_{n_1}$ with $n_1\in\frac{1}{2}\Z_{\ge 3}\setminus\Z$ and $n_1>-\lambda_2+\frac{1}{2}$ or $\pi_1=\cD_{n_1}^+$ with $n_1\in\frac{1}{2}\Z_{\ge 3}\setminus\Z$, we have 
\[
\dim \cJ_{\rho_{m,\pi_1},\pi_{\lambda}}(\tl^*)^{00}=0.
\]
\end{enumerate}
\end{thm}
\begin{thm}\label{explicit-FJ-II}
Let $\pi_{\lambda}$ be as in Theorem \ref{explicit-FJ-I} and suppose that $m<0$.
\begin{enumerate}
\item If $\pi_1=\cD_{n_1}^+$ with $n_1\in\frac{1}{2}\Z_{\ge 3}\setminus\Z$ and $n_1>\lambda_1+\frac{1}{2}$ we have 
\[
\dim \cJ_{\rho_{m,\pi_1},\pi_{\lambda}}(\tl^*)^{00}=1.
\]
The coefficient function $c_{j_0,k_0}(a)$ of the restriction of $W\in \cJ_{\rho_{m,\pi_1},\pi_{\lambda}}(\tl^*)^{00}$ to $A_J$ with $j_0\in J$ and $0\le  k_0\le d_{\Lambda}$ such that $l(j_0,k_0)=n_1$ is
\[
e^{\frac{1}{2}x'}G_{2,3}^{3,0}\left(x'\left|
\begin{array}{c}
\dfrac{2\lambda_1+5-2j_0}{4},\dfrac{2\lambda_1+7-2j_0}{4}\\
\dfrac{\lambda_1+2}{2},\dfrac{\lambda_1+3}{2},\dfrac{-k_0+d_{\Lambda}-\lambda_2+2}{2}
\end{array}\right.\right)
\]
up to constants. The other coefficients are obtained inductively from $c_{j_0,k_0}$ by the recurrence relations in  \cite[Proposition 7.1]{Hi-1}.
\item If $\pi_1=\cP_s^{\tau}~(\tau=\pm\frac{1}{2},~s\in\sqrt{-1}\R)$ or $C_s^{\tau}~(\tau=\pm\frac{1}{2},~0<s<\frac{1}{2})$, we have 
\[
\dim \cJ_{\rho_{m,\pi_1},\pi_{\lambda}}(\tl^*)^{00}=0.
\]  
\item If $\pi_1=\cD_{n_1}^+$ with $n_1\in\frac{1}{2}\Z_{\ge 3}\setminus\Z$ and $n_1\le\lambda_1+\frac{1}{2}$ or $\pi_1=\cD^-_{n_1}$ with $n_1\in\frac{1}{2}\Z_{\ge 3}\setminus\Z$, we have 
\[
\dim \cJ_{\rho_{m,\pi_1},\pi_{\lambda}}(\tl^*)^{00}=0.
\]
\end{enumerate}
\end{thm}
We can review Hirano's results for $\pi_{\lambda}$ with $\lambda\in\Xi_{III}$ from \cite[Theorems 7.7,~7.8]{Hi-1}). 
However, let us recall that the (generalized) Schmid operators for $\pi_{\lambda}$ and its contragredient $\pi_{\lambda}^*$ are conjugate to each other by $\delta^{-1}\xi$ as we have remarked in the proof of Theorem \ref{LargeDSWhittaker} (2). We can reproduce \cite[Theorems 7.7,~7.8]{Hi-1} by the following proposition.
\begin{prop}\label{explicit-FJ-III}
For each large discrete series representation $\pi_{\lambda}$ with $\lambda=(\lambda_1,\lambda_2)\in\Xi_{III}$ the Fourier-Jacobi type spherical functions $W^*$ is given by
\[
W^*(g)=W(\delta g\delta^{-1}\xi)\quad(g\in G),
\]
up to constant multiples, with the Fourier-Jacobi type spherical function $W$ for the large discrete series $\pi_{\lambda^*}$ with $\lambda^*=(-\lambda_2,-\lambda_1)\in\Xi_{II}$. 
\end{prop}
\subsection*{(III)~The case of $P_J$-principal series representations}
All the theorems for this case are included in Hirano \cite{Hi-2}. Though the irreducibility of $P_J$-principal series representations are assumed for the theorems there we can remove such assumption since the proofs of them in \cite{Hi-2} do not use the irreducibility to obtain explicit Fourier-Jacobi type spherical functions. 

Throughout the explanation of this case we denote simply by $\pi$ the $P_J$-principal series representation  $\Ind_{P_J}^G(1_{N_J}\otimes\nu_z\otimes\sigma)$ of $G$, with $\sigma=(\cD_n^{\pm},\epsilon)$ and $\nu_z$ as in Section \ref{P_JPS-rep}. 
We remind the readers that every $P_J$-principal series representation has the corner $K$-type, whose dimension is one or two~(cf.~Section \ref{P_JPS-rep}). We divide Hirano's formula of this case into two cases in terms of the dimensions of  the corner $K$-types.
\subsection*{(1)~The case of the one dimensional corner $K$-type}
We first review Hirano's result \cite[Theorems 5.4,~5.6]{Hi-2} for $\pi$ given by $\sigma=(\cD^+_n,\epsilon)$ with $n\in\Z_{\ge 1}$, $\epsilon(\diag(1,-1,1,-1))=(-1)^n$. The corner $K$-type of $\pi$ then has the highest weight $\Lambda=(n,n)$. 
Let $j_0\in J$ be the minimum index of $J$ such that $l(j_0,0)\in L$, which is given by
\begin{equation}\label{Index-J}
\begin{cases}
\frac{1}{2}&\left(\pi_1=
\begin{cases}
\cP_{\tau}^s,~\cC_{\tau}^s,~\tau+\frac{1}{2}\equiv n\mod 2,\\
\cD_{n_1}^+,~n_1+\frac{1}{2}\equiv n\mod 2, n_1+\frac{1}{2}<n,
\end{cases}\right),\\
\frac{3}{2}&\left(\pi_1=
\begin{cases}
\cP_{\tau}^s,~\cC_{\tau}^s,~\tau-\frac{1}{2}\equiv n\mod 2,\\
\cD_{n_1}^+,~n_1+\frac{3}{2}\equiv n\mod 2, n_1+\frac{3}{2}<n,
\end{cases}\right),\\
n+n_1&(\pi_1=\cD_{n_1}^-).
\end{cases}
\end{equation}
\begin{thm}\label{explicit-FJ-PJ-I}
Let $\pi$ and $j_0$ be as above and suppose that $m>0$. 
\begin{enumerate}
\item If $\pi_1=\cD_{n_1}^+~(n_1\in\frac{1}{2}\Z_{\ge 1}\setminus\Z,~n_1>n-\frac{1}{2})$ then 
\[
\dim\cJ_{\rho_{m,\pi_1},\pi}(\tl^*)^{00}=0.
\]
\item If $\pi_1=\cD_{n_1}^+~(n_1\in\frac{1}{2}\Z_{\ge 1}\setminus\Z,~n_1\le n-\frac{1}{2})$ then 
\[
\dim\cJ_{\rho_{m,\pi_1},\pi}(\tl^*)^{00}\le 1.
\]
\item If $\pi_1\not=\cD_{n_1}^+$ and $j_0=\frac{1}{2}$ or $\frac{3}{2}$, that is, $\pi_1=\cP_s^{\tau}~(\tau=\pm\frac{1}{2},~s\in\sqrt{-1}\R),~\cC_s^{\tau}~(\tau=\pm\frac{1}{2},~0<s<\frac{1}{2})$ or $\cD^-_{\frac{1}{2}}$ with $n=1$, then 
\[
\dim\cJ_{\rho_{m,\pi_1},\pi}(\tl^*)^{00}\le 1.
\]
\item If $j_0=n+n_1>\frac{3}{2}$, that is, $\pi_1=\cD_{n_1}^-$ and $(n,n_1)\not=(1,\frac{1}{2})$, then 
\[
\dim\cJ_{\rho_{m,\pi_1},\pi}(\tl^*)^{00}=0.
\]
\end{enumerate}
If $\cJ_{\rho_{m,\pi_1},\pi}(\tl^*)^{00}\not=\{0\}$ the coefficient function $c_{j_0,0}(a_J)$ for the restriction of $W\in\cJ_{\rho_{m,\pi_1},\pi}(\tl^*)^{00}$ to $A_J$ is given by
\[
e^{\frac{1}{2}x}G_{2,3}^{3,0}\left(x\left|
\begin{array}{c}
\dfrac{z_0+2+j_0}{2},\dfrac{-z_0+2+j_0}{2}\\
\dfrac{n+1}{2},\dfrac{z+2}{2},\dfrac{-z+2}{2}
\end{array}
\right.\right),
\]
up to constant multiple. Here $z_0$ means $s$~(respectively~$n_1-1$) for $\pi_1=\cP_s^{\tau}$ or $\cC_s^{\tau}$~(respectively~$\cD_{n_1}^{\pm}$). The other coefficients are given inductively from $c_{j_0,0}$ by the recurrence relation \cite[(5.2) in Proposition 5.2]{Hi-2}.
\end{thm}
\begin{thm}\label{explicit-FJ-PJ-II}
Let $\pi$ be as in Theorem \ref{explicit-FJ-PJ-I} and suppose that $m<0$.
\begin{enumerate}
\item If $\pi_1=\cD_{n_1}^+~(n_1\in\frac{1}{2}\Z\setminus\Z,~n_1>n-\frac{1}{2})$, then 
\[
\dim\cJ_{\rho_{m,\pi_1},\pi}(\tl^*)^{00}\le 1.
\]
Suppose that $\cJ_{\rho_{m,\pi_1},\pi}(\tl^*)^{00}\not=\{0\}$. For $j'_0=n-n_1$ the coefficient function $c_{j'_0,0}(a_J)$ for the restriction of $W\in\cJ_{\rho,\pi}(\tl^*)^{00}$ to $A_J$ is
\[
e^{\frac{1}{2}x'}G_{2,3}^{3,0}\left(x'\left|
\begin{array}{c}
\dfrac{n_1+\frac{3}{2}}{2},\dfrac{n_1+\frac{5}{2}}{2}\\
\dfrac{n+1}{2},\dfrac{z+2}{2},\dfrac{-z+2}{2}
\end{array}
\right.\right),
\]
up to constant multiples. The other coefficients are deduced inductively from $c_{j'_0,0}$ by the recurrence relation \cite[(5.2) in Proposition 5.2]{Hi-2}.
\item If $\cD^+_{n_1}~(n_1\in\frac{1}{2}\Z_{\ge 1}\setminus\Z,~n_1\le n-\frac{1}{2})$, $\cP_s^{\tau}~(\tau=\pm\frac{1}{2},~s\in\sqrt{-1}\R)$ or $\cC_s^{\tau}~(\tau=\pm\frac{1}{2},~0<s<\frac{1}{2})$, then 
\[
\dim\cJ_{\rho_{m,\pi_1},\pi}(\tl^*)^{00}=0.
\]
\item If $\pi_1=\cD_{n_1}^-~(n_1\in\frac{1}{2}\Z_{\ge 1}\setminus\Z)$, then 
\[
\dim\cJ_{\rho_{m,\pi_1},\pi}(\tl^*)^{00}=0.
\] 
\end{enumerate}
\end{thm}
We next suppose that $\pi$ is another $P_J$-principal series representation with one dimensional corner $K$-type, namely $\pi$ is given by $\sigma=(\cD_n^-,\epsilon)$ with $n\in\Z_{\ge 1}$ and $\epsilon(\diag(1,-1,1,-1))=(-1)^n$. Its corner $K$-type then has the highest weight $(-n,-n)$. 
We can do the reduction of the problem similar to the case of large discrete series representations. 
Following the proof of Theorem \ref{PJWhittaker}, we see that the shift operators for $P_J$-principal series representations with $\sigma=(\cD_n^-,\epsilon)$ and $(\cD_n^+,\epsilon)$ are conjugate to each other by $\delta^{-1}\xi$ and that the Casimir operator is invariant under such conjugation. We can reproduce Hirano's result \cite[Theorems 5.7,~5.8]{Hi-2} for this case by the following proposition.
\begin{prop}
For every $P_J$-principal series representation above with $\sigma=(\cD_n^-,\epsilon)$ the Fourier-Jacobi type spherical function $W^*$ is given by 
\[
W^*(g)=W(\delta g\delta^{-1}\xi)\quad(g\in G),
\]
up to constant multiples, 
with the Fourier-Jacobi type spherical function $W$ for a $P_J$-principal series representation with one-dimensional corner $K$-type and $\sigma=(\cD_n^+,\epsilon)$. 
\end{prop}
\subsection*{(2)~The case of the two dimensional corner $K$-type}
For this case we begin with the $P_J$-principal series representation $\pi$ given by $\sigma=(\cD^+_{n},\epsilon)$ with $n\in\Z_{\ge 1}$ and $\epsilon(\diag(1,-1,1,-1))=-(-1)^n$. Then the corner $K$-type of $\pi$ has the highest weight $\Lambda=(n,n-1)$~(cf.~Section \ref{P_JPS-rep}). The results \cite[Theorems 6.3,~6.5]{Hi-2} by Hirano for this case are stated as follows:
\begin{thm}\label{explicit-FJ-PJ-III}
Let $\pi$ be as above and $j_0$ be as in (\ref{Index-J}). Suppose $m>0$. 
\begin{enumerate}
\item If $\pi_1=\cD_{n_1}^+~(n_1\in\frac{1}{2}\Z_{\ge 1}\setminus\Z,~n_1>n-\frac{1}{2})$ then 
\[
\dim\cJ_{\rho_{m,\pi_1},\pi}(\tl^*)^{00}=0.
\]
\item If $\pi_1=\cD_{n_1}^+~(n_1\in\frac{1}{2}\Z_{\ge 1}\setminus\Z,~n_1\le n-\frac{1}{2})$ then 
\[
\dim\cJ_{\rho_{m,\pi_1},\pi}(\tl^*)^{00}\le 1.
\]
\item If $\pi_1\not=\cD_{n_1}^+$ and $j_0=\frac{1}{2}$ or $\frac{3}{2}$, that is, $\pi_1=\cP_s^{\tau}~(\tau=\pm\frac{1}{2},~s\in\sqrt{-1}\R),~\cC_s^{\tau}~(\tau=\pm\frac{1}{2},~0<s<\frac{1}{2})$ or $\cD^-_{\frac{1}{2}}$ with $n=1$, then 
\[
\dim\cJ_{\rho_{m,\pi_1},\pi}(\tl^*)^{00}\le 1.
\]
\item If $j_0=n+n_1>\frac{3}{2}$, that is, $\pi_1=\cD_{n_1}^-$ and $(n,n_1)\not=(1,\frac{1}{2})$, then 
\[
\dim\cJ_{\rho_{m,\pi_1},\pi}(\tl^*)^{00}=0.
\]
\end{enumerate}
If $\cJ_{\rho_{m,\pi_1},\pi}(\tl^*)^{00}\not=\{0\}$ the coefficient function $c_{\frac{1}{2},k_0}(a_J)$ for the restriction of  $W\in\cJ_{\rho_{m,\pi_1},\pi}(\tl^*)^{00}$ to $A_J$ with $k_0:=\frac{3}{2}-j_0$ is given by
\[
e^{\frac{1}{2}x}G_{2,3}^{3,0}\left(x\left|
\begin{array}{c}
\dfrac{z_0+\frac{7}{2}}{2},\dfrac{-z_0+\frac{7}{2}}{2}\\
\dfrac{n+1+k_0}{2},\dfrac{z+3-k_0}{2},\dfrac{-z+3-k_0}{2}
\end{array}
\right.\right),
\]
up to constant multiple. Here $z_0$ means $s$~(respectively~$n_1-1$) for $\pi_1=\cP_s^{\tau}$ or $\cC_s^{\tau}$~(respectively~$\cD_{n_1}^{\pm}$). The other coefficients are given inductively from $c_{\frac{1}{2},k_0}$ by the recurrence relations \cite[(6.1),~(6.2) in Proposition 6.1]{Hi-2}.
\end{thm}
\begin{thm}\label{explicit-FJ-PJ-IV}
Let $\pi$ be as in Theorem \ref{explicit-FJ-PJ-I} and suppose that $m<0$.
\begin{enumerate}
\item If $\pi_1=\cD_{n_1}^+~(n_1\in\frac{1}{2}\Z\setminus\Z,~n_1>n-\frac{1}{2})$, then 
\[
\dim\cJ_{\rho_{m,\pi_1},\pi}(\tl^*)^{00}\le 1.
\]
Suppose that $\cJ_{\rho_{m,\pi_1},\pi}(\tl^*)^{00}\not=\{0\}$. For $j'_0=n-n_1$ the coefficient function $c_{j'_0,1}(a_J)$ for the restriction of $W\in\cJ_{\rho_{m,\pi_1},\pi}(\tl^*)^{00}$ to $A_J$ is
\[
e^{\frac{1}{2}x'}G_{2,3}^{3,0}\left(x'\left|
\begin{array}{c}
\dfrac{n_1+\frac{3}{2}}{2},\dfrac{n_1+\frac{5}{2}}{2}\\
\dfrac{n+2}{2},\dfrac{z+2}{2},\dfrac{-z+2}{2}
\end{array}
\right.\right),
\]
up to constant multiples. The other coefficients are deduced inductively from $c_{j'_0,1}$ by the recurrence relations \cite[(6.1),~(6.2) in Proposition 6.1]{Hi-2}.
\item If $\cD^+_{n_1}~(n_1\in\frac{1}{2}\Z_{\ge 1}\setminus\Z,~n_1\le n-\frac{1}{2})$, $\cP_s^{\tau}~(\tau=\pm\frac{1}{2},~s\in\sqrt{-1}\R)$ or $\cC_s^{\tau}~(\tau=\pm\frac{1}{2},~0<s<\frac{1}{2})$, then 
\[
\dim\cJ_{\rho_{m,\pi_1},\pi}(\tl^*)^{00}=0.
\]
\item If $\pi_1=\cD_{n_1}^-~(n_1\in\frac{1}{2}\Z_{\ge 1}\setminus\Z)$, then 
\[
\dim\cJ_{\rho_{m,\pi_1},\pi}(\tl^*)^{00}=0.
\] 
\end{enumerate}
\end{thm}
We next suppose that $\pi$ is given by $\sigma=(\cD_{n}^-,\epsilon)$ with $\epsilon(\diag(1,-1,1,-1))=-(-1)^n$ and $n\in\Z_{\ge 1}$. The corner $K$ type $\tl$ of $\pi$ has the highest weight $\Lambda=(-n+1,n)$. 
As in the previous case, we do the similar reduction of the problem by the following proposition to reproduce Hirano's result \cite[Theorems 6.6,~6.7]{Hi-2}.
\begin{prop}
For every $P_J$-principal series representation above with $\sigma=(\cD_n^-,\epsilon)$ the Fourier-Jacobi type spherical function $W^*$ is given by 
\[
W^*(g)=W(\delta g\delta^{-1}\xi)\quad(g\in G),
\]
up to constant multiples, with the Fourier-Jacobi type spherical function $W$ for a $P_J$-principal series representation $\pi$ with two-dimensional corner $K$-type and $\sigma=(\cD_n^+,\epsilon)$. 
\end{prop}
\subsection*{(IV)~The case of the principal series representations~(induced from the minimal parabolic subgroup)}
All the theorems of this case are included in Hirano \cite{Hi-3}. As we have seen in Section \ref{PS-rep} we can divide the principal series representations $\pi$ into even representations and odd representations. 

We first think of the case of even principal series representations with the multiplicity one $K$-type $\tl$ with $\Lambda=(n,n)$~(cf.~Section \ref{PS-rep}). 
Let $j_0\in J$ be the minimum index of $J$ such that $l(j,0)\in L$ given by
\begin{equation}\label{Index-J-2}
\begin{cases}
\dfrac{1}{2}&\left(\pi_1=
\begin{cases}
\cP_{\tau}^s,~\cC_{\tau}^s,~\tau+\frac{1}{2}\equiv n\mod 2,\\
\cD_{n_1}^-,~n_1-\frac{1}{2}\equiv -n\mod 2, n_1\le-n+\dfrac{1}{2},
\end{cases}\right),\\
\dfrac{3}{2}&\left(\pi_1=
\begin{cases}
\cP_{\tau}^s,~\cC_{\tau}^s,~\tau-\frac{1}{2}\equiv n\mod 2,\\
\cD_{n_1}^-,~n_1-\dfrac{3}{2}\equiv -n\mod 2, n_1\le-n+\dfrac{3}{2},
\end{cases}\right),\\
n+n_1&(\pi_1=\cD_{n_1}^-,~n_1>-n+\dfrac{3}{2}).
\end{cases}
\end{equation}
From \cite[Theorem 4.4]{Hi-3} we review the explicit formula of this case, for which we do not have to assume the irreducibility of the principal series since the irreducibility is not used to obtain the formula.
\begin{thm}\label{explicit-FJ-P-I}
Let $\pi$ be an even principal series representation and $\tl$ be the multiplicity one $K$-type $\tl$ above. 
If $\pi_1=\cD_{n_1}^{{\rm sign}(m)}$, $\dim\cJ_{\rho_{m,\pi_1},\pi}(\tl^*)^{00}=0$. Otherwise $\dim\cJ_{\rho_{m,\pi_1},\pi}(\tl^*)^{00}\le 1$.
 
When $m>0$ let $j_0$ be as in (\ref{Index-J-2}). If $\cJ_{\rho_{m,\pi_1},\pi}(\tl^*)^{00}\not=\{0\}$ we have the following:
\begin{enumerate}
\item If $j_0=\dfrac{1}{2}$ or $\dfrac{3}{2}$, i.e. $\pi_1=\cP_{s}^{\tau},~\cC_s^{\tau}$ or $\cD_{n_1}^-$ with $n_1\le -n+\dfrac{3}{2}$, then
\[
c_{j_0,0}(a_J)=e^{\frac{1}{2}x}G_{3,4}^{4,0}\left(x\left|
\begin{array}{c}
\dfrac{z_0+2+j_0}{2},\dfrac{-z_0+2+j_0}{2},\dfrac{-n+3}{2}\\
\dfrac{z_1+2}{2},\dfrac{-z_1+2}{2},\dfrac{z_2+2}{2},\dfrac{-z_2+2}{2}
\end{array}\right.\right),
\]
up to constant multiples. Here $z_0$ means $s$~(respectively~$n_1-1$) for $\pi_1=\cP_s^{\tau}$ or $\cC_s^{\tau}$~(respectively~$\cD_{n_1}^-$).
\item If $j_0=n+n_1$, i.e. $\pi_1=\cD_{n_1}^-$ with $n_1>-n+\dfrac{3}{2}$, then 
\[
c_{j_0,0}(a_J)=e^{\frac{1}{2}x}G_{3,4}^{4,0}\left(x\left|
\begin{array}{c}
\dfrac{n_1+\frac{5}{2}}{2},\dfrac{n_1+\frac{3}{2}}{2},\dfrac{n+3}{2}\\
\dfrac{z_1+2}{2},\dfrac{-z_1+2}{2},\dfrac{z_2+2}{2},\dfrac{-z_2+2}{2}
\end{array}\right.\right),
\]
up to constant multiples. 
\end{enumerate}
Let $m<0$. We put $j'_0:={\rm max}\{j\in J\mid l(j,0)\in L\}$. 
If $\cJ_{\rho_{m,\pi_1},\pi}(\tl^*)^{00}\not=\{0\}$ we have the following:
\begin{enumerate}
\item If $j'_0=-\dfrac{1}{2}$ or $-\dfrac{3}{2}$ i.e. $\pi_1=\cP_s^{\tau}$, $\cC_s^{\tau}$ or $\cD_{n_1}^+$ with $n_1\le n+\dfrac{3}{2}$, then
\[
c_{j'_0,0}(a_J)=e^{\frac{1}{2}x'}G_{3,4}^{4,0}\left(x'\left|
\begin{array}{c}
\dfrac{z_0+2-j'_0}{2},\dfrac{-z_0+2-j'_0}{2},\dfrac{n+3}{2}\\
\dfrac{z_1+2}{2},\dfrac{-z_1+2}{2},\dfrac{z_2+2}{2},\dfrac{-z_2+2}{2}
\end{array}\right.\right),
\]
up to constant multiples. Here $z_0$ means $s$~(respectively~$n_1-1$) for $\pi_1=\cP_s^{\tau}$ or $\cC_s^{\tau}$~(respectively~$\cD_{n_1}^+$).
\item If $j'_0=n-n_1$, i.e. $\pi_1=\cD_{n_1}^+$ with $n_1>n+\dfrac{3}{2}$, then
\[
c_{j'_0,0}(a_J)=e^{\frac{1}{2}x'}G_{3,4}^{4,0}\left(x'\left|
\begin{array}{c}
\dfrac{n_1+\frac{5}{2}}{2},\dfrac{n_1+\frac{3}{2}}{2},\dfrac{-n+3}{2}\\
\dfrac{z_1+2}{2},\dfrac{-z_1+2}{2},\dfrac{z_2+2}{2},\dfrac{-z_2+2}{2}
\end{array}\right.\right)
\]
up to constant multiples.
\end{enumerate}
For both of $m>0$ and $m<0$ the other coefficients are obtained inductively from the above coefficients by the recurrence relation \cite[(4.2) in Lemma 4.2]{Hi-3}.
\end{thm}
We next think of the case of the odd principal series representations. For this case we introduce the notation 
$(\tilde{z},\tilde{z}'):=(z_1,z_2)$ when $\sigma=(1,-1)$~(respectively~$(-1,1)$) and $n$ is even~(respectively~odd), 
and let $(\tilde{z},\tilde{z}'):=(z_2,z_1)$ otherwise~(for this notation see \cite[Proposition 4.1]{Hi-3}). 
As in Theorem \ref{explicit-FJ-P-I} we do not have to assume the irreducibility of the principal series to state the following theorem~(cf.~\cite[Theorem 4.5]{Hi-3}).
\begin{thm}\label{explicit-FJ-P-II}
Let $\pi$ be an odd principal series representation and $\tl$ be the multiplicity one $K$-type $\tl$ with $\Lambda=(n,n-1)$~(cf.~Section \ref{PS-rep}), and let $\rho:=\rho_{\pi_1,m}$. 
If $\pi_1=\cD_{n_1}^{{\rm sign}(m)}$, $\dim\cJ_{\rho_{m,\pi_1},\pi}(\tl^*)^{00}=0$. Otherwise $\dim\cJ_{\rho_{m,\pi_1},\pi}(\tl^*)^{00}\le 1$.
 
When $m>0$ let $j_0$ be as in (\ref{Index-J-2}). If $\cJ_{\rho_{m,\pi_1},\pi}(\tl^*)^{00}\not=\{0\}$ we have the following::
\begin{enumerate}
\item If $j_0=\dfrac{1}{2}$ or $\dfrac{3}{2}$, i.e. $\pi_1=\cP_{s}^{\tau},~\cC_s^{\tau}$ or $\cD_{n_1}^-$ with $n_1\le -n+\dfrac{3}{2}$, then
\[
c_{\frac{1}{2},k_0}(a_J)=e^{\frac{1}{2}x}G_{3,4}^{4,0}\left(x\left|
\begin{array}{c}
\dfrac{z_0+\frac{7}{2}}{2},\dfrac{-z_0+\frac{7}{2}}{2},\dfrac{-n+3+k_0}{2}\\
\dfrac{\tilde{z}+2+k_0}{2},\dfrac{-\tilde{z}+2+k_0}{2},\dfrac{\tilde{z}'+3-k_0}{2},\dfrac{-\tilde{z}'+3-k_0}{2}
\end{array}\right.\right)
\]
with $k_0=\frac{3}{2}-j_0$, up to constant multiples. Here $z_0$ means $s$~(respectively~$n_1-1$) for $\pi_1=\cP_s^{\tau}$ or $\cC_s^{\tau}$~(respectively~$\cD_{n_1}^-$).
\item If $j_0=n+n_1$, i.e. $\pi_1=\cD_{n_1}^-$ with $n_1>-n+\dfrac{3}{2}$, then 
\[
c_{j_0-1,0}(a_J)=e^{\frac{1}{2}x}G_{3,4}^{4,0}\left(x\left|
\begin{array}{c}
\dfrac{n_1+\frac{5}{2}}{2},\dfrac{n_1+\frac{3}{2}}{2},\dfrac{n+3}{2}\\
\dfrac{\tilde{z}+2}{2},\dfrac{-\tilde{z}+2}{2},\dfrac{\tilde{z}'+3}{2},\dfrac{-\tilde{z}'+3}{2}
\end{array}\right.\right),
\]
up to constant multiples. 
\end{enumerate}
Let $m<0$. We put $j'_0:={\rm max}\{j\in J\mid l(j,0)\in L\}$. If $\cJ_{\rho_{m,\pi_1},\pi}(\tl^*)^{00}\not=\{0\}$ we have the following:
\begin{enumerate}
\item If $j'_0=-\dfrac{1}{2}$ or $-\dfrac{3}{2}$ i.e. $\pi_1=\cP_s^{\tau}$, $\cC_s^{\tau}$ or $\cD_{n_1}^+$ with $n_1\le n+\dfrac{3}{2}$, then
\[
c_{-\frac{1}{2},k'_0}(a_J)=e^{\frac{1}{2}x'}G_{3,4}^{4,0}\left(x'\left|
\begin{array}{c}
\dfrac{z_0+\frac{7}{2}}{2},\dfrac{-z_0+\frac{7}{2}}{2},\dfrac{n+3-k'_0}{2}\\
\dfrac{\tilde{z}+2+k'_0}{2},\dfrac{-\tilde{z}+2+k'_0}{2},\dfrac{\tilde{z}'+3-k'_0}{2},\dfrac{-\tilde{z}'+3-k'_0}{2}
\end{array}\right.\right)
\]
with $k'_0=-\frac{1}{2}-j'_0$, up to constant multiples. Here $z_0$ means $s$~(respectively~$n_1-1$) for $\pi_1=\cP_s^{\tau}$ or $\cC_s^{\tau}$~(respectively~$\cD_{n_1}^+$).
\item If $j'_0=n-n_1$, i.e. $\pi_1=\cD_{n_1}^+$ with $n_1>n+\dfrac{3}{2}$, then
\[
c_{j'_0+1,1}(a_J)=e^{\frac{1}{2}x'}G_{3,4}^{4,0}\left(x'\left|
\begin{array}{c}
\dfrac{n_1+\frac{5}{2}}{2},\dfrac{n_1+\frac{3}{2}}{2},\dfrac{-n+4}{2}\\
\dfrac{\tilde{z}+3}{2},\dfrac{-\tilde{z}+3}{2},\dfrac{\tilde{z}'+2}{2},\dfrac{-\tilde{z}'+2}{2}
\end{array}\right.\right)
\]
up to constant multiples.
\end{enumerate}
For both of $m>0$ and $m<0$ the other coefficients are obtained inductively from the above coefficients by the recurrence relation in \cite[Lemma 4.2]{Hi-3}.
\end{thm}
\newpage
\section{Eichler-Zagier correspondence in representation theoretic formulation}\label{EZcorrespondence}
We recall that we have introduced the discrete subgroup $G_J(\Z)$ of $G_J$ given by $G_J(\Z)=N_J(\Z)\rtimes SL_2(\Z)$~(cf.~Section \ref{Groups}).
Following \cite[Section 4.2]{Be-Sc} we introduce the cuspidal space for $G_J$ as follows:
\begin{defn}\label{cuspidal-subgp}
(1)~(cf.~\cite[Definition 4.2.1]{Be-Sc})~A subgroup $N^*_J$ of $G_J$ is called horo-spherical if $N^*_J$ is $G_J$-conjugate to $V_J:=\left\{\left.
\begin{pmatrix}
1 & 0 & 0 & u_2\\
0 & 1 & u_2 & u_3\\
0 & 0 & 1 & 0\\
0 & 0 & 0 & 1
\end{pmatrix}~\right|~u_2,~u_3\in\R\right\}$. A horo-spherical subgroup $N^*_J$ is defined to be cuspidal for $G_J(\Z)$ if $N^*_J/N^*_J\cap G_J(\Z)$ is compact.\\
(2)~(cf.~\cite[Definition 4.2.5]{Be-Sc})~For a cuspidal subgroup $N_J^*$, which is in bijection with  $\R^2$, we denote by $dn$ the measure of $N_J^*$ induced by the Euclidean measure of $\R^2$. The cuspidal space $\cH^0$ of $L^2(G_J(\Z)\backslash G_J)$ is defined as
\[
\cH^0:=\left\{\phi\in L^2(G_J(\Z)\backslash G_J)~\left|~\displaystyle\int_{N^*_J\cap G_J(\Z)\backslash N^*_J}\phi(ng)dn=0~
\begin{array}{c}
\text{for almost all $g\in G_J$ and}\\
\text{any cuspidal subgroup $N^*_J$}
\end{array}\right.\right\}.
\]
\end{defn}
For each $m\in\Z$ we introduce 
\[
\cH_m:=\{\phi\in L^2(G_J(\Z)\backslash G_J)\mid \phi(n(0,z,0)g)=\e(mz)\phi(g)~\forall(z,g)\in\R\times G_J\},~\cH_m^0:=\cH_m\cap\cH^0.
\]
On the cuspidal space above we cite the following fact:
\begin{prop}\label{Spectral-Jacobi}
(1)~(cf.~\cite[Theorem 4.3.1]{Be-Sc})~The representation of $G_J$ on $\cH^0_m$ defined by the right translation is completely reducible, and each irreducible component occurs in $\cH^0_m$ with a finite multiplicity. The same assertion holds for $\cH^0$ since $\cH^0=\oplus_{m\in\Z}\cH^0_m$.\\
(2)~(cf.~\cite[Proposition 4.2.6 (iii)]{Be-Sc})~Let $V_J$ be as in Definition \ref{cuspidal-subgp} and for $m_2,~m_3\in\R$,  we put 
\[
\psi^{m_2,m_3}(\begin{pmatrix}
1 & 0 & 0 & u_2\\
0 & 1 & u_2 & u_3\\
0 & 0 & 1 & 0\\
0 & 0 & 0 & 1
\end{pmatrix}):=\e(m_2u_2+m_3u_3)\quad((u_2,u_3)\in\R^2).
\]
Then $\phi\in\cH_m$ belongs to $\cH^0_m$ if and only if, for almost all $g_0\in G_J$, we have that 
\[
\displaystyle\int_{V_J\cap G_J(\Z)\backslash V_J}\phi(ng_0)\overline{\psi^{n,s}(n)}dn=0
\]
for $(s,r)\in\Z$ such that $4ms-r^2=0$.
\end{prop}

Now let us recall that, for $m\in\R\setminus\{0\}$ and irreducible genuine representations $\pi_1$ of $\widetilde{SL}_2(\R)$, 
we have introduced unitary representations $\rho_{m.\pi_1}$ of $G_J$, which exhaust the unitary dual of $G_J$~(cf.~Proposition \ref{UnitaryDual-G_J}) except for the case of Proposition \ref{UnitaryDual-G_J} (2). 
\begin{defn}
Let $m\in\Z\setminus\{0\}$ and $\pi_1$ be an irreducible genuine representation of $\widetilde{SL}_2(\R)$. Furthermore let $\rho_{m,\pi_1}$ be as above. We define the space of Jacobi forms of index $m$ and type $\pi_1$ as
\[
\rm{Hom}_{G_J}(\rho_{m,\pi_1},\cH^0_m).
\]
\end{defn}
As an immediate consequence from Proposition \ref{Spectral-Jacobi} (1) we have the following: 
\begin{prop}
For a non-zero $m\in\Z$ the space $\rm{Hom}_{G_J}(\rho_{m,\pi_1},\cH^0_m)$ is finite dimensional.
\end{prop}
We next consider the space of the following intertwining operators
\[
\Phi_m:=\rm{Hom}_{G_J}(\nu_m,L^2(N_J(\Z)\backslash N_J))
\]
for a non-zero integer $m\in\Z$. We recall that the representation space $\cU_m$ of $\nu_m$ is identified with $L^2(\R)$.  From a general theory by Corwin-Greenleaf \cite[Theorem 5.1,~Sections 5,~6]{Co-G-1} we can provide an explicit description of  a basis  of $\Phi_m$ as follows:
\begin{prop}\label{Basis-theta}
The space $\Phi_m$ has $\{\theta_{\alpha}\}_{\alpha\in\Z/2m\Z}$ as a basis, where
\[
\theta_{\alpha}(h)(n(u_0,u_1,u_2)):=\sum_{k\in\Z}\e(mu_1+(2km+\alpha)u_2)h(u_0+k+\frac{\alpha}{2m})\quad(h\in L^2(\R)\simeq\cU_m).
\]
Namely we have $\dim\Phi_m=2|m|$. To be precise, this is the right translation of the theta series defined in \cite{Co-G-1} by $n(\frac{\alpha}{2m},0,0)$. 
\end{prop}

Now we recall that $\widetilde{SL}_2(\R)$ is realized as the group consisting of pairs $(M,\phi)$ 
with $M=
\begin{pmatrix}
a & b\\
c & d
\end{pmatrix}\in SL_2(\R)$ and a holomorphic function $\phi$ on the complex upper half plane ${\frak h}$ satisfying $\phi^2(\tau)=c\tau+d~(\tau\in{\frak h})$. For $M=
\begin{pmatrix}
a & b\\
c & d
\end{pmatrix}\in SL_2(\R)$ we put $\tilde{M}=
(\begin{pmatrix}
a & b\\
c & d
\end{pmatrix},\sqrt{c\tau+d})$, where we choose the principal branch to define the square root $\sqrt{z}$ of $z\in\C$. We define $\widetilde{SL}_2(\Z)$ as the inverse image of $SL_2(\Z)$ by the covering map $\widetilde{SL}_2(\R)\rightarrow SL_2(\R)$, which is the double cover of $SL_2(\Z)$. 

We now review that the Weil representation $\omega_m$ induces the $\widetilde{SL}_2(\Z)$-module structure of $\Phi_m$ defined by
\[
(\Omega_m(\delta)\cdot\theta)(h):=\theta(\omega_m(\delta)h)\quad(\delta\in\widetilde{SL}_2(\Z))
\]
for $(\theta,h)\in \Phi_m\times L^2(\R)$. 
From Borcherds \cite[p.505]{Bo}, Bruineir \cite[Proposition 1.1]{Br} and Shintani \cite[Proposition 1.6]{Shi} we know such module structure of $\Phi_m$ explicitly. 
Given a fixed $m\in\Z\setminus\{0\}$, let $\Z$ be equipped with the bilinear form $\Z\times\Z\ni(x,y)\mapsto 2mxy\in\Z$. The dual lattice of $\Z$ with respect to this bilinear form is $\displaystyle\frac{1}{2m}\Z$. We can explicitly describe the $\widetilde{SL}_2(\Z)$-module structure of $\Phi_m$, which is realized as a representation of $\widetilde{SL}_2(\Z)$ on the group algebra $\C[\displaystyle\frac{1}{2m}\Z/\Z]$~(cf.~\cite[p.505]{Bo},~\cite[Proposition 1.1]{Br}). 
\begin{prop}\label{TransformationTheta}
For $\gamma=
\begin{pmatrix}
a & b\\
c & d
\end{pmatrix}\in SL_2(\Z)$ let $\tilde{\gamma}$ be  as above. We have $\Omega_m(\tilde{\gamma})\theta_{\alpha}=\sum_{\beta\in\Z/2m\Z}c(\alpha,\beta)_{\gamma}\theta_{\beta}$, where 
$c(\alpha,\beta)_{\gamma}:=$
\[
\begin{cases}
\sqrt{i}^{-{\rm sgn}(m)(1-{\rm sgn}(d))}\delta_{\alpha,a\beta}\e(\displaystyle\frac{ab}{2}(2m)(\displaystyle\frac{\alpha}{2m})^2)&(c=0),\\
\displaystyle\frac{1}{\sqrt{2|cm|}}\sqrt{i}^{-{\rm sgn(m)}{\rm sgn}(c)}\sum_{r\in\Z/c\Z}\e(\displaystyle\frac{a(2m)(\displaystyle\frac{\alpha}{2m}+r)^2-2(2m)(\displaystyle\frac{\alpha}{2m}+r)(\frac{\beta}{2m})+d(2m)(\displaystyle\frac{\beta}{2m})^2}{2c})&(c\not=0),
\end{cases}
\]
where $\delta_{*,*}$ denotes the Kronecker delta.
\end{prop}
\begin{proof}
This is essentially the transformation formula of theta functions by Shintani \cite[Proposition 1.6]{Shi}. More precisely we modify Shintani's formulation by \cite[Proposition 1.1]{Br}~(see also \cite[p.505]{Bo}). For the proof we remark that Shintani's formula describes the transformation formula for $\theta_{\alpha}(h)(n(0,0,0))$ for $h\in L^2(\R)$. The formula still holds if we replace $h$ with $\nu_m(n(u_0,u_1,u_2))\cdot h$, for which note that $\theta_{\alpha}(\nu_m(n(u_0,u_1,u_2))\cdot h)(n(0,0,0))=\theta_{\alpha}(h)(n(u_0,u_1,u_2))$.
\end{proof}
We introduce the notion of $\Phi_m$-valued cusp forms with respect to $\widetilde{SL}_2(\Z)$.
\begin{defn}\label{Vector-valuedCuspfm}
For an irreducible genuine unitary representation $\pi_1$ of $\widetilde{SL_2}(\R)$ we define $\cS_{\pi_1}(\widetilde{SL}_2(\Z);\Phi_m)$ to be the space of $\Phi_m$-valued smooth functions $f$ on $\widetilde{SL_2}(\R)$ satisfying the followings:
\begin{enumerate}
\item $f(\delta bu)=\tau_{\rm min}(u)^{-1}\Omega_m(\delta)^{-1}f(b)$ for $(\delta,b,u)\in \widetilde{SL}_2(\Z)\times\widetilde{SL_2}(\R)\times \widetilde{SO_2}(\R)$, where recall that $\tau_{\rm min}$ denotes the minimal $\widetilde{SO_2}(\R)$-type of $\pi_1$~(cf.~Section \ref{RepSL2}).
\item each coefficient of $f$ is a cusp form with respect to $\widetilde{\Gamma(4m)}$~(i.e. left $\widetilde{\Gamma(4m)}$-invariant), and generates $\pi_1$ as a $(\g_1,\widetilde{SO_2}(\R))$-module, where recall that $\g_1$ denotes the Lie algebra of $\widetilde{SL}_2(\R)$ or $SL_2(\R)$~(cf.~Section \ref{RepSL2}).
\end{enumerate}
Here $\widetilde{\Gamma(4m)}$ denotes the pull-back of the principal congruence subgroup $\Gamma(4m)$ of $SL_2(\Z)$ of level $4m$ by the covering map $\widetilde{SL_2(\R)}\rightarrow SL_2(\R)$. 
\end{defn}
For this definition we note that $N=4|m|$ is the minimal positive integer such that $\displaystyle\frac{N\alpha^2}{4m}\in\Z$ for any $\alpha\in\Z$ and that the representation $\Omega_m$ factors through the double cover   of $SL_2(\Z/4m\Z)$. 
\begin{thm}\label{Eichler-ZagierCorresp}
We have an isomorphism
\[
S_{\pi_1}(\widetilde{SL}_2(\Z);\Phi_m)\simeq\rm{Hom}_{G_J}(\rho_{m,\pi_1},\cH^0_m).
\]
\end{thm}
\begin{proof} 
Let $w$ be the minimal $\widetilde{SO}(2)$-type vector of $\pi_1$ and recall that $\{u_j^m\}$ has denoted  a basis of the representation space of $\nu_m$~(cf.~Section \ref{UnitaryRepJacobi}). Firstly, given $\phi\in\rm{Hom}_{G_J}(\rho_{\pi_1,m},\cH^0_m)$, we explain how to associate $\phi$ with an element in $S_{\pi_1}(\widetilde{SL}_2(\Z);\Phi_m)$. After this we discuss the inverse direction. 
The former takes a much longer argument.  

For $\phi\in\rm{Hom}_{G_J}(\rho_{\pi_1,m},\cH^0_m)$ we regard  $\phi(w\otimes u^m_j)(ng)$ with a fixed $g\in SL_2(\R)$ as a function in $n\in N_J$. Consider an expansion of $\phi(w\otimes u^m_j)(ng)$ with respect to $\{\theta_{\alpha}(\omega_m(g)u_j^m)(n)\mid 1\le\alpha\le 2|m|\}$, where recall that $\Phi_m$ has a basis $\{\theta_{\alpha}\mid 1\le \alpha \le 2|m|\}$~(cf.~Proposition \ref{Basis-theta}). We have
\[
\phi(w\otimes u_j^m)(ng)=\phi(\pi_1(g)w\otimes\omega_m(g)u_j^m)(n)=\sum_{\alpha=1}^{2|m|}c_{\alpha}(g)\theta_{\alpha}(\omega_m(g)u_j^m)(n)
\]
with $C^{\infty}$-functions $c_{\alpha}(g)$ on $SL_2(\R)$. For this note that $c_{\alpha}$ is independent of $u_j^m$s since it depends only on $\pi_1(g)w$. 
Since the representation $\rho_{m,\pi_1}$ factors though the image of the covering map $\widetilde{SL_2}(\R)\ni\tilde{g}\mapsto g\in SL_2(\R)$,  $\phi(\pi_1(\tilde{g})w\otimes\omega_m(\tilde{g})u_j^m)=\phi(\pi_1(g)w\otimes\omega_m(g)u_j^m)$ is regarded as a well-defined function on $SL_2(\R)$ and also on $\widetilde{SL_2}(\R)$. 
\begin{lem}\label{Trans-Formula}
For $\gamma\in SL_2(\Z)$ we have
\[
\phi(w\otimes u_j^m)(\gamma ng)=\sum_{\alpha=1}^{2|m|}c_{\alpha}(\gamma g)(\Omega_m(\gamma)\theta_{\alpha})(\omega_m(g)u_j^m)(n),
\]
where we note that the representations of $\widetilde{SL}_2(\R)$ and $\widetilde{SL}_2(\Z)$ are reduced to those of $SL_2(\R)$ and $SL_2(\Z)$ respectively in view of such property of the representations of $G_J$. 
\end{lem} 
\begin{proof}
We start with checking
\[
\phi(w\otimes u_j^m)(ng)=\phi(\pi_1(g)w\otimes\omega_m(g)u_j^m)(n)=\sum_{\alpha=1}^{2|m|}c_{\alpha}(\gamma g)\theta_{\alpha}(\omega_m(\gamma g)u_j^m)(\gamma n\gamma^{-1})
\]
for $\gamma\in SL_2(\Z)$ in a straight forward manner by noting $\gamma ng=\gamma n\gamma^{-1}\cdot\gamma g$ and the left $SL_2(\Z)$-invariance of $\phi(w\otimes u_j^m)$. 
To continue the argument we remark that $\theta_{\alpha}$ is the intertwining operator between $\nu_m$ and  the right translation $r_{N_J}$ on $L_2(N_J(\Z)\backslash N_J)$. We then have
\begin{align*}
\sum_{\alpha=1}^{2|m|}c_{\alpha}(\gamma g)\theta_{\alpha}(\omega_m(\gamma g)u_j^m)(\gamma n\gamma^{-1})
&=\sum_{\alpha=1}^{2|m|}c_{\alpha}(\gamma g)(r_{N_J}(\gamma n\gamma^{-1})\theta_{\alpha})(\omega_m(\gamma g)u_j^m)(1)\\
&=\sum_{\alpha=1}^{2|m|}c_{\alpha}(\gamma g)\theta_{\alpha}(\nu_m(\gamma n\gamma^{-1})\omega_m(\gamma g)u_j^m)(1).
\end{align*}
We now note the following fundamental formula
\begin{equation}\label{WeilRep-basicrel}
\nu_m(n)\omega_m(g)=\omega_m(g)\nu_m(g^{-1}ng)\quad(n\in N_J)
\end{equation}
of the Weil representation $\omega_m$ to obtain
\begin{align*}
\sum_{\alpha=1}^{2|m|}c_{\alpha}(\gamma g)\theta_{\alpha}(\nu_m(\gamma n\gamma^{-1})\omega_m(\gamma g)u_j^m)(1)&=\sum_{\alpha=1}^{2|m|}c_{\alpha}(\gamma g)\theta_{\alpha}(\omega_m(\gamma g)\nu_m(g^{-1}ng)u_j^m)(1)\\
&=\sum_{\alpha=1}^{2|m|}c_{\alpha}(\gamma g)\theta_{\alpha}(\omega_m(\gamma)(\omega_m(g)\nu_m(g^{-1}ng))u_j^m)(1)\\
&=\sum_{\alpha=1}^{2|m|}c_{\alpha}(\gamma g)(\Omega_m(\gamma)\theta_{\alpha})(\nu_m(n)\omega_m(g))u_j^m)(1)\\
&=\sum_{\alpha=1}^{2|m|}c_{\alpha}(\gamma g)(\Omega_m(\gamma)\theta_{\alpha})(\omega_m(g))u_j^m)(n).
\end{align*}
Here note that we can reduce the representations $\omega_m$ and $\Omega_m$ to those of $SL_2(\R)$ and $SL_2(\Z)$ respectively as we have remarked in the assertion~(also just before the assertion). We are therefore done.
\end{proof}  
We see that each coefficient function $c_{\alpha}$ can be extended to a function on $\widetilde{SL_2}(\R)$ and that the left $SL_2(\Z)$-invariance of $\phi(w\otimes u_j^m)$ can be extended to such invariance with respect to $\widetilde{SL}_2(\Z)$. The latter means that 
\[
\sum_{\alpha=1}^{2|m|}c_{\alpha}(\tilde{\gamma} \tilde{g})\theta_{\alpha}=\Omega_m(\tilde{\gamma})^{-1}\sum_{\alpha=1}^{2|m|}c_{\alpha}(\tilde{g})\theta_{\alpha}
\]
for $(\tilde{\gamma},\tilde{g})\in \widetilde{SL}_2(\Z)\times\widetilde{SL_2}(\R)$. Then we see the left $\widetilde{\Gamma(4m)}$-invariance of $c_{\alpha}$, which is remarked after Definition \ref{Vector-valuedCuspfm}.
 Now note that each $c_{\alpha}$ depends only on $\pi_1(g)w$. 
Therefore it generates $\pi_1$ as a $({\frak g}_1,\tilde{SO}(2))$-module and satisfies the right $\widetilde{SO_2}(\R)$-equivariance with respect to $\tau_{\rm min}$.

What remains to see 
\[
\sum_{\alpha=1}^{2|m|}c_{\alpha}(g)\theta_{\alpha}\in S_{\pi_1}(\widetilde{SL}_2(\Z);\Phi_m)
\]
is to verify that each $c_{\alpha}$ is cuspidal, i.e. it has no constant term at each $\Gamma(4m)$-cusp. 

To this end we need a remark on the Fourier expansion of $c_{\alpha}(g)\theta_{\alpha}(\omega_m(g)u_j^m)(n)$. 
As is well known, $SL_2(\Z)$ gives a set of representatives of the cusps. In view of the transformation law with respect to $\Omega_m$~(cf.~Lemma \ref{Trans-Formula},~Proposition \ref{TransformationTheta}), the constant term of $\sum_{\alpha=1}^{2|m|}c_{\alpha}(\tilde{g})\theta_{\alpha}$ at a cusp $\delta\in SL_2(\Z)$ is related to the constant term at the cusp represented by $1$ via the unitary operator $\Omega_m(\tilde{\delta})^{-1}$. The problem is therefore reduced to the vanishing of the constant term at the cusp $1$. We now note that the Jacobi group $G_J$ includes 
$N_S$~(for the notation see Section \ref{Groups}) as a subgroup and we can then consider the Fourier expansion of $c_{\alpha}(g)\theta_{\alpha}(\omega_m(g)u_j^m)(n)$ along this group, where $N_S$ is known as the unipotent radical of the Siegel parabolic subgroup. This Fourier expansion is indexed by the set of the matrices  
\[
\bigcup_{l}\left\{\left.
\begin{pmatrix}
m & \frac{1}{2}(2km+\alpha)\\
\frac{1}{2}(2km+\alpha) & m(k+\frac{\alpha}{2m})^2+l
\end{pmatrix}~\right|~k\in\Z\right\}
\]
with $l\in \frac{1}{4m}\Z$ ranging over the indexes for the Fourier expansion of $c_{\alpha}$ at the cusp $1$. 
Reviewing Proposition \ref{Basis-theta} we can verify this by the following explicit formula for $\omega_m$:
\begin{equation}\label{Explicit-Weil}
.\omega_m(
\begin{pmatrix}
a & b\\
c & d
\end{pmatrix})u(x)
\begin{cases}
|c|^{1/2}\displaystyle\int_{\R}\e(m((ax+cy)(bx+dy)-xy))u(ax+cy)dy&(c\not=0),\\
|a|^{1/2}\e(mabx^2)u(ax)&(c=0),
\end{cases}
\end{equation}
for $u\in L^2(\R)$~(cf.~\cite[(1.4)]{Shi}), which holds modulo the group of complex numbers of norm 1. 
More precisely we use the second formula to verify. 
We further note that the matrix $\begin{pmatrix}
m & \frac{1}{2}(2km+\alpha)\\
\frac{1}{2}(2km+\alpha) & m(k+\frac{\alpha}{2m})^2+l
\end{pmatrix}$ has zero determinant if and only if $l=0$. 

Suppose now that the constant term of $\sum_{\alpha=1}^{2|m|}c_{\alpha}(\tilde{g})\theta_{\alpha}$ at $1$ is non-zero. This means that there is some $c_{\alpha}$ with a constant term, which implies that $\phi(w_0\otimes u_j^m)(ng)\not\in {\cal H}^0$ in view of Proposition \ref{Spectral-Jacobi} (2) and of the above remark on the Fourier expansion of $c_{\alpha}(g)\theta_{\alpha}(\omega_m(g)u_j^m)(n)$. This leads to a contradiction.

We now come to the remaining part of the proof, which is said to be the discussion of the inverse direction. Namely, given $\sum_{\alpha=1}^{2|m|}c_{\alpha}(g)\theta_{\alpha}\in S_{\pi_1}(\widetilde{SL}_2(\Z);\Phi_m)$ with cusp forms $c_{\alpha}$ with respect to $\widetilde{\Gamma(4m)}$,  we explain how to associate it with an element in $\Hom_{G_J}(\rho_{m,\pi_1},\cH_m^0)$.  We see $\sum_{\alpha=1}^{2|m|}c_{\alpha}(g)\theta_{\alpha}(\omega_m(g)u_j^m)(n)\in \cH^0_m$ in view of the vanishing of the constant terms for $c_{\alpha}$s. It is a  standard idea to consider the $(\g_1,\widetilde{SO_2(\R)})$-module generated by $c_{\alpha}$s in order to regard $c_{\alpha}$s as intertwining operators in 
\[
\Hom_{(\g_1,\widetilde{SO_2}(\R))}(\pi_1,L^2_{\rm cusp}(\widetilde{\Gamma(4m)}\backslash\widetilde{SL_2}(\R))
\]
with the cuspidal subspace $L^2_{\rm cusp}(\widetilde{\Gamma(4m)}\backslash\widetilde{SL_2}(\R))$ of $L^2(\widetilde{\Gamma(4m)}\backslash\widetilde{SL_2}(\R))$. This space of intertwining operators as $(\g_1,\widetilde{SO_2}(\R))$-modules can be viewed as those for unitary representations of $\widetilde{SL}_2(\R)$ since infinitesimally equivarent irreducible unitary representations are unitarily equivarent~(cf.~\cite[Theorem 3.4.11]{W-2}). We thereby see that $\sum_{\alpha=1}^{2|m|}c_{\alpha}(g)\theta_{\alpha}$ can be uniquely extended to an element in $\Hom_{G_J}(\rho_{m,\pi_1},\cH_m^0)$.

The map from $S_{\pi_1}(\widetilde{SL}_2(\Z);\Phi_m)$ to $\rm{Hom}_{G_J}(\rho_{\pi_1,m},\cH^0_m)$ implied by the argument just above is the inverse map of 
\[
\rm{Hom}_{G_J}(\rho_{\pi_1,m},\cH^0_m)\ni\phi\mapsto\sum_{\alpha=1}^{2|m|}c_{\alpha}(*)\theta_{\alpha}\in S_{\pi_1}(\widetilde{SL}_2(\Z);\Phi_m),
\]
which we have firstly discussed. 
As a result we are done.
\end{proof}
\section{Fourier-Jacobi expansion}\label{FJ-exp}
\subsection{Cusp forms and the working assumption}\label{GCusp-WA}
We first discuss cusp forms on $G=Sp(2,\R)$ with respect to the Siegel modular group $Sp(2,\Z)$ in a general context by representation theory. 
\begin{defn}\label{Def-cuspforms}
Let $\pi$ be an irreducibel admissible representation of $G$ and $(\tau,V)$ be a $K$-type of $\pi$ with the representation space $V$. By $(\tau^*,V^*)$ we denote the contragredient representation of $(\tau,V)$ with the representation space $V^*$. Let $F:G\rightarrow V^*$ be a cusp form of weight $\tau^*$ with respect to $Sp(2,\Z)$, namely a cusp form $F$ satisfying
\[
F(\gamma gk)=\tau^*(k)^{-1}F(g)\quad\forall(\gamma,g,k)\in Sp(2,\Z)\times G\times K.
\]
A cusp form $F$ of weight $\tau^*$ with respect to $Sp(2,\Z)$ is said to be generating $\pi$ if the $G$-module generated by the right $G$-translations of coefficient functions $\{\langle F(g),v\rangle\mid v\in V\}$ of $F$ is isomorphic to $\pi$ as $(\g,K)$-modules, where $V^*\times V\ni (v^*,v)\mapsto\langle v^*,v\rangle\in\C$ denotes a $K$-invariant pairing.
\end{defn}
When $\pi$ is a holomorphic~(respectively~anti-holomorphic) discrete series representation of $G$ and $\tau$ is the minimal $K$-type of $\pi$, this notion is nothing but holomorphic~(respectively~anti-holomorphic) Siegel cusp forms. 

For this definition recall that the cuspidal spectrum of a reductive group decomposes discretely into a sum of irreducible admissible representations with finite multiplicities~(cf.~\cite{G-G-P},~\cite{Gd}). It is thus quite natural to assume the irreducibility of the representations which cusp forms generate. 
Hereafter, every admissible representation of $G$ is therefore assumed to be irreducible. 
Discrete series representations of $G$ are irreducible by definition. The $P_J$-principal series representations with  non-zero purely imaginary parameters $z$ are known to be irreducible as we have remarked just before Theorem \ref{PJWhittaker}. The irreducibility of the principal series representations of $G$ with the assumption (\ref{PS-assumption}) and with purely imaginary $z_1$ and $z_2$, has been pointed out just before Theorem \ref{PS-Whittaker}. For a general criteria on irreducibility (or reducibility) of generalized principal series we cite \cite[Theorem 4.1]{Kn-Z} and \cite[Theorem 1.1]{Sp-V}. 

Now let us note that cusp forms are rapidly decreasing as is well known~(cf.~\cite[Chapter I,~Section 4]{Hc-2},~\cite{Gd}). The Fourier-Jacobi expansion we are going to discuss covers that of holomorphic Siegel cusp forms on $Sp(2,\R)$. We should note that this is applicable also when an irreducible admissible representation $\pi$ admits a rapidly decreasing  Whittaker model, namely, for a non-degenerate character $\psi$ of $N_0$ we have
\[
\Hom_{(\g,K)}(\pi, \cS_{\psi}(N_0\backslash G))\not=0
\]
(for the notation $\cS_{\psi}(N_0\backslash G)$ see Section \ref{ReviewWhittaker}). 

An irreducible admissible representation $\pi$ with the property $\Hom_{(\g,K)}(\pi, \cS_{\psi}(N_0\backslash G))\not=0$ is called generic and we  can thus refer to cusp forms generating such a representation as generic cusp forms. 
All the representations we have taken up except for holomorphic or anti-holomorphic discrete series~(cf.~Sections \ref{DS-rep},~\ref{P_JPS-rep},~\ref{PS-rep}) can be said to be generic in the sense that they admits a rapidly decreasing Whittaker model. For this note that holomorphic or anti-holomorphic discrete series representation are not generic. The genericity of large discrete series representations is due to Theorem \ref{LargeDSWhittaker}. Theorems \ref{PJWhittaker},~\ref{PS-Whittaker} verify that the $P_J$ principal series and the principal series are generic in the above sense.


 To discuss the Fourier-Jacobi expansion of cusp forms generating an irreducible admissible representations $\pi$ we make the following working assumption on the Whittaker functions and the Fourier-Jacobi type spherical functions for $\pi$:
\begin{asm}
There is a multiplicity one $K$-type $\tl$ of $\pi$ such that
\begin{itemize}
\item there is no rapidly decreasing element in $W_{\psi,\pi}(\tl^*)^0$~(for the notation see Section \ref{ReviewWhittaker}) for any degenerate $\psi\in\hat{N_0}$,
\item $\dim\cJ_{\rho,\pi}(\tl^*)^{00}\le 1$~(for the notation see Section \ref{Result-Hirano}) holds for any irreducible unitary representations $\rho$ of $G_J$ with a non-trivial central character.
\end{itemize}
Here $\tl^*$ denotes the contragredient of $\tl$. 
\end{asm}
We have seen that this assumption is satisfied by discrete series representations, irreducible $P_J$-principal series representations (e.g. with purely imaginary $z$) and irreducible principal series representations~(e.g. with the purely imaginary $z_1,~z_2$ satisfying (\ref{PS-assumption})). 
We remark that $\dim W_{\psi,\pi}(\tl^*)^0\le 1$ holds for general irreducible admissible representations of quasi-split real reductive groups and for non-degenerate $\psi$ in view of Wallach \cite[Theorem 8.8]{W-1}. 
From Liu-Sun \cite{Li-Su} we can expect that $\dim\cJ_{\rho,\pi}(\tl^*)^{00}\le 1$ would hold in general though it does not deal with the case of the real symplectic group.  
\subsection{The statement of the main result~(Fourier-Jacobi expansion)}\label{FJ-expansion}
The Fourier Jacobi expansion of a cusp form $F$ is written as
\[
F(g)=\sum_{m\in\Z}F_m(g)(g\in G),\quad F_m(g):=\displaystyle\int_{\R/\Z}F(n(0,0,z)g)\e(-mz)dz.
\]
Let $F$ be of weight $\tl^*$ and generate an irreducible admissible representation $\pi$. To state the result on this we introduce a couple of ingredients. 
\begin{enumerate}
\item {\bf Whittaker functions of $F$.}\\
We first let 
\[
F_{\xi_0,\xi_3}(g):=\displaystyle\int_{N_0(\Z)\backslash N_0}F(n(u_0,u_1,u_2,u_3)g)\psi_{\xi_0,\xi_3}(n(u_0,u_1,u_2,u_3))^{-1}dn\quad(g\in G)
\]
for a unitary character 
\[
\psi_{\xi_0,\xi_3}:N_0\ni n(u_0,u_1,u_2,u_3)\mapsto\exp(2\pi\sqrt{-1}(\xi_0u_0+\xi_3u_3))
\]
of $N_0$ with $(\xi_0,\xi_3)\in\Z^2$, where $dn$ denotes the invariant measure of $N_0(\Z)\backslash N_0$ normalized so that ${\rm vol}(N_0(\Z)\backslash N_0)=1$. As we have remarked, by \cite[Theorem 8.8]{W-1},  $F_{\xi_0,\xi_3}$ is a constant multiple of a Whittaker function in $W_{\psi_{\xi_0,\xi_3},\pi}(\tl^*)^0$ when $\psi_{\xi_0,\xi_3}$ is non-degenerate. This is necessary to write all the $F_m$-terms.
\item {\bf Eisenstein-Poincar{\'e} series on $G_J$.}\\
To introduce the Eisenstein-Poinar{\'e} series on $G_J$ in our concern we begin with the remark that singular semi-integral matrices of degree two with the fixed upper left entry $m$ is of the forms 
\[
S_{\alpha,m}(k)=
\begin{pmatrix}
m & km+\frac{\alpha}{2}\\
km+\frac{\alpha}{2} & m(k+\frac{\alpha}{2m})^2
\end{pmatrix}
\]
with $k\in\Z$ and with $\alpha\in\Z$ such that $\displaystyle\frac{\alpha^2}{4m}\in\Z$, where $\alpha$ is determined modulo $2m$. We put $\gamma_{\alpha,m}(k):=n(k+\frac{\alpha}{2m},0,0,0)$ for $k\in\Z$ and
\[
S_{\alpha,m}:=S_{\alpha,m}(0),\quad\gamma_{\alpha,m}:=\gamma_{\alpha,m}(0).
\]
Here we note that the set $\{\gamma_{\alpha,m}\mid 1\le\alpha\le 2|m|,~\frac{\alpha^2}{4m}\in\Z\}$ form representatives of $G_J(\Z)$-cusps which are essential to characterize the cupsidal condition for $\cH_m^0$~(cf.~\cite[Section 4.2]{Be-Sc}). With these preparation we introduce the followings:
\begin{itemize}
\item We put  
\[
F_{S_{\alpha,m}}(g):=\displaystyle\int_{N_S(\Z)\backslash N_S}F(n(0,u_1,u_2,u_3)g)\e(-\tr(S_{\alpha,m}
\begin{pmatrix}
u_1 & u_2\\
u_2 & u_3
\end{pmatrix}))du_1du_2du_3.
\]
\item Let
\[
L_{\alpha,m}:=d_{\alpha,m}\Z\quad\text{with $d_{\alpha,m}:=(\dfrac{2m}{\alpha})^2$},
\] 
where we take $\alpha$ in $1\le \alpha\le 2|m|$ such that $\dfrac{\alpha^2}{4m}\in\Z$. 
The lattice $L_{\alpha,m}$ coincides with
\[
\{u_0\in\R\mid \gamma_{\alpha,m}^{-1}{}^tn(u_0,0,0,0)\gamma_{\alpha,m}\in Sp(2,\Z)\}.
\]
The dual lattice $\widehat{L_{\alpha,m}}$ of $L_{\alpha,m}$ is $\frac{1}{d_{\alpha,m}}\Z$. We then introduce
\[
F_{S_{\alpha,m},n}(g):=\displaystyle\int_{\R/L_{\alpha,m}}F_{S_{\alpha,m}}(\gamma_{\alpha,m}^{-1}{}^tn(u_0,0,0,0)\gamma_{\alpha,m}g)\e(-nu_0)du_0
\]
for $n\in\widehat{L_{\alpha,m}}\setminus\{0\}$, where the measure $du_0$ is normalized so that $\vol(\R/L_{\alpha,m})=1$. This turns out to be the left translate of a Whittaker function by $({}^t\gamma_{\alpha,m}\xi)^{-1}$~(cf.~Lemma \ref{Whittaker-F_m} in Section \ref{ProofThm}), where  recall that $\xi=
\begin{pmatrix}
J'_2 & 0_2\\
0_2 & J'_2
\end{pmatrix}$ with $J'_2=
\begin{pmatrix}
0 & 1\\
1 & 0
\end{pmatrix}$~(cf.~Theorem \ref{LargeDSWhittaker} (2)).
\item
For $\alpha\mod 2m\in\Z/2m\Z$ with $\displaystyle\frac{\alpha^2}{4m}\in\Z$ we introduce
\[
E_{\alpha}(F_{S_{\alpha,m},n}(*g))(r):=\sum_{\gamma\in Z_J(\Z)(G_J(\Z)\cap\gamma_{\alpha,m}^{-1}V_J\gamma_{\alpha,m})\backslash G_J(\Z)}F_{S_{\alpha,m},n}(\gamma\cdot rg),
\]
where see Section \ref{Groups} (2) for the notation $Z_J(\Z)$. We want to call this a Jacobi Eisenstein-Poincar{\'e} series with the test function $F_{S_{\alpha,m},n}(g)$. 
Though this is called an ``incomplete theta series'' and should be denoted by $\theta_{\alpha}$ as in  \cite[Section 4.4]{Be-Sc}, we use the notation $E_{\alpha}$ to avoid the confusion with $\theta_{\alpha}\in\Phi_m$.
\end{itemize}
\item {\bf Fourier-Jacobi type spherical functions.}\\
We further recall from section \ref{Result-Hirano} that the Fourier-Jacobi type spherical function of type $(\rho,\pi;\tl^*)$ restricted to $A_J$ has been written as 
\[
\displaystyle\sum_{
\begin{subarray}{c}
j\in J,~0\le k\le d_{\Lambda}\\
\text{s.t.}~l=l(j,k)\in L
\end{subarray}}c_{j,k}(a_J)w_l\otimes u_j^m\otimes v_{k}^*\quad (a_J\in A_J)
\]
when the irreducible unitary representation $\rho$ has the non-trivial central character indexed by $m\in\Z\setminus\{0\}$. 
In what follows, when $\rho$ is specified as $\rho_{m,\pi_1}$ with an irreducible genuine unitary representation $\pi_1$ of $\widetilde{SL}_2(\R)$, we denote $c_{j,k}$ by $c_{j,k}^{(\pi_1)}$ in order to indicate the dependence of $c_{j,k}$ on $\pi_1$. We thus write 
\[
\sum_{
\begin{subarray}{c}
j\in J,~0\le k\le d_{\Lambda}\\
\text{s.t.}~l=l(j,k)\in L
\end{subarray}}c_{j,k}^{(\pi_1)}(a_J)w_l\otimes u_j^m\otimes v_{k}^*\quad(a_J\in A_J)
\]
for the Fourier-Jacobi type spherical function. 
To describe the $F_m$-term with $m\not=0$ we need this spherical functions as well as the Eisenstein-Poincar{\'e} series introduced above. 
\end{enumerate}
We remark that the Jacobi parabolic subgroup $P_J$ coincides with $G_JA_J$, for which recall that $A_J=\{
a_J=\begin{pmatrix}
a_1 & & & \\
& 1 & &  \\
& & a_1^{-1} & \\
& & & 1
\end{pmatrix}\mid \alpha\in\R^{\times}_+\}$. With this coordinate we state the theorem as follows:
\begin{thm}\label{F-J-exp-mainthm}
Let $\pi$ be an irreducible admissible representation of $G$ with the multiplicity one $K$-type $\tl$ satisfying the working assumption and let $F$ be a cusp form of weight $\tl^*$ with respect to $Sp(2,\Z)$ generating $\pi$. 

Each term $F_m$ of the Fourier-Jacobi expansion $\sum_{m\in\Z}F_m$ of $F$ is expressed as 
\[
F_m(ra_J)=\begin{cases}
\sum_{(\xi_0,\xi_3)\in\Z^2,~\xi_0\xi_3\not=0}
\sum_{
\begin{pmatrix}
a & b\\
c & d
\end{pmatrix}\in SL_2(\Z)_{\infty}\backslash SL_2(\Z)}F_{\xi_0,\xi_3}(
\begin{pmatrix}
1 & 0 & 0 & 0\\
0 & a & 0 & b\\
0 & 0 & 1 & 0\\
0 & c & 0 & d
\end{pmatrix}ra_J)&(m=0)\\
\sum_{
\begin{subarray}{c}
1\le\alpha\le 2|m|\\
\text{s.t.~$\alpha^2/4m\in\Z$}
\end{subarray}}\sum_{n\in\widehat{L_{\alpha,m}}\setminus\{0\}}b_{m,\alpha}(F)E_{\alpha}(F_{S_{\alpha,m},n}(*a_J))(r)+\\
\underset{\pi_1\in\widehat{\widetilde{SL}_2(\R)},m(\pi_1)\not=0}{\sum}\sum_{i=1}^{m(\pi_1)}b_{m,i}^{(\pi_1)}(F)F_{m,i}^{(\pi_1)}(ra_J)&(m\not=0),
\end{cases}
\]
with
\[
F_{m,i}^{(\pi_1)}(ra_J):=
\underset{\begin{subarray}{c}
j\in J,~0\le k\le d_{\Lambda}\\
\text{s.t.}~l=l(j,k)\in L
\end{subarray}}{\sum}c_{j,k}^{(\pi_1)}(a_J)\phi_{\pi_1}^{(i)}(w_l\otimes u_j^m)(r)\otimes v_k^*,
\]
where 
\begin{itemize}
\item $SL_2(\Z)_{\infty}:=\left\{\left.
\begin{pmatrix}
1 & n\\
0 & 1
\end{pmatrix}~\right|~n\in\Z\right\}$,
\item For an element $\pi_1$ in the unitary dual $\widehat{\widetilde{SL}_2(\R)}$~(the equivalence classes of irreducible unitary representations of $\widetilde{SL}_2(\R)$), we put 
\[
\m(\pi_1):=\dim\Hom_{G_J}(\rho_{m,\pi_1},\cH_m^0)
\]
and $\{\phi^{(i)}_{\pi_1}\}$ denotes a basis of $\Hom_{G_J}(\rho_{m,\pi_1},\cH_m^0)$,
\item $b_{m,i}^{(\pi_1)}(F)$~(respectively~$b_{m,\alpha}(F)$) is a constant depending on the normalization of the Fourier-Jacobi type spherical function for $\rho=\rho_{m,\pi_1}$ and on a choice of a basis $\{\phi_{\pi_1}^{(i)}\}$ (respectively~on $m$ and $\alpha$).
\end{itemize}.
\end{thm}
We can then carry out a more specific description of the Fourier-Jacobi expansion for discrete series representations, $P_J$-principal series representations and principal series representations in view of the results collected in Section \ref{ReviewSpherical}.
\begin{cor}\label{FJ-exp-fourcases}
Let the notation be as in Theorem \ref{F-J-exp-mainthm}.\\
(1)~Suppose that $F$ is a cusp form of weight $\tl^*$ generating a holomorphic discrete series representation $\pi_{\lambda}$ with $\lambda\in\Xi_{I}$, where $\tl^*$ is the contragredient of the minimal $K$-type $\tl$ of $\pi_{\lambda}$. For each $F_m$-term of the Fourier-Jacobi expansion $F=\sum_{m>0}F_m$, $F_m(ra_J)$ has the following expression:
\[
F_m(ra_J)=
\begin{cases}
\sum_{
\begin{subarray}{c}
n_1\in\frac{1}{2}\Z_{\ge 3}\setminus\Z,\\
\lambda_2+\frac{3}{2}\le n_1\le\lambda_1+\frac{1}{2}\\
\text{s.t.}~m(\cD_{n_1}^+)\not=0
\end{subarray}
}
\sum_{i=1}^{m(\cD_{n_1}^+)}b_{m,i}^{(\cD_{n_1}^+)}(F)F_{m,i}^{(\cD_{n_1}^+}(ra_J)&(m>0),\\ 
0&(m\le 0),
\end{cases}
\]
with
\[
F_{m,i}^{(\cD_{n_1}^+)}(ra_J)=\underset{\begin{subarray}{c}
n_1\le l\le \lambda_1+\frac{1}{2}\\
l\equiv n_1\mod 2
\end{subarray}}{\sum}\sum_{k=k(l)}^{d_{\Lambda}}
c_{k-k(l)+\frac{1}{2},k}^{(\cD_{n_1}^+)}(a_J)\phi_{\cD_{n_1}^+}^{(i)}(w_l\otimes u_{k-k(l)+\frac{1}{2}}^m)(r)\otimes v_k^*,
\]
where $k(l):=l-\lambda_2-\frac{3}{2}$.

Suppose next that $F$ is a cusp form of weight $\tl^*$ generating an anti-holomorphic discrete series representation $\pi_{\lambda}$ with $\lambda\in\Xi_{IV}$, where $\tl^*$ is the contragredient of the minimal $K$-type $\tl$ of $\pi_{\lambda}$. Each $F_m$-term of the Fourier-Jacobi expansion of $F$ has the expression as follows:
\[
F_m(ra_J)=\begin{cases}
\sum_{
\begin{subarray}{c}
n_1\in\frac{1}{2}\Z_{\ge 3}\setminus\Z,\\
-\lambda_1+\frac{3}{2}\le n_1\le-\lambda_2+\frac{1}{2}\\
\text{s.t.}~m(\cD_{n_1}^-)\not=0
\end{subarray}
}
\sum_{i=1}^{m(\cD_{n_1}^-)}b_{m,i}^{(\cD_{n_1}^-)}(F)F_{m,i}^{(\cD_{n_1}^-)}(ra_J)&(m<0),\\
0&(m\ge 0),
\end{cases}
\]
with 
\[
F_{m,i}^{(\cD_{n_1}^-)}(ra_J)=\underset{\begin{subarray}{c}
\lambda_2-\frac{1}{2}\le l\le -n_1\\
l\equiv -n_1\mod 2
\end{subarray}}{\sum}\sum_{k=0}^{k(l)}
c_{k-k(l)-\frac{1}{2},k}^{(\cD_{n_1}^-)}(a_J)\phi_{\cD_{n_1}^-}^{(i)}(w_l\otimes u_{k-k(l)-\frac{1}{2}}^m)(r)\otimes v_k^*,
\]
where $k(l):=l-\lambda_2+\frac{1}{2}$.\\
(2)~Suppose that $F$ is a cusp form of weight $\tl^*$ generating a large discrete series representation $\pi_{\lambda}$ with $\lambda\in\Xi_{II}$, where $\tl^*$ is the contragredient of the minimal $K$-type $\tl$ of $\pi_{\lambda}$. 
For each term $F_m$ of the Fourier-Jacobi expansion $F=\sum_{m\in\Z}F_m$, $F_m(ra_J)$ has the following expression:
\begin{itemize}
\item $\sum_{
\begin{subarray}{c}
(\xi_0,\xi_3)\in\Z^2\\
\xi_0\not=0,~\xi_3>0
\end{subarray}}\sum_{
\begin{pmatrix}
a & b\\
c & d
\end{pmatrix}\in SL_2(\Z)_{\infty}\backslash SL_2(\Z)}F_{\xi_0,\xi_3}(\begin{pmatrix}
1 & 0 & 0 & 0\\
0 & a & 0 & b\\
0 & 0 & 1 & 0\\
0 & c & 0 & d
\end{pmatrix}ra_J)~(m=0)$,
\item $\sum_{\begin{subarray}{c}
1\le\alpha\le 2m\\
\text{s.t.~$\frac{\alpha^2}{4m}\in\Z$}
\end{subarray}}\sum_{n\in\widehat{L_{\alpha,m}}\setminus\{0\}}b_{m,\alpha}(F)E_{\alpha}(F_{S_{\alpha,m},n}(*a_J))(r)+\\
\sum_{
\begin{subarray}{c}
\pi_1=\cP_s^{\tau}~(\tau=\pm\frac{1}{2},~s\in\sqrt{-1}\R)\\
\text{or}~\cC_s^{\tau}~(\tau=\pm\frac{1}{2},~0<s<\frac{1}{2})\\
\text{s.t.}~m(\pi_1)\not=0
\end{subarray}
}
\sum_{i=1}^{m(\pi_1)}b_{m,i}^{(\pi_1)}(F)F_{m,i}^{(\pi_1)}(ra_J)+\sum_{
\begin{subarray}{c}
n_1\in\frac{1}{2}\Z_{\ge 3}\setminus\Z,\\
n_1\le-\lambda_2+\frac{1}{2}\\
\text{s.t.}~m(\cD_{n_1}^-)\not=0
\end{subarray}
}
\sum_{i=1}^{m(\cD_{n_1}^-)}b_{m,i}^{(\cD_{n_1}^-)}(F)F_{m,i}^{(\cD_{n_1}^-)}(ra_J)$\\
$(m>0)$,
\item $\sum_{
\begin{subarray}{c}
n_1\in\frac{1}{2}\Z_{\ge 3}\setminus\Z,~n_1>\lambda_1+\frac{1}{2}\\
\text{s.t.}~m(\cD_{n_1}^+)\not=0
\end{subarray}
}
\sum_{i=1}^{m(\cD_{n_1}^+)}b_{m,i}^{(\cD_{n_1}^+)}(F)F_{m,i}^{(\cD_{n_1}^+)}(ra_J)$
$(m<0)$.
\end{itemize}

Suppose next that $F$ is a cusp form of weight $\tl^*$ generating a large discrete series representation $\pi_{\lambda}$ with $\lambda\in\Xi_{III}$, where $\tl^*$ is the contragredient of the minimal $K$-type $\tl$ of $\pi_{\lambda}$. 
For each term $F_m$, $F_m(ra_J)$ has the expression as follows:
\begin{itemize}
\item $\sum_{
\begin{subarray}{c}
(\xi_0,\xi_3)\in\Z^2\\
\xi_0\not=0,~\xi_3<0
\end{subarray}}\sum_{
\begin{pmatrix}
a & b\\
c & d
\end{pmatrix}\in SL_2(\Z)_{\infty}\backslash SL_2(\Z)}F_{\xi_0,\xi_3}(\begin{pmatrix}
1 & 0 & 0 & 0\\
0 & a & 0 & b\\
0 & 0 & 1 & 0\\
0 & c & 0 & d
\end{pmatrix}ra_J)~(m=0)$,
\item $\sum_{
\begin{subarray}{c}
n_1\in\frac{1}{2}\Z_{\ge 3}\setminus\Z,~n_1>-\lambda_2+\frac{1}{2}\\
\text{s.t.}~m(\cD_{n_1}^-)\not=0
\end{subarray}
}
\sum_{i=1}^{m(\cD_{n_1}^-)}b_{m,i}^{(\cD_{n_1}^-)}(F)F_{m,i}^{(\cD_{n_1}^-)}(ra_J)$
$(m>0)$,
\item $\sum_{\begin{subarray}{c}
1\le\alpha\le 2|m|\\
\text{s.t.~$\frac{\alpha^2}{4m}\in\Z$}
\end{subarray}}\sum_{n\in\widehat{L_{\alpha,m}}\setminus\{0\}}b_{m,\alpha}(F)E_{\alpha}(F_{S_{\alpha,m},n}(*a_J))(r)+\\
\sum_{
\begin{subarray}{c}
\pi_1=\cP_s^{\tau}~(\tau=\pm\frac{1}{2},~s\in\sqrt{-1}\R)\\
\text{or}~\cC_s^{\tau}~(\tau=\pm\frac{1}{2},~0<s<\frac{1}{2})\\
\text{s.t.}~m(\pi_1)\not=0
\end{subarray}
}
\sum_{i=1}^{m(\pi_1)}b_{m,i}^{(\pi_1)}(F)F_{m,i}^{(\pi_1)}(ra_J)+\sum_{
\begin{subarray}{c}
n_1\in\frac{1}{2}\Z_{\ge 3}\setminus\Z,\\
n_1\le\lambda_1+\frac{1}{2}\\
\text{s.t.}~m(\cD_{n_1}^+)\not=0
\end{subarray}
}
\sum_{i=1}^{m(\cD_{n_1}^+)}b_{m,i}^{(\cD_{n_1}^+)}(F)F_{m,i}^{(\cD_{n_1}^+)}(ra_J)$\\$(m<0)$.
\end{itemize}
(3)~Suppose that $F$ generates an irreducible $P_J$-principal series representation with $\sigma=(\cD_n^+,\epsilon)$ and the weight of $F$ is the contragredient of its corner $K$-type. For each term $F_m$ of the Fourier-Jacobi expansion $F=\sum_{m\in\Z}F_m$, $F_m(ra_J)$ has the following expression:
\begin{itemize} 
\item $\sum_{
\begin{subarray}{c}
(\xi_0,\xi_3)\in\Z^2\\
\xi_0\not=0,~\xi_3>0
\end{subarray}}\sum_{
\begin{pmatrix}
a & b\\
c & d
\end{pmatrix}\in SL_2(\Z)_{\infty}\backslash SL_2(\Z)}F_{\xi_0,\xi_3}(\begin{pmatrix}
1 & 0 & 0 & 0\\
0 & a & 0 & b\\
0 & 0 & 1 & 0\\
0 & c & 0 & d
\end{pmatrix}ra_J)~(m=0)$,
\item $\sum_{\begin{subarray}{c}
1\le\alpha\le 2m\\
\text{s.t.~$\frac{\alpha^2}{4m}\in\Z$}
\end{subarray}}\sum_{n\in\widehat{L_{\alpha,m}}\setminus\{0\}}b_{m,\alpha}(F)E_{\alpha}(F_{S_{\alpha,m},n}(*a_J))(r)+\\
\sum_{
\begin{subarray}{c}
\pi_1=\cP_s^{\tau}~(\tau=\pm\frac{1}{2},~s\in\sqrt{-1}\R),\\
\cC_s^{\tau}~(\tau=\pm\frac{1}{2},~0<s<\frac{1}{2}),\\
\text{or}~\cD_{\frac{1}{2}}^-~\text{with $n=1$}\\
\text{s.t.}~m(\pi_1)\not=0
\end{subarray}
}
\sum_{i=1}^{m(\pi_1)}b_{m,i}^{(\pi_1)}(F)F_{m,i}^{(\pi_1)}(ra_J)+
\sum_{
\begin{subarray}{c}
n_1\in\frac{1}{2}\Z_{\ge 3}\setminus\Z,\\
n_1\le n-\frac{1}{2}\\
\text{s.t.}~m(\cD_{n_1}^+)\not=0
\end{subarray}
}
\sum_{i=1}^{m(\cD_{n_1}^+)}b_{m,i}^{(\cD_{n_1}^+)}(F)F_{m,i}^{(\cD_{n_1}^+)}(ra_J)$\\
$(m>0)$,
\item $\sum_{
\begin{subarray}{c}
n_1\in\frac{1}{2}\Z_{\ge 3}\setminus\Z,~n_1>n-\frac{1}{2}\\
\text{s.t.}~m(\cD_{n_1}^+)\not=0
\end{subarray}
}
\sum_{i=1}^{m(\cD_{n_1}^+)}b_{m,i}^{(\cD_{n_1}^+)}(F)F_{m,i}^{(\cD_{n_1}^+)}(ra_J)$~$(m<0)$.
\end{itemize}
Suppose next that $F$ generates an irreducible $P_J$-principal series representation with $\sigma=(\cD_{n}^-,\epsilon)$ and the weight of $F$ is the contragredient of its corner $K$-type. 
For each term $F_m$, $F_m(ra_J)$ has the expression as follows:
\begin{itemize}
\item $\sum_{
\begin{subarray}{c}
(\xi_0,\xi_3)\in\Z^2\\
\xi_0\not=0,~\xi_3<0
\end{subarray}}\sum_{
\begin{pmatrix}
a & b\\
c & d
\end{pmatrix}\in SL_2(\Z)_{\infty}\backslash SL_2(\Z)}F_{\xi_0,\xi_3}(\begin{pmatrix}
1 & 0 & 0 & 0\\
0 & a & 0 & b\\
0 & 0 & 1 & 0\\
0 & c & 0 & d
\end{pmatrix}ra_J)~(m=0)$,
\item $\sum_{
\begin{subarray}{c}
n_1\in\frac{1}{2}\Z_{\ge 3}\setminus\Z,~n_1>n-\frac{1}{2}\\
\text{s.t.}~m(\cD_{n_1}^-)\not=0
\end{subarray}
}
\sum_{i=1}^{m(\cD_{n_1}^-)}b_{m,i}^{(\cD_{n_1}^-)}(F)F_{m,i}^{(\cD_{n_1}^-)}(ra_J)$~
$(m>0)$,
\item $\sum_{\begin{subarray}{c}
1\le\alpha\le 2|m|\\
\text{s.t.~$\frac{\alpha^2}{4m}\in\Z$}
\end{subarray}}\sum_{n\in\widehat{L_{\alpha,m}}\setminus\{0\}}b_{m,\alpha}(F)E_{\alpha}(F_{S_{\alpha,m},n}(*a_J))(r)+\\
\sum_{
\begin{subarray}{c}
\pi_1=\cP_s^{\tau}~(\tau=\pm\frac{1}{2},~s\in\sqrt{-1}\R),\\
\cC_s^{\tau}~(\tau=\pm\frac{1}{2},~0<s<\frac{1}{2}),\\
\text{or}~\cD_{\frac{1}{2}}^+~\text{with $n=1$}\\
\text{s.t.}~m(\pi_1)\not=0
\end{subarray}
}
\sum_{i=1}^{m(\pi_1)}b_{m,i}^{(\pi_1)}(F)F_{m,i}^{(\pi_1)}(ra_J)+
\sum_{
\begin{subarray}{c}
n_1\in\frac{1}{2}\Z_{\ge 3}\setminus\Z,\\
n_1\le n-\frac{1}{2}\\
\text{s.t.}~m(\cD_{n_1}^-)\not=0
\end{subarray}
}
\sum_{i=1}^{m(\cD_{n_1}^-)}b_{m,i}^{(\cD_{n_1}^-)}(F)F_{m,i}^{(\cD_{n_1}^-)}(ra_J)$\\
$(m<0)$.
\end{itemize}
(4)~Suppose that $F$ generates an irreducible principal series representation with the assumptions (\ref{PS-assumption}), and that the weight of $F$ is the contragredient of the minimal $K$-type of $\pi$, which is a scalar $K$-type when $\pi$ is even. For each term $F_m$, $F_m(ra_J)$ has the expression as follows:
\begin{itemize}
\item $\sum_{
\begin{subarray}{c}
(\xi_0,\xi_3)\in\Z^2\\
\xi_0\xi_3\not=0,
\end{subarray}}\sum_{
\begin{pmatrix}
a & b\\
c & d
\end{pmatrix}\in SL_2(\Z)_{\infty}\backslash SL_2(\Z)}F_{\xi_0,\xi_3}(\begin{pmatrix}
1 & 0 & 0 & 0\\
0 & a & 0 & b\\
0 & 0 & 1 & 0\\
0 & c & 0 & d
\end{pmatrix}ra_J)~(m=0)$,
\item $\sum_{\begin{subarray}{c}
1\le\alpha\le 2m\\
\text{s.t.~$\frac{\alpha^2}{4m}\in\Z$}
\end{subarray}}\sum_{n\in\widehat{L_{\alpha,m}}\setminus\{0\}}b_{m,\alpha}(F)E_{\alpha}(F_{S_{\alpha,m},n}(*a_J))(r)+\\
\sum_{
\begin{subarray}{c}
\pi_1=\cP_s^{\tau}~(\tau=\pm\frac{1}{2},~s\in\sqrt{-1}\R),\\
\text{or}~\cC_s^{\tau}~(\tau=\pm\frac{1}{2},~0<s<\frac{1}{2})\\
\text{s.t.}~m(\pi_1)\not=0
\end{subarray}
}
\sum_{i=1}^{m(\pi_1)}b_{m,i}^{(\pi_1)}(F)F_{m,i}^{(\pi_1)}(ra_J)+
\sum_{
\begin{subarray}{c}
n_1\in\frac{1}{2}\Z_{\ge 3}\setminus\Z\\
\text{s.t.}~m(\cD_{n_1}^-)\not=0
\end{subarray}
}
\sum_{i=1}^{m(\cD_{n_1}^-)}b_{m,i}^{(\cD_{n_1}^-)}(F)F_{m,i}^{(\cD_{n_1}^-)}(ra_J)$\\
$(m>0)$,
\item $\sum_{\begin{subarray}{c}
1\le\alpha\le 2|m|\\
\text{s.t.~$\frac{\alpha^2}{4m}\in\Z$}
\end{subarray}}\sum_{n\in\widehat{L_{\alpha,m}}\setminus\{0\}}b_{m,\alpha}(F)E_{\alpha}(F_{S_{\alpha,m},n}(*a_J))(r)+\\
\sum_{
\begin{subarray}{c}
\pi_1=\cP_s^{\tau}~(\tau=\pm\frac{1}{2},~s\in\sqrt{-1}\R),\\
\text{or}~\cC_s^{\tau}~(\tau=\pm\frac{1}{2},~0<s<\frac{1}{2})\\
\text{s.t.}~m(\pi_1)\not=0
\end{subarray}
}
\sum_{i=1}^{m(\pi_1)}b_{m,i}^{(\pi_1)}(F)F_{m,i}^{(\pi_1)}(ra_J)+
\sum_{
\begin{subarray}{c}
n_1\in\frac{1}{2}\Z_{\ge 3}\setminus\Z\\
\text{s.t.}~m(\cD_{n_1}^+)\not=0
\end{subarray}
}
\sum_{i=1}^{m(\cD_{n_1}^+)}b_{m,i}^{(\cD_{n_1}^+)}(F)F_{m,i}^{(\cD_{n_1}^+)}(ra_J)$\\
$(m<0)$.
\end{itemize}
\end{cor}
\begin{proof}
We give a proof only for the cases of discrete series representations $\pi_{\lambda}$ with $\lambda\in\Xi_{I}\cup\Xi_{II}$. Discrete series $\pi_{\lambda}$ with $\lambda\in\Xi_{I}$ are nothing but holomorphic discrete series. This case is the easiest to handle and it is more accessible to understand  $F_{m,i}^{(\cD_{n_1}^+)}$ explicitly than non-holomorphic cases. The assertion for this case is a consequence from Theorems \ref{FJ-holom-antiholom-DS},~\ref{F-J-exp-mainthm} and the well known fact that holomorphic discrete series do not admit rapidly decreasing Whittaker models~(also those without the rapidly decreasing property). The latter fact is verified also for the case of degenerate characters~(cf.~\cite[Theorems 6.7,~8.2]{Hi-1},~\cite[Proposition 7.1~(1),~(2)]{Nr-1}). 
We can thereby remark that $F_0\equiv 0$ and $F_m^{\rm c}\equiv 0$ hold for $m>0$. 
In addition we remark that the case of anti-holomorphic discrete series is settled by an argument parallel to the holomorphic case. 

Let us next consider the case of $\pi_{\lambda}$ with $\lambda\in\Xi_{II}$, which is a large discrete series representation. From Theorem \ref{LargeDSWhittaker}  and Proposition \ref{DegnerateWhittaker} we have known that $\dim W_{\psi_{\xi_0,\xi_3},\pi_{\lambda}}(\tl^*)^0=1$ if and only if $\xi_0\not=0,~\xi_3>0$. 
We furthermore recall from Theorems \ref{explicit-FJ-I},~\ref{explicit-FJ-II} that $\dim\cJ_{\rho_{\pi_1,m},\pi_{\lambda}}(\tl^*)^{00}=1$ if and only if 
\begin{itemize}
\item $m>0$ and $\pi_1=\cP_s^{\tau}~(\tau=\pm\frac{1}{2},~s\in\sqrt{-1}\R)$ or $\cC_s^{\tau}~(\tau=\pm\frac{1}{2},~0<s<\frac{1}{2})$,
\item $m>0$ and $\pi_1=\cD_{n_1}^-$ with $n_1\in\frac{1}{2}\Z_{\ge 3}\setminus\{0\}$ and $n_1\le-\lambda_2+\frac{1}{2}$,
\item $m<0$ and $\pi_1=\cD_{n_1}^+$ with $n_1\in\frac{1}{2}\Z_{\ge 3}\setminus\Z$ and $n_1>\lambda_1+\frac{1}{2}$.
\end{itemize}
The result of this case is obtained by combining Theorem \ref{F-J-exp-mainthm} with these. The other cases are settled similarly by Proposition \ref{DegnerateWhittaker} and the theorems etc. in Section \ref{ReviewSpherical}~(or Hirano's results cited there).
\end{proof}
\begin{rem}
In terms of the analogy of the Selberg conjecture on the minimal eigenvalue of the hyperbolic Laplacian, the contribution from the complementary series representations $\cC_s$ is conjectured to be zero. However, it seems that there is no advancement of the Selberg conjecture for cusp forms on $\widetilde{SL}_2(\R)$, namely for the case of half-integral weights. We can not therefore deny the contribution of $\cC_s$ to the Fourier-Jacobi expansion for non-holomorphic cases.
\end{rem}
\subsection{Proof of Theorem \ref{F-J-exp-mainthm}}\label{ProofThm}
\subsection*{(I)~$F_0$-term}
For $m=0$, the representations of $N_J$ contributing to $F_0$ are characters of $N_J$. 
Hirano \cite[Section 8]{Hi-1} essentially pointed out that Fourier-Jacobi type spherical functions for a character $\eta_{\xi_0,\xi_2}$ of $N_J$~(for $\eta_{\xi_0,\xi_2}$ see Proposition \ref{UnitaryDual-N_J} (1)) are Whittaker functions attached to unitary characters of $N_J\rtimes{\rm Stab}_{SL_2(\R)}(\eta_{\xi_0,\xi_2})$, which is isomorphic to $N_0$. We note that this remark needs Proposition \ref{UnitaryDual-G_J} (2) part 2. We now recall  that the working assumption excludes the contribution to the Fourier-Jacobi expansion of a cusp form $F$ by Whittaker functions attached to degenerate characters of $N_0$.  
We are going to show that the $\eta_{\xi_0,\xi_2}$-term of $F_0$ with $(\xi_0,\xi_2)\not=(0,0)$ is a sum of the Whittker functions attached to non-degenerate characters of $N_J\rtimes{\rm Stab}_{SL_2(\R)}(\eta_{\xi_0,\xi_2})$. For this note that there is no $\eta_{0,0}$-term of $F_0$, due to the definition of a cusp form on $G$. 

Let us expand $F_0$ by characters of $N_J$. By $W_{\xi_0,\xi_2}$ we denote the $\eta_{\xi_0,\xi_2}$-term of the expansion of $F_0$ and have
\[
F_0(g)=\sum_{(\xi_0,\xi_2)\in\Z^2\setminus\{(0,0)\}}W_{\xi_0,\xi_2}(g)~(g\in G).
\]
We first describe $W_{\xi_0,0}$ in detail for $\xi_0\in\Z\setminus\{0\}$. For this the group ${\rm Stab}_{SL_2(\R)}(\eta_{\xi_0,0})$ coincides with $\left\{\left.
\begin{pmatrix}
1 & x\\
0 & 1
\end{pmatrix}~\right|~x\in\R\right\}$, which implies $N_J\rtimes{\rm Stab}_{SL_2(\R)}(\eta_{\xi_0,0})=N_0$. Then 
the Fourier expansion of $W_{\xi_0,0}$ along ${\rm Stab}_{SL_2(\R)}(\eta_{\xi_0,0})=\left\{\left.
\begin{pmatrix}
1 & x\\
0 & 1
\end{pmatrix}~\right|~x\in\R\right\}$ is
\[
W_{\xi_0,0}(g)=\sum_{\xi_3\in\Z\setminus\{0\}}F_{\xi_0,\xi_3}(g),
\]
for which we remark that we can exclude $F_{\xi_0,0}$ since there is no contribution to $F_0$ by Whittaker functions attached to degenerate characters of $N_0$ because of the working assumption. 

To understand $W_{\xi_0,\xi_2}$ for $\xi_2\not=0$ we note the following:
\begin{itemize}
\item Every $(\xi_0,\xi_2)\in\Z^2\setminus\{(0,0)\}$ can be written as $
\begin{pmatrix}
\xi_0\\
\xi_2
\end{pmatrix}=\gamma
\begin{pmatrix}
\delta\\
0
\end{pmatrix}$ with some $\gamma\in SL_2(\Z)$, where $\delta$ is the greatest common divisor of $(\xi_0,\xi_2)$ when $\xi_0\xi_2\not=0$.
\item For each $\xi\in\Z\setminus\{0\}$ the group of elements in $SL_2(\Z)$ fixing $
\begin{pmatrix}
\xi\\
0
\end{pmatrix}$ is $SL_2(\Z)_{\infty}$.
\item $N_J\rtimes{\rm Stab}_{SL_2(\R)}(\eta_{\xi_0,\xi_2})$ is conjugate to $N_0$ by 
$\begin{pmatrix}
1 & 0 & 0 & 0\\
0 & a & 0 & b\\
0 & 0 & 1 & 0\\
0 & c & 0 & d
\end{pmatrix}$ with some $
\begin{pmatrix}
a & b\\
c & d
\end{pmatrix}\in SL_2(\Z)$.
\end{itemize}
As a result of this 
\[
W_{\xi_0,\xi_2}(g)=W_{\xi,0}(
\begin{pmatrix}
1 & 0 & 0 & 0\\
0 & a & 0 & b\\
0 & 0 & 1 & 0\\
0 & c & 0 & d
\end{pmatrix}g)=\sum_{\xi_3\in\Z\setminus\{0\}}F_{\xi,\xi_3}(
\begin{pmatrix}
1 & 0 & 0 & 0\\
0 & a & 0 & b\\
0 & 0 & 1 & 0\\
0 & c & 0 & d
\end{pmatrix}
g)
\]
with some $\xi\in\Z\setminus\{0\}$ and $
\begin{pmatrix}
a & b\\
c & d
\end{pmatrix}\in SL_2(\Z)$. Consequently we have the description of $F_0$-term in the assertion.
\subsection*{(II)~$F_m$-term~($m\not=0$)}
Let $m\not=0$. For each fixed $g\in G$ we have $F_m(rg)\in\cH_m\otimes V_{\Lambda}^*$ as a function in $r\in G_J$. From \cite[Section 4.4,~Corollary 4.4.3]{Be-Sc} we recall that the orthogonal complement $\cH_m^{\rm c}$ of $\cH_m^0$ in $\cH_m$ is proved to be the continuous spectrum of $\cH_m$. 
When $F$ does not generate holomorphic or antiholomorphic discrete series, the $F_m$-term has the contribution from $\cH_m^{\rm c}$ as well as the contribution from $\cH_m^0$ in the Fourier-Jacobi expansion of $F$. We will denote the former~(respectively~the latter) by $F_m^{\rm c}$~(respectively~$F_m^0$). We shall first describe the $\cH_m^{\rm c}\otimes V_{\Lambda}^*$-part $F_m^{\rm c}$. 
\subsection*{(II-1)~The contribution from $\cH_m^{\rm c}$}
We shall prove that the contribution
\[
\sum_{
\begin{subarray}{c}
1\le\alpha\le 2|m|\\
\text{s.t.~$\alpha^2/4m\in\Z$}
\end{subarray}}\sum_{n\in\widehat{L_{\alpha,m}}\setminus\{0\}}b_{m,\alpha}(F)E_{\alpha}(F_{S_{\alpha,m},n}(*a_J))(r)
\]
to the Fourier-Jacobi expansion coincides with $F_m^{\rm c}(ra_J)$. 
To discuss this contribution we consider the Fourier expansion of $F_m$ along 
\[
N_S:=\left\{\left.
\begin{pmatrix}
1_2 & X\\
0_2 & 1_2
\end{pmatrix}~\right|~X\in{\rm Sym}_2(\R)\right\},
\]
whose summation is parametrized by the dual lattice of $\Sym_2(\Z)$ with respect to the trace, namely the set of semi-integral matrices of degree two. More precisely, the semi-integral matrices contributing to the expansion has the fixed upper-left entry $m$. 

Let us think of the terms parametrized by the singular matrices for the expansion of $F_m$ along $N_S$. 
As we have pointed out before the assertion of the theorem the singular semi-integral symmetric matrices with the upper-left entry $m$ is of the form
\[
S_{\alpha,m}(k)=
\begin{pmatrix}
m & km+\frac{\alpha}{2}\\
km+\frac{\alpha}{2} & m(k+\frac{\alpha}{2m})^2
\end{pmatrix}
\]
with some integer $k\in\Z$ and with $\alpha\in\Z$ such that $\dfrac{\alpha^2}{4m}\in\Z$, where $\alpha$ is determined module $2m$. 
The term $F_{S_{\alpha,m}(k)}(g)$ of $F_m(g)$ indexed by $S_{\alpha,m}(k)$ is defined as
\[
\int_{N_S(\Z)\backslash N_S}F(n(0,u_1,u_2,u_3)g)\e(-\tr(S_{\alpha,m}(k)
\begin{pmatrix}
u_1 & u_2\\
u_2 & u_3
\end{pmatrix}))du_1du_2du_3.
\]
We introduce
\[
E_{\alpha}(F_{S_{\alpha,m}(k)}(*g))(r):=\sum_{\gamma\in N_S(\Z)\backslash G_J(\Z)}F_{S_{\alpha,m}(k)}(\gamma rg).
\]
and recall that we have introduced $\gamma_{\alpha,m}(k)$ for $k\in\Z$, and put $\gamma_{\alpha,m}:=\gamma_{\alpha,m}(0)$ and $S_{\alpha,m}:=S_{\alpha,m}(0)$~(cf.~Section \ref{FJ-expansion}). 
We now state the following:
\begin{lem}\label{Jacobi-EisenPoincare}
For any integer $k$ we have
\begin{enumerate}
\item $E_{\alpha}(F_{S_{\alpha,m}(k)}(*g))(r)\in\cH_m^c\otimes V_{\Lambda}^*$,
\item $E_{\alpha}(F_{S_{\alpha,m}(k)}(*g))(r)=E_{\alpha}(F_{S_{\alpha,m}}(*g))(r)$.
\end{enumerate}
\end{lem}
\begin{proof}
We can verify $F_{S_{\alpha,m}(k)}(g)=F_{S_{\alpha,m}}(n(k,0,0,0)g)$ for $k\in\Z$. Once the first assertion is proved, we can say that this implies the second one. 
To show the first assertion we remark that, for each fixed $g\in G$, $F_{S_{\alpha,m}(k)}(\gamma_{\alpha,m}(k)^{-1}rg)$ belong to $L^2_m(V_J\backslash G_J)\otimes V_{\Lambda}^*$ as functions in $r\in G_J$, for $F_m$ is rapidly decreasing since so is the cusp form $F$. 
Here  $L^2_m(V_J\backslash G_J)$ denotes the space of square-integrable functions on $V_J\backslash G_J$ satisfying the right equivariance with respect to the central character of $N_J$ indexed by $m$. 
 Following \cite[Section 4.4]{Be-Sc} we introduce
\[
\sum_{\gamma\in Z_J(\Z)(G_J(\Z)\cap \gamma_{\alpha,m}^{-1}V_J\gamma_{\alpha,m})\backslash G_J(\Z)}F_{S_{\alpha,m}(k)}(\gamma_{\alpha,m}(k)^{-1}\gamma_{\alpha,m}\gamma\cdot rg)
\]
(for $Z_J(\Z)$ see Section \ref{Groups}). 
This is independent of $k\in\Z$ by $\gamma_{\alpha,m}(k)^{-1}\gamma_{\alpha,m}=n(-k,0,0,0)$ and $F_{S_{\alpha,m}(k)}(g)=F_{S_{\alpha,m}}(n(k,0,0,0)g)$ as above for $k\in\Z$. In fact, this coincides with 
\[
\sum_{\gamma\in Z_J(\Z)(G_J(\Z)\cap \gamma_{\alpha,m}^{-1}V_J\gamma_{\alpha,m})\backslash G_J(\Z)}F_{S_{\alpha,m}}(\gamma\cdot rg).
\]
Now let us note that we can verify that $Z_J(\Z)(G_J(\Z)\cap \gamma_{\alpha,m}^{-1}V_J\gamma_{\alpha,m})$ is a subgroup of finite index for $N_S(\Z)$ by a direct calculation of matrices. We therefore see that $E_{\alpha}(F_{S_{\alpha,m}}(*g))(r)$ is equal to
\[
\frac{1}{[N_S(\Z):Z_J(\Z)(G_J(\Z)\cap \gamma_{\alpha,m}^{-1}V_J\gamma_{\alpha,m})]}\sum_{\gamma\in Z_J(\Z)(G_J(\Z)\cap \gamma_{\alpha,m}^{-1}V_J\gamma_{\alpha,m})\backslash G_J(\Z)}F_{S_{\alpha,m}}(\gamma\cdot rg).
\]
By \cite[Section 4.4,~Corollary 4.4.3]{Be-Sc} we have $E_{\alpha}(F_{S_{\alpha,m}}(*g))(r)\in\cH_m^{\rm c}\otimes V_{\Lambda}^*$. 
\end{proof}
We now want to show that $F_m^c(rg)$ belongs to the finite dimensional $\C$-vector space spanned by $\{E_{\alpha}(F_{S_{\alpha,m}}(*g))(r)\mid 1\le \alpha\le 2|m|~\text{s.t.}~\frac{\alpha^2}{4m}\in\Z\}$. If this is settled we then see 
\[
F_m^c(rg)=\sum_{\begin{subarray}{c}
1\le\alpha\le 2m\\
\text{s.t.~$\frac{\alpha^2}{4m}\in\Z$}
\end{subarray}}b_{m,\alpha}(F)E_{\alpha}(F_{S_{\alpha,m}}(*g))(r)
\]
with constants $\{b_{m,\alpha}(F)\}$ depending on $m$ and $\alpha$. 
Here let us note that, for any $k\in\Z$, 
\begin{equation}\label{EisenPoincare-eq}
E_{\alpha}(F_{S_{\alpha,m}(k)}(*g))(r)=E_{\alpha}(F_{S_{\alpha,m}}(*g))(r),
\end{equation}
which is nothing but Lemma \ref{Jacobi-EisenPoincare} part 2. To see the assertion above we first note that the inner product $\displaystyle\int_{G_J(\Z)\backslash G_J}\langle F_m(rg),E_{\alpha}(F_{S_{\alpha,m}(k)}(*g))(r)\rangle dr$ 
with a $K$-invariant inner product $\langle *,*\rangle$ of $(\tau^*,V^*)$ is calculated to be
\begin{align*}
\int_{G_J(\Z)\backslash G_J}\sum_{N_S(\Z)\backslash G_J(\Z)}\langle F_m(\gamma rg),F_{S_{\alpha,m}(k)}(\gamma rg)\rangle dr
&=\int_{N_S(\Z)\backslash G_J}\langle F_m(rg),F_{S_{\alpha,m}(k)}(rg)\rangle dr\\
&=\int_{N_S\backslash G_J}\int_{N_S(\Z)\backslash N_S}\langle F_m(n\overset{\cdot}{r}g),F_{S_{\alpha,m}(k)}(nrg)\rangle dndr\\
&=\int_{N_S\backslash G_J}\langle F_{S_{\alpha,m}(k)}(rg),F_{S_{\alpha,m}(k)}(rg)\rangle dr\ge 0,
\end{align*}
where the notation $dr$~(respectively~$dn$) uniformly denotes the quotient measures of $G_J(\Z)\backslash G_J$, $N_S(\Z)\backslash G_J$ or $N_S\backslash G_J$ induced by an invariant measure of $G_J$~(respectively~the quotient measure of $N_S(\Z)\backslash N_S$ induced by an invariant measure of $N_S$). 
From this and (\ref{EisenPoincare-eq}) we deduce that $F_m(rg)$ is orthogonal to the $\C$-span of the $E_{\alpha}(F_{S_{\alpha,m}}(*g))(r)$s if and only if $F_m$ has no terms indexed by singular semi-integral matrices. 
If $F^c_m$ belongs to the orthogonal complement of such $\C$-span in $\cH_m^c$ we then have $F^c_m\equiv 0$ from Proposition \ref{Spectral-Jacobi} (2). Thereby $F_m^c$ is a linear combination of the $E_{\alpha}(F_{S_{\alpha,m}}(*g))(r)$s.

As the next step we remark that $F_{S_{\alpha,m}}(g)$ is equal to
\[
\underset{\gamma_{\alpha,m}N_S(\Z)\gamma_{\alpha,m}^{-1}\backslash N_S}{\displaystyle\int}F(\gamma_{\alpha,m}^{-1}n(0,u_1,u_2,u_3)\gamma_{\alpha,m}g)\e(-\tr(
\begin{pmatrix}
m & 0\\
0 & 0
\end{pmatrix}
\begin{pmatrix}
u_1 & u_2\\
u_2 & u_3
\end{pmatrix})du_1du_2du_3.
\]
We then consider the expansion of  $F_{S_{\alpha,m}}(g)$ along 
${}^tN_L(L_{\alpha,m})\backslash {}^tN_L\simeq\R/L_{\alpha,m}$~(for the notation $N_L$ and $L_{\alpha,m}$ see Sections \ref{Groups} and \ref{FJ-expansion} respectively), 
where ${}^{t}N_L(L_{\alpha,m}):=\{{}^tn(u_0,0,0,0)\mid u_0\in L_{\alpha,m}\}\subset{}^tN_L$. For $n\in \widehat{L_{\alpha,m}}\setminus\{0\}$ we put
\[
F_{S_{\alpha,m},n}(g):=\underset{\R/L_{\alpha,m}}{\displaystyle\int}F_{S_{\alpha,m}}(\gamma_{\alpha,m}^{-1}{}^tn(u_0,0,0,0)\gamma_{\alpha,m}g)\e(-nu_0)du_0,
\]
where recall that $du_0$ is the measure normalized so that $\vol(\R/L_{\alpha,m})=1$. 
We then have
\begin{equation}\label{ExpansionSingCoeff}
F_{S_{\alpha,m}}(g)=\sum_{n\in\widehat{L_{\alpha,m}}\setminus\{0\}}F_{S_{\alpha,m},n}(g).
\end{equation}
The following lemma implies that the explicit formulas for the Whittaker functions we have reviewed are useful to know $F_{S_{\alpha,m}}$ explicitly. 
\begin{lem}\label{Whittaker-F_m}
We have 
\[
F_{S_{\alpha,m},n}(({}^t\gamma_{\alpha,m}\xi)^{-1}g)\in W_{\psi_{n,m},\pi}(\tl^*)^0,
\] 
where recall that $\xi:=
\begin{pmatrix}
J'_2 & 0_2\\
0_2 & J'_2
\end{pmatrix}$ with $J'_2=
\begin{pmatrix}
0 & 1\\
1 & 0
\end{pmatrix}$. The function $F_{S_{\alpha,m},n}(({}^t\gamma_{\alpha,m}\xi)^{-1}g)$ is therefore a constant multiple of a Whittaker function for $\pi$ with $K$-type $\tl^*$ attached to $\psi_{n,m}$.
\end{lem}
\begin{proof}
We first note that $F(\xi g)=F(g)$ for $g\in G$ by the left $Sp(2,\Z)$-invariance of $F$. We thereby see that $F_{S_{\alpha,m}}(g)$ is equal to
\begin{align*}
&\underset{\gamma_{\alpha,m}N_S(\Z)\gamma_{\alpha,m}^{-1}\backslash N_S}{\displaystyle\int}F(\xi\gamma_{\alpha,m}^{-1}n(0,u_1,u_2,u_3)\gamma_{\alpha,m}g)\e(-\tr(
\begin{pmatrix}
m & 0\\
0 & 0
\end{pmatrix}
\begin{pmatrix}
u_1 & u_2\\
u_2 & u_3
\end{pmatrix}))du_1du_2du_3\\
=&\underset{\gamma_{\alpha,m}N_S(\Z)\gamma_{\alpha,m}^{-1}\backslash N_S}{\displaystyle\int}F(\xi\gamma_{\alpha,m}^{-1}\xi^{-1}\cdot\xi n(0,u_1,u_2,u_3)\xi^{-1}\cdot\xi\gamma_{\alpha,m}\xi^{-1}\cdot\xi g)\\
&\times\e(-\tr(J'_2
\begin{pmatrix}
m & 0\\
0 & 0
\end{pmatrix}{J'_2}^{-1}\cdot J'_2
\begin{pmatrix}
u_1 & u_2\\
u_2 & u_3
\end{pmatrix}{J'_2}^{-1})du_1du_2du_3\\
=&\underset{\gamma_{\alpha,m}N_S(\Z)\gamma_{\alpha,m}^{-1}\backslash N_S}{\displaystyle\int}F({}^t\gamma_{\alpha,m}^{-1}\cdot n(0,u_3,u_2,u_1)\cdot{}^t\gamma_{\alpha,m}\cdot\xi g)\\
&\times\e(-\tr(
\begin{pmatrix}
0 & 0\\
0 & m
\end{pmatrix}
\begin{pmatrix}
u_3 & u_2\\
u_2 & u_1
\end{pmatrix}))du_1du_2du_3.
\end{align*}
Let $N_0(L_{\alpha,m}):={}^t\gamma_{\alpha,m}(N_S(\Z)\rtimes N_L(L_{\alpha,m})){}^t\gamma_{\alpha,m}^{-1}$. 
From this we deduce $F_{S_{\alpha,m},n}(g)=$
\[
\underset{N_0(L_{\alpha,m})\backslash N_0}{\displaystyle\int}F(({}^t\gamma_{\alpha,m}\xi)^{-1}\cdot n(u_0,u_3,u_2,u_1)\cdot{}^t\gamma_{\alpha,m}\xi g)\psi_{n,m}(n(u_0,u_3,u_2,u_1))^{-1}du_0du_1du_2du_3,
\]
for $n\in\widehat{L_{\alpha,m}}\setminus\{0\}$, 
where note that 
\[
\xi n(u_0,0,0,0)\xi^{-1}={}^tn(u_0,0,0,0).
\]
We now verify the left equivariance of $F_{S_{\alpha,m},n}$ with respect to $\psi_{n,m}$, namely
\[
F_{S_{\alpha,m},n}(({}^t\gamma_{\alpha,m}\xi)^{-1}n(t_0,t_1,t_2,t_3)g)=\psi_{n,m}(n(t_0,t_1,t_2,t_3))F_{S_{\alpha,m},n}(({}^t\gamma_{\alpha,m}\xi)^{-1}g)
\]
for $(t_0,t_1,t_2,t_3)\in\R^4$. 
We therefore see that $F_{S_{\alpha,m},n}(({}^t\gamma_{\alpha,m}\xi)^{-1}g)\in W_{\psi_{n,m},\pi}(\tl^*)^0$, which is uniquely determined up to constant multiples.
\end{proof}
From (\ref{ExpansionSingCoeff}) we deduce 
\[
E_{\alpha}(F_{S_{\alpha,m}}(*g)(r)=\sum_{n\in\widehat{L_{\alpha,m}}\setminus\{0\}}E_{\alpha}(F_{S_{\alpha,m},n}(*g))(r),
\]
which can be regarded as a Fourier expansion along ${}^tN_L(L_{\alpha,m})\backslash {}^tN_L\simeq\R/L_{\alpha,m}$. 
This leads to
\[
F_m^{\rm c}(rg)=\sum_{\begin{subarray}{c}
1\le\alpha\le 2m\\
\text{s.t.~$\frac{\alpha^2}{4m}\in\Z$}
\end{subarray}}\sum_{n\in\widehat{L_{\alpha,m}}\setminus\{0\}}b_{m,\alpha}(F)E_{\alpha}(F_{S_{\alpha,m},n}(*g))(r).
\]
\subsection*{(II-2)~The contribution from $\cH_m^{0}$}
We now describe $F_m^0\in\cH_m^0\otimes V_{\Lambda}^*$. By Proposition \ref{Spectral-Jacobi} (1) we have known that $\cH_m^0$ decomposes discretely into irreducible pieces and each irreducible unitary representation $\rho_{m,\pi_1}$ occurs in $\cH_m^0$ with a finite multiplicity. 

Let us denote the representation space of $\rho_{m,\pi_1}$ by $\cH_{m,\pi_1}$. 
Every $\rho_{m,\pi_1}$-isotypic part of $F_m^0(rg)$ is a finite sum of $\phi_{\pi_1}^{(i)}(\cH_{m,\pi_1})\boxtimes V_{\Lambda}^*$-valued functions $F_{m,\pi_1}^{(i)}(rg)$ in $g\in G$ with $\phi^{(i)}$ ranging over a basis of $\Hom_{G_J}(\rho_{m,\pi_1},\cH_m^0)$, and each $F_{m,\pi_1}^{(i)}$ satisfies the left and right equivariance 
\[
F_{m,i}^{(\pi_1)}(rr_0gk)=R_{G_J}(r_0)\boxtimes{\tl^*(k)}^{-1}F_{m,i}^{(\pi_1)}(rg)\quad\forall(r,r_0,g,k)\in G_J\times G_J\times G\times K,
\]
where $R_{G_J}$ denotes the right regular representation of $G_J$ on $\phi^{(i)}(\cH_{m,\pi_1})$. 
Since $\phi^{(i)}$ intertwines $R_{G_J}$ with $\rho_{m,\pi_1}$ we can therefore say that $F_{m,i}^{(\pi_1)}(rg)$ is the image of  of the Fourier-Jacobi type spherical function of type $(\rho_{m,\pi_1},\pi;\tl^*)$ by the intertwining operator $\phi_{\pi_1}^{(i)}$. 
More precisely, up to constant multiples, $F_{m,i}^{(\pi_1)}(rr_0g)$ with $r,r_0\in G_J$ is written as
\begin{align*}
\sum_{
\begin{subarray}{c}
j\in J,~0\le k\le d_{\Lambda}\\
\text{s.t.}~l=l(j,k)\in L
\end{subarray}}\phi_{\pi_1}^{(i)}(c_{j,k}^{(\pi_1)}(g)w_l\otimes u_j^m)(rr_0)\otimes v_{k}^*
&=\sum_{
\begin{subarray}{c}
j\in J,~0\le k\le d_{\Lambda}\\
\text{s.t.}~l=l(j,k)\in L
\end{subarray}}\phi_{\pi_1}^{(i)}(\rho_{m,\pi_1}(r_0)(c_{j,k}^{(\pi_1)}(g)w_l\otimes u_j^m))(r)\otimes v_{k}^*,\\
&=\sum_{
\begin{subarray}{c}
j\in J,~0\le k\le d_{\Lambda}\\
\text{s.t.}~l=l(j,k)\in L
\end{subarray}}\phi_{\pi_1}^{(i)}(c_{j,k}^{(\pi_1)}(r_0g)w_l\otimes u_j^m)(r)\otimes v_{k}^*,
\end{align*}
with the notation in Sections \ref{Result-Hirano} and \ref{FJ-expansion}. 
For this let us note that the working assumption on the Fourier-Jacobi type spherical functions implies the uniqueness of these functions, up to scalars. 
From this we therefore obtain 
\[
F_m^0(ra_J)=\underset{\pi_1\in\widehat{\widetilde{SL}_2(\R)},m(\pi_1)\not=0}{\sum}\sum_{i=1}^{m(\pi_1)}b_{m,i}^{(\pi_1)}(F)
\underset{\begin{subarray}{c}
j\in J,~0\le k\le d_{\Lambda}\\
\text{s.t.}~l=l(j,k)\in L
\end{subarray}}{\sum}c_{j,k}^{(\pi_1)}(a_J)\phi_{\pi_1}^{(i)}(w_l\otimes u_j^m)(r)\otimes v_k^*
\]
as in the assertion by putting $g=a_J\in A_J$ and $r_0=1$.
As a result we have completed the proof of the theorem.
\begin{rem}
When $n\not=0$ the function $F_{S_{\alpha,m}(k),n}$ can be explicitly understood for the large discrete series, $P_J$-principal series and principal series of $G$ since explicit formulas for the Whittaker functions in $W_{\pi,\psi_{n,m}}(\tl^*)^0$ are available. On the other hand, the function $F_{S_{\alpha,m}(k)}$ is understood by the notion of ``Siegel-Whittaker functions''~(cf.~\cite{Is-1},~\cite{Go-Od},~\cite{Mi-1},~\cite{Mo-2}) for singular symmetric matrices. However, no explicit formula for them has not been available yet. 
The idea to expand $F_{S_{\alpha,m}}$ in terms of $F_{S_{\alpha,m},n}$ is inspired by Moriyama \cite[Section 1.2]{Mo-2}.
\end{rem}
\subsection{An explicit description of Jacobi cusp forms contributing to the $F_m$-terms for $m\not=0$}\label{Explicit-FJ}
We provide some explicit description of Jacobi cusp forms $\phi_{\pi_1}^{(i)}(w_l\otimes u_j^m)$. This together with the explicit Fourier-Jacobi type spherical functions~(cf.~Section \ref{Result-Hirano}) leads to some explicit (partial) description of the $F_m$-terms for $m\not=0$. For the case of holomorphic discrete series we are able to have a simpler description of the Fourier-Jacobi expansion with the help of the Fourier expansion along the minimal parabolic subgroup~(cf.~\cite{Nr-1}). We should remark that the holomorphic case deals with vector valued holomorphic cusp forms as well as scalar valued ones. 

To this end we prepare a couple of ingredients.
\begin{enumerate}
\item We introduce a subgroup $B:=\left\{\left.
\begin{pmatrix}
1 &  x\\
0 & 1
\end{pmatrix}
\begin{pmatrix}
\sqrt{y} & 0\\
0 & \sqrt{y}^{-1}
\end{pmatrix}~\right|~x\in\R,~y\in\R_{>0}\right\}$ of $SL_2(\R)$. 
As $SL_2(\R)$ is viewed as a subgroup of $G_J$ we can regard $B$ as a subgroup of $G_J$ by embedding it into the $SL_2(\R)$-factor of $G_J$. 
For this we note that there is a bijection 
\[
B\ni\begin{pmatrix}
1 &  x\\
0 & 1
\end{pmatrix}
\begin{pmatrix}
\sqrt{y} & 0\\
0 & \sqrt{y}^{-1}
\end{pmatrix}\mapsto x+\sqrt{-1}y\in{\frak h},
\]
where ${\frak h}$ denotes the complex upper half plane. 
Hereafter we put $z:=x+\sqrt{-1}y\in{\frak h}$. 
In addition we use the notation $n_J:=n(u_0,u_1,u_2)$ to denote an element of $N_J$. 
\item As a basis of $\cU_m\simeq L^2(\R)$ we take the Hermite functions 
\[
\{h_j(t)=e^{2\pi|m|t^2}\dfrac{d^j}{dt^j}(e^{-4\pi|m|t^2})=H_j(t)e^{-2\pi|m|t^2}\}_{j\ge 0},
\]
where $H_j(t)$ is the Hermite polynomial indexed by $j$. For this we recall that $j$ corresponds to ${\rm sign}(m)(j+\displaystyle\frac{1}{2})\in J$~(cf.~Section \ref{UnitaryRepJacobi}). 
With $h_j$ we note that, using (\ref{Explicit-Weil}), we have
\begin{align*}
\theta_{\alpha}(\omega_m(b)h_j)(n_J)=\sum_{k\in\Z}&y^{\frac{1}{4}}H_j(\sqrt{y}(u_0+k+\frac{\alpha}{2m}))e^{-2\pi |m| y(u_0+k+\frac{\alpha}{2m})^2}\\
&\times\e(m(u_0+k+\frac{\alpha}{2m})^2x+mu_1+(2km+\alpha)u_2).
\end{align*}

When $m$ is a positive integer and $j=0$ we can verify that this is the lift of the holomorphic theta series $\theta_{m,\alpha}$~(cf.~\cite[p58~(4)]{E-Z}) to a Jacob form on the group $G_J$. 
For the proof we use the argument analogous to \cite[Theorem 1.4]{E-Z}. More precisely we note that the factor of automorphy is replaced by that for Jacobi forms of half-integral weights in order to apply the argument of \cite[Theorem 1.4]{E-Z}.  In fact, $\theta_{m,\alpha}$ has the weight $\frac{1}{2}$ and index $m$. We will also need $\theta_{\alpha}(\omega_m(b)h_j)(n_J)$ for a negative index $m$. When $m<0$ and $j=0$ this is verified to be a Jacobi form on $G_J$ lifted from the anti-holomorphic theta series, which is a complex conjugate of a holomorphic theta series of the form $\theta_{-m,-\alpha}$. In what follows, we deal with Jacobi forms on ${\frak h}\times\C$ and cusp forms on ${\frak h}$ as automorphic forms on the groups $G_J$ and $SL_2(\R)$ respectively. 
\item Furthermore, from Theorem \ref{F-J-exp-mainthm} we recall that $\{\phi_{\pi_1}^{(i)}\}_{1\le i\le m(\pi_1)}$ has denoted a basis of $\Hom_{G_J}(\rho_{m,\pi_1}, \cH_m^0)$. 
Let $\{\sum_{1\le\alpha\le 2m}f^{(i)}_{\alpha}\theta_{\alpha}\}_{1\le i\le m(\pi_1)}$ be a basis of  $\cS_{\pi_1}(\widetilde{SL}_2(\Z);\Phi_m)$ with scalar-valued cusp forms $f^{(i)}_{\alpha}$ with respect to $\widetilde{\Gamma(4m)}$. We assume that 
\[
\sum_{1\le\alpha\le 2|m|}f^{(i)}_{\alpha}\theta_{\alpha}~\text{corresponds to}~\phi_{\pi_1}^{(i)}
\]
for each $i$ with $1\le i\le m(\pi_1)$ by the generalized Eichler-Zagier correspondence~(cf.~Theorem \ref{Eichler-ZagierCorresp}). 
The Jacobi cusp forms we are going to describe is $\phi_{\pi_1}^{(i)}(w_l\otimes h_j)(n_Jb)$ for some specified extremal index $(l,j)=(l(j_0,k_0),j_0)$ with some $k_0$ in $0\le k_0\le d_{\Lambda}$~(for the notation $l(j_0,k_0)$ see Section \ref{Result-Hirano}). It is expressed as a sum over $1\le\alpha\le2|m|$ of Jacobi forms denoted by $\phi_{m,\alpha}^{(i)}(\pi_1)(n_Jb)$.
\end{enumerate}
\subsection*{(I)~The case of holomorphic discrete series representation}
Let $\pi_{\lambda}$ be a holomorphic discrete series representation with Harish Chandra parameter $\lambda\in\Xi_{I}$. 
From Theorem \ref{FJ-holom-antiholom-DS} we recall that $\dim\cJ_{\rho_{\pi_1,m},\pi_{\lambda}}(\tl^*)^0=1$ if and only if $m>0$ and $\pi_1=\cD_{n_1}^+$ with $\lambda_2+\frac{3}{2}\le n_1\le\lambda_1+\frac{1}{2}$, 
where recall that $\tl^*$ denotes the contragredient of the minimal $K$-type $\tl$ of $\pi_{\lambda}$. 
We explicitly describe Jacobi forms appearing in the Fourier-Jacobi expansion of cusp forms generating $\pi_{\lambda}$ of weight $\tl^*$~(cf.~Corollary \ref{FJ-exp-fourcases} (1)), which are nothing but holomorphic Siegel cusp forms including vector valued ones.
\subsection*{(I-1)~An example related to Corollary \ref{FJ-exp-fourcases} (1)}
Let $\cS_{\pi_1}(\widetilde{SL}_2(\Z),\Phi_m)$ be the space of $\Phi_m$-valued holomorphic cusp forms of half integral weight $n_1$, where $\lambda_2+\frac{3}{2}\le n_1\le \lambda_1+\frac{1}{2}$. We put $j_0:=\frac{1}{2}$ and $k_0:=n_1-\lambda_2-\frac{3}{2}$. Then we have $l(\frac{1}{2},k_0)=n_1$. 

For each $i$ with $1\le i\le m(\cD_{n_1}^+)$ we see that $c_{\frac{1}{2},k_0}^{(\cD_{n_1}^+)}(a_J)\phi_{\cD_{n_1}^+}^{(i)}(w_{n_1}\otimes h_0)(n_Jb)=$
\[
a_1^{\lambda_1+\lambda_2-n_1+\frac{5}{2}}e^{-2\pi ma_1^2}\sum_{1\le\alpha\le 2m}f_{\alpha}^{(i)}(b)\theta_{\alpha}(\omega_m(b)h_0)(n_J)
\]
for $(n_J,b)\in N_J\times B$, up to constant multiples, where $\sum_{1\le\alpha\le 2m}f_{\alpha}^{(i)}\theta_{\alpha}\in\cS_{\pi_1}(\widetilde{SL}_2(\Z),\Phi_m)$ with holomorphic cusp forms $f_{\alpha}^{(i)}$ with respect to $\widetilde{\Gamma(4m)}$. We recall that holomorphic cusp forms $f_{\alpha}^{(i)}$ have the Fourier expansion
\[
f_{\alpha}^{(i)}(b)=\sum_{l>0,~l\in\frac{1}{4m}\Z}c_{\alpha}^{(i)}(l)y^{n_1}q^l~(q=\e(x+\sqrt{-1}y))
\]
with Fourier coefficients $\{c_{\alpha}^{(i)}(l)\}$. 

For $i$ with $1\le i\le m(\cD_{n_1}^+)$ we introduce
\[
\phi_{m,\alpha}^{(i)}(\cD_{n_1}^+)(n_Jb):=\sum_{
\begin{subarray}{c}
t\in\Z\\
l\in\frac{1}{4m}\Z\setminus\{0\}
\end{subarray}}c_{\alpha}^{(i)}(l)y^{\frac{n_1}{2}+\frac{1}{4}}e^{-2\pi my(t+\frac{\alpha}{2m})^2}\e(m(u_0^2z)+(2mt+\alpha)u_0z)\exp(-2\pi ly)\e(\Tr(T_{t,l}X),
\]
where $T_{t,l}=
\begin{pmatrix}
m & \frac{2mt+\alpha}{2}\\
\frac{2mt+\alpha}{2} & m(t+\frac{\alpha}{2m})^2+l
\end{pmatrix}$ and $X=
\begin{pmatrix}
u_1 & u_2\\
u_2 & x
\end{pmatrix}$. We then have
\[
c_{\frac{1}{2},k_0}^{(\cD_{n_1}^+)}(a_J)\phi_{\cD_{n_1}^+}^{(i)}(w_{n_1}\otimes h_0)(n_Jb)=
a_1^{\lambda_1+\lambda_2-n_1+\frac{5}{2}}e^{-2\pi ma_1^2}\sum_{1\le \alpha\le 2m}\phi_{m,\alpha}^{(i)}(\cD_{n_1}^+)(n_Jb).
\]
The form $\sum_{1\le \alpha\le 2m}\phi_{m,\alpha}^{(i)}(\cD_{n_1}^+)(n_Jb)$ is a holomorphic Jacobi cusp forms on $G_J$ with respect to $G_J(\Z)$.
\subsection*{(I-2)~Another viewpoint in terms of the Fourier expansion along the minimal parabolic subgroup}
For general indexes $(j,k,l)$ with $j\in J$,~$0\le k\le d_{\Lambda}$ and $l\in L_{\cD_{n_1}}$~(for $L_{\cD_{n_1}}$ see (\ref{index-genuine})) it is possible to write $c_{j,k}^{(\cD_{n_1}^+)}(a_J)\phi_{\cD_{n_1}^+}^{(i)}(w_l\otimes h_j)(n_Jb)$ more explicitly. It is obtained inductively from $c_{\frac{1}{2},k_0}(a_J)\phi_{\cD_{n_1}^+}^{(i)}(w_{n_1}\otimes h_0)(n_Jb)$ by the recurrence relations \cite[(6.4),~(6.5)]{Hi-1}, the differential operators $V^{\pm}$~(cf.~(\ref{Infinitesimal-MaassOp})) and the differential operators raising the index of $h_j$, where note that $V^{\pm}$ correspond to the Maass weight raising or lowering operators in the classical setting~(cf.~\cite[Chapter 3,~Section 3.2]{Bu}). 
However, with this approach, to obtain a further more explicit  description of the Fourier-Jacobi expansion in Corollary \ref{FJ-exp-fourcases} (1) we cannot avoid a very complicated computation especially for vector valued holomorphic cusp forms.

Instead we use the theory of the Fourier expansion along the minimal parabolic subgroup~(cf.~\cite{Nr-1}). 
From \cite[Theorem 9.6,~Section 10]{Nr-1}, we can deduce another expression of the Fourier expansion of $F$. Using the coordinate $n(u_0,u_1,u_2,u_3)a_0$ for $N_0A_0$~(cf.~Section \ref{Groups}) we have $F(n(u_0,u_1,u_2,u_3)a_0)=$
\[
\sum_{T=
\left(
\begin{smallmatrix}
t_1 & t_2/2\\
t_2/2 & t_3
\end{smallmatrix}\right)>0}\sum_{k=0}^{d_{\Lambda}}C_T(F)a_1^{\Lambda_1-k}a_2^{\Lambda_2+k}(u_0+t_2/2t_1)^k\e(\Tr(T
\begin{pmatrix}
(a_1^2+a_2^2u_0^2)\sqrt{-1}+u_1 & a_2^2u_0\sqrt{-1}+u_2\\
a_2^2u_0\sqrt{-1}+u_2 & a_2^2\sqrt{-1}+u_3
\end{pmatrix}))v_k^*,
\]
where $T$ ranges over the set of positive definite semi-integral matrices of degree two and $C_T(F)$ denotes the Fourier coefficient for $T$. 
For this we recall that $\Lambda=(\Lambda_1,\Lambda_2)=(\lambda_1+1,\lambda_2+2)$ has denoted the Blattner parameter of the holomorphic discrete series $\pi_{\lambda}$ with $\lambda=(\lambda_1,\lambda_2)\in\Xi_I$~(cf.~Section \ref{DS-rep}). 
This is noting but a reformulation of the classical Fourier expansion of $F$ by the coordinate of $N_0A_0$. 
Following the well known manner as in \cite[Theorem 6.1]{E-Z}, we have the Fourier-Jacobi expansion
\[
F(n(u_0,u_1,u_2,u_3)a_0)=\sum_{m\in\Z_{>0}}\sum_{k=0}^{d_{\Lambda}}a_1^{\Lambda_1-k}\e(m(a_1^2\sqrt{-1}+u_1))
\phi^{(k)}_m(u_0,u_2,u_3,a_2)v_k^*
\]
from the classical expansion above, where $\phi_m^{(k)}(u_0,u_2,u_3,a_2):=$
\[
\sum_{\begin{subarray}{c}
t_2,t_3\in\Z\\
4mt_3-t_2^2>0
\end{subarray}}C_{\left(
\begin{smallmatrix}
m & t_2/2\\
t_2/2 & t_3
\end{smallmatrix}\right)}(F)a_2^{\Lambda_2+k}(u_0+t_2/2m)^k\e(m(a_2^2u_0^2\sqrt{-1})+t_2(a_2^2u_0\sqrt{-1}+u_2)+
t_3(a_2^2\sqrt{-1}+u_3)).
\]
The function $\e(mu_1)\phi_m^{(k)}(u_0,u_2,u_3,a_2)$ is a cuspidal Jacobi form on $G_J$, whose left $G_J(\Z)$-invariance follows from that of $F$. 

We now put 
\[
h_{\alpha}(b):=\sum_{N>0}c_{\alpha}(N)\sqrt{y}^{\Lambda_2-\frac{1}{2}}q^{\frac{N}{4m}}
\]
with
\[
c_{\alpha}(N):=
\begin{cases}
C_{\left(
\begin{smallmatrix}
m & t_2/2\\
t_2/2 & \frac{N+t_2^2}{4m}
\end{smallmatrix}\right)}(F)&(t_2\equiv \alpha\mod 2m,~N\equiv -\alpha^2\mod 4m)\\
0&(\text{otherwise})
\end{cases},
\]
and replace the coordinate $n(u_0,u_1,u_2,u_3)a_0$ by $n_Jn(0,0,0,x)\diag(a_1,\sqrt{y},a_1^{-1},\sqrt{y}^{-1})$, which coincides with $n(u_0,u_1+u_0^2x,u_2+u_0x,x)\diag(a_1,\sqrt{y},a_1^{-1},\sqrt{y}^{-1})$. 
Let us apply the theta decomposition~(cf.~\cite[Section 5]{E-Z}) to $\phi_m^{(k)}$ and obtain
\[
\e(m(u_1+u_0^2x))\phi_m^{(k)}(u_0,u_2+u_0x,x,\sqrt{y})=\sum_{1\le \alpha\le 2m}h_{\alpha}(b)\theta_{\alpha}(\omega(b)\eta_k)(n_J)
\]
with a function $\eta_k(t):=t^{k}\exp(-2\pi myt^2)\in L^2(\R)$ in $t$. 
We should remark that $h_{\alpha}$ is independent of $k$. In addition, for the definition of $c_{\alpha}(N)$, we remark that it depends only on $t_2\mod 2m$ and $\det(T)=N$. This is verified by the well known property $C_{{}^t\gamma T\gamma}(F)=C_T(F)$ for $\gamma\in SL_2(\Z)$, which follows from the left $Sp(2,\Z)$-invariance of $F$. The $G_J(\Z)$-invariance of $\e(mu_0)\phi_m^{(0)}(u_0,u_2,u_3,a_2)$ implies that 
\[
\sum_{1\le \alpha\le 2m}h_{\alpha}(b)\theta_{\alpha}\in\cS_{\cD_{\Lambda_2-\frac{1}{2}}^+}(\widetilde{SL}_2(\Z);\Phi_m),
\]
which follows from the argument as in the former part of the proof for Theorem \ref{Eichler-ZagierCorresp}. 
To explain this in more detail we claim that $h_{\alpha}$ is proved to be holomorphic by its Fourier expansion and its weight is verified to be $\Lambda_2-\frac{1}{2}$ since the weight of $v_0^*$ with respect to $\sqrt{-1}
T_2=\frac{1}{2}(H-Z)$ is $\Lambda_2$~(cf.~Section \ref{RepMaxCpt}) and $\theta_{\alpha}(\omega_m(b)\eta_0)(n_J)~(=\theta_{m,\alpha})$ is of weight $\frac{1}{2}$.  

It is curious that the present argument needs only a single $\cS_{\cD_{\Lambda_2-\frac{1}{2}}^+}(\widetilde{SL}_2(\Z);\Phi_m)$ to describe the $F_m$-term of the Fourier-Jacobi expansion of $F$ while the description of $F_m$ by Corollary \ref{FJ-exp-fourcases} (1) needs $\cS_{\cD_{n_1}}^+(\widetilde{SL}_2(\Z),\Phi_m)$ for $n_1$ with $\lambda_2+\frac{3}{2}\le n_1\le \lambda_1+\frac{1}{2}$. 
Here we note that $\Lambda_2-\frac{1}{2}=\lambda_2+\frac{3}{2}$.

We can verify that $\e(m(u_1+u_0^2x))\phi_m^{(0)}(u_0,u_2+u_0x,x,\sqrt{y})$ is a holomorphic Jacobi cusp form on $G_J$ for the case of $n_1=\lambda_2+\frac{3}{2}$, given in (I-1). 
Now let us introduce the root vectors $E_{e_1\pm e_2}$  of the restricted roots $e_1\pm e_2$ for the Lie algebra of $A_0$~(cf.~\cite[Section 1.4]{Nr-1}). Then we further remark that $\phi_m^{(k)}$ is inductively obtained from $\phi_m^{(0)}$ by the recurrence relation
\[
\phi_m^{(k+1)}=-\frac{1}{8\pi m a_1}(dR(E_{e_1-e_2})-4\pi\sqrt{-1}dR(E_{e_1+e_2}))\phi_m^{(k)}-k\phi_m^{(k-1)},~\phi_m^{(-1)}\equiv 0\quad(0\le k< d_{\Lambda}),
\]
where $dR$ denotes the differential of the right translation $R$ of $G_J$. 
For this we note that the differential operator $\frac{1}{a_1}(dR(E_{e_1-e_2})-4\pi\sqrt{-1}dR(E_{e_1+e_2}))$ coincides with the infinitesimal action $a_2^{-1}d\eta(E_{e_1-e_2})+4\pi a_2\sqrt{-1}d\eta(E_{e_1+e_2})$~(cf.~\cite[p569]{Nr-1}), where $d\eta$ denotes the differential of the ``generic representation'' $\eta$ of $N_0$~(cf.~\cite[Remark 4.4 (1)]{Nr-1}) with the parameter $\xi_1=m$.
\begin{rem}
(1)~The theory of the Fourier expansion along the minimal parabolic subgroup is known only for holomorphic modular forms on tube domains~(cf.~\cite{Nr-1},~\cite{Nr-2},~\cite{Nr-3}). For non-holomorphic cases there is completely no such theory.\\
(2)~In the classical setting  Ibukiyama-Kyomura \cite{Ib-Ky} discussed the notion of vector valued Jacobi forms  which appear in the Fourier-Jacobi expansions of holomorphic vector valued Siegel modular forms of degree two.  They constructed a differential operator inducing a linear isomorphism between a space of the aforementioned vector valued Jacobi forms and a direct sum of some spaces of holomorphic scalar valued Jacobi forms. This work yields another way to know that the Fourier-Jacobi coefficients of holomorphic vector valued Siegel modular forms are the functions obtained as the translations of holomorphic scalar valued Jacobi forms by differential operators.  
\end{rem}
\subsection*{(II)~The case of large discrete series representations with Harish Chandra parameters in $\Xi_{II}$}
Let $\pi_{\lambda}$ be a large discrete series representation with Harish Chandra parameter $\lambda\in\Xi_{II}$. From Theorems \ref{explicit-FJ-I},~\ref{explicit-FJ-II} we recall that $\dim\cJ_{\rho_{\pi_1,m},\pi_{\lambda}}(\tl^*)^0=1$ if and only if 
\begin{itemize}
\item $m>0$ and $\pi_1=\cP_s^{\tau}~(\tau=\pm\frac{1}{2},~s\in\sqrt{-1}\R)$ or $\cC_s^{\tau}~(\tau=\pm\frac{1}{2},~0<s<\frac{1}{2})$,
\item $m>0$ and $\pi_1=\cD_{n_1}^-$ with $n_1\in\frac{1}{2}\Z_{\ge 3}\setminus\Z$ and $n_1\le-\lambda_2+\frac{1}{2}$,
\item $m<0$ and $\pi_1=\cD_{n_1}^+$ with $n_1\in\frac{1}{2}\Z_{\ge 3}\setminus\Z$ and $n_1>\lambda_1+\frac{1}{2}$,
\end{itemize}
where recall that $\tl$ denotes the minimal $K$-type of $\pi_{\lambda}$. The examples below for this case appear in the Fourier-Jacobi expansions of cusp forms generating $\pi_{\lambda}$ of weight $\tl^*$~(cf.~Corollary \ref{FJ-exp-fourcases} (2)). 
It is worthwhile to remark that Jacobi forms in the second and third cases (c.f~(II-2),~(II-3)) are skew-holomorphic Jacobi cusp forms.
\subsection*{(II-1)~The case of $\lambda\in\Xi_{II}$,~$\pi_1=\cP_s^{\tau}$ and $m>0$}
For this case $\cS_{\pi_1}(\widetilde{SL}_2(\Z);\Phi_m)$ can be referred to as the space of $\Phi_m$-valued Maass cusp forms of weight $\tau\in\{\pm\frac{1}{2}\}$. Let $\tau=-\frac{1}{2}$ and $j_0=\frac{1}{2}$ and take $k_0$ so that $l(\frac{1}{2},k_0)=-\frac{1}{2}$, then $k_0=-\lambda_2\ge0$. 

For each $i$ with $1\le i\le m(\cP_s^{-1/2})$ we see that $c_{\frac{1}{2},-\lambda_2}^{(\cP_s^{-1/2})}(a_J)\phi_{\cP_s^{\tau}}^{(i)}(w_{-\frac{1}{2}}\otimes h_0)(n_Jb)$
\[
=e^{2\pi ma_1^2}G_{2,3}^{3,0}\left(4\pi m a_1^2\left|
\begin{array}{c}
\dfrac{2s+5+2d_{\Lambda}}{4},\dfrac{-2s+5+2d_{\Lambda}}{4}\\
\dfrac{\lambda_1+2}{2},\dfrac{\lambda_1+3}{2},\dfrac{d_{\Lambda}+2}{2}
\end{array}\right.\right)\sum_{1\le \alpha\le 2m}f_{\alpha}^{(i)}(b)\theta_{\alpha}(\omega_m(b)h_0)(n_J)
\]
for $(n_J,b)\in N_J\times B$.
 Each cusp form $f^{(i)}_{\alpha}$ has a Fourier expansion
\[
f_{\alpha}^{(i)}(b)=\sum_{l\not=0,~l\in\frac{1}{4m}\Z}c_{\alpha}^{(i)}(l)W_{-{\rm sign}(l)\frac{1}{4},s}(4\pi|l|y)\e(lx)
\]
with Fourier coefficients $\{c_{\alpha}^{(i)}(l)\}$ (for $W_{\kappa,\mu}$ see \cite[Chapter 16]{W-W} and Theorem \ref{LargeDSWhittaker}). 
By solving the differential equation arising from the infinitesimal action of the Casimir operator, $W_{-{\rm sign}(l)\frac{1}{4},s}(4\pi|l|y)\e(lx)$ is verified to be a Whittaker function on $\widetilde{SL}_2(\R)$ for $\cP_s^{-1/2}$ with respect to the additive character $\e(lx)$ indexed by $l$. As a reference of the Fourier expansion we cite \cite[Section 2]{Sar}. For this we note that the parameter $s$ is replaced by $s-\frac{1}{2}$ in \cite[Section 2]{Sar}. The difference of the parameter $s$ is understood by noting that $\cP_s^{-1/2}$ is realized by the normalized parabolic induction. 

For $i$ with $1\le i\le m(\cP_s^{-1/2})$ we introduce 
\[
\phi_{m,\alpha}^{(i)}(\cP_s^{-1/2})(n_Jb):=f_{\alpha}^{(i)}(b)\theta_{\alpha}(\omega_m(b)h_0)(n_J).
\]
This can be written as
\[
\sum_{
\begin{subarray}{c}
t\in\Z\\
l\in\frac{1}{4m}\Z\setminus\{0\}
\end{subarray}}c_{\alpha}^{(i)}(l)y^{\frac{1}{4}}e^{-2\pi my(t+\frac{\alpha}{2m})^2}\e(m(u_0^2z)+(2tm+\alpha)u_0z)W_{-{\rm sign}(l)\frac{1}{4},s}(4\pi|l|y)\e(\Tr(T_{t,l}X),
\]
where $T_{t,l}=
\begin{pmatrix}
m & \frac{2tm+\alpha}{2}\\
\frac{2tm+\alpha}{2} & m(t+\frac{\alpha}{2m})^2+l
\end{pmatrix}$ and $X=
\begin{pmatrix}
u_1 & u_2\\
u_2 & x
\end{pmatrix}$. We then have 
\begin{align*}
&c_{\frac{1}{2},-\lambda_2}^{(\cP_s^{-1/2})}(a_J)\phi_{\cP_s^{-1/2}}^{(i)}(w_{-\frac{1}{2}}\otimes h_0)(n_Jb)=\\
&e^{2\pi ma_1^2}G_{2,3}^{3,0}\left(4\pi ma_1^2\left|
\begin{array}{c}
\dfrac{2s+5+2d_{\Lambda}}{4},\dfrac{-2s+5+2d_{\Lambda}}{4}\\
\dfrac{\lambda_1+2}{2},\dfrac{\lambda_1+3}{2},\dfrac{d_{\Lambda}+2}{2}
\end{array}\right.\right)\sum_{1\le \alpha\le 2m}\phi_{m,\alpha}^{(i)}(\cP_{s}^{-1/2})(n_Jb).
\end{align*}
The form $\sum_{1\le \alpha\le 2m}\phi_{m,\alpha}^{(i)}(\cP_{s}^{-1/2})(n_Jb)$ can be referred to as a Maass Jacobi cusp form on $G_J$.
\subsection*{(II-2)~The case of $\lambda\in\Xi_{II}$,~$\pi_1=\cD^-_{n_1}$ and $m>0$}
For this case $\cS_{\pi_1}(\widetilde{SL}_2(\Z),\Phi_m)$ is the space of $\Phi_m$-valued anti-holomorphic cusp forms of half  integral weight $-n_1$, where $3/2\le n_1\le-\lambda_2+\frac{1}{2}$. Let us take $j_0=\frac{1}{2}$ and $k_0=-n_1+\frac{1}{2}-\lambda_2$. Then $l(j_0,k_0)=-n_1$. For each $i$ with $1\le i\le m(\cD_{n_1}^-)$ we see that $c_{\frac{1}{2},k_0}^{(\cD_{n_1}^-)}(a_J)\phi_{\cD_{n_1}^-}^{(i)}(w_{-n_1}\otimes h_0)(n_Jb)=$
\[
(2\pi ma_1^2)^{\frac{3+2\lambda_1}{4}}W_{-\frac{1}{4}(1+2k_0),-\frac{1}{4}}(2\pi ma_1^2)\sum_{1\le\alpha\le 2m}f_{\alpha}^{(i)}(b)\theta_{\alpha}(\omega_m(b)h_0)(n_J)
\]
for $(n_J,b)\in N_J\times B$, where $\sum_{1\le\alpha\le 2m}f_{\alpha}^{(i)}\theta_{\alpha}\in\cS_{\cD_{n_1}^{-}}(\widetilde{SL}_2(\Z),\Phi_m)$ with anti-holomorphic cusp forms $f_{\alpha}^{(i)}$ with respect to $\widetilde{\Gamma(4m)}$. We note that each  $f_{\alpha}^{(i)}$ has a Fourier expansion
\[
f_{\alpha}^{(i)}(b)=\sum_{l>0,~l\in\frac{1}{4m}\Z}c_{\alpha}^{(i)}(l)y^{n_1/2}\bar{q}^l
\]
with Fourier coefficients $\{c_{\alpha}^{(i)}(l)\}$, and that $\theta_{\alpha}(\omega_m(b)h_0)(n_J)$ has been  described in the beginning of this section. Let 
\[
\phi_{m,\alpha}^{(i)}(\cD_{n_1}^-)(n_Jb):=f_{\alpha}^{(i)}(b)\theta_{\alpha}(\omega_m(b)h_0)(n_J),
\]
which can be written as
\[
\sum_{
\begin{subarray}{c}
t\in\Z\\
l\in\frac{1}{4m}\Z_{<0}
\end{subarray}}c_d^{(i)}(-l)y^{\frac{n_1}{2}+\frac{1}{4}}e^{-2\pi my(t+\frac{\alpha}{2m})^2}\e(m(u_0^2z)+(2mt+\alpha)u_0z)\exp(2\pi ly)\e(\Tr(T_{t,l}X)).
\]
We then have $c_{\frac{1}{2},k_0}^{(\cD_{n_1}^-)}(a_J)\phi_{\cD_{n_1}^-}^{(i)}(w_{-n_1}\otimes h_0)(n_Jb)=$
\[
(2\pi ma_1^2)^{\frac{3+2\lambda_1}{4}}W_{-\frac{1}{4}(1+2k_0),-\frac{1}{4}}(2\pi ma_1^2)\sum_{1\le \alpha\le 2m}\phi_{m,\alpha}^{(i)}(\cD_{n_1}^-)(n_Jb).
\]
The sum $\sum_{1\le\alpha\le 2m}\phi_{m,\alpha}^{(i)}(\cD_{n_1}^-)(nb)$ is a skew-holomorphic Jacobi cusp form on $G_J$ since it is a sum of products of anti-holomorphic cusp forms and holomorphic Jacobi theta series.\\ 
\subsection*{(II-3)~The case of $\lambda\in\Xi_{II},~\pi_1=\cD_{n_1}^+$ and $m<0$}
The space $\cS_{\pi_1}(\widetilde{SL}_2(\Z),\Phi_m)$ is the space of $\Phi_m$-valued holomorphic cusp forms of half integral weight $n_1$, where $n_1>\lambda_1+\frac{1}{2}$. We put $j_0=-\frac{1}{2}$ and $k_0:=n_1-\lambda_2-\frac{1}{2}$. Then $l(-\frac{1}{2},k_0)=n_1$. For this we note that $k_0>\lambda_1+\frac{1}{2}-\lambda_2-\frac{1}{2}=\lambda_1-\lambda_2=d_{\Lambda}-1$. Thereby $k_0=d_{\Lambda}$ is the only one possible choice of $k_0$. With this choice $n_1=\lambda_1+\frac{3}{2}>\frac{3}{2}$. 

For each $i$ with $1\le i\le m(\cD_{n_1}^+)$ we see that $c_{-\frac{1}{2},d_{\Lambda}}^{(\cD_{n_1}^+)}(a_J)\phi_{\cD_{n_1}^+}^{(i)}(w_{n_1}\otimes h_0)(n_Jb)=$
\[
e^{-2\pi ma_1^2}G_{2,3}^{3,0}\left(-4\pi ma_1^2\left|
\begin{array}{c}
\dfrac{2\lambda_1+6}{4},\dfrac{2\lambda_1+8}{4}\\
\dfrac{\lambda_1+2}{2},\dfrac{\lambda_1+3}{2},\dfrac{-\lambda_2+2}{2}
\end{array}\right.\right)\sum_{1\le\alpha\le 2|m|}f_{\alpha}^{(i)}(b)\theta_{\alpha}(\omega_m(b)h_0)(n_J)
\]
for $(n_J,b)\in N_J\times B$, where $\sum_{1\le \alpha\le 2|m|}f_{\alpha}^{(i)}\theta_{\alpha}\in \cS_{\pi_1}(\widetilde{SL}_2(\Z),\Phi_m)$ with holomorphic cusp forms $f_{\alpha}^{(i)}$ with respect to $\widetilde{\Gamma(4m)}$. 
Each $f_{\alpha}^{(i)}$ has the Fourier expansion
\[
f_{\alpha}^{(i)}(b)=\sum_{l>0,~l\in\frac{1}{4m}\Z}c_{\alpha}^{(i)}(l)y^{n_1/2}q^l
\]
with Fourier coefficients $\{c_{\alpha}^{(i)}(l)\}$. 
Put 
\[
\phi_{m,\alpha}^{(i)}(\cD_{n_1}^+)(n_Jb):=f_{\alpha}^{(i)}(b)\theta_{\alpha}(\omega_m(b)h_0)(n_J),
\]
which coincides with 
\[
\sum_{
\begin{subarray}{c}
t\in\Z\\
l\in\frac{1}{4m}\Z_{<0}
\end{subarray}}c_{\alpha}^{(i)}(l)y^{\frac{n_1}{2}+\frac{1}{4}}e^{2\pi my(t+\frac{\alpha}{2m})^2}\e(m(u_0^2\bar{z})+(2mt+\alpha)u_0\bar{z})\exp(-2\pi ly)\e(\Tr(T_{t,l}X).
\]
Then we have $c_{-\frac{1}{2},d_{\Lambda}}^{(\cD_{n_1}^+)}(a_J)\phi_{\cD_{n_1}^+}^{(i)}(w_{n_1}\otimes h_0)(n_Jb)=$
\[
e^{-2\pi ma_1^2}G_{2,3}^{3,0}\left(-4\pi ma_1^2\left|
\begin{array}{c}
\dfrac{2\lambda_1+6}{4},\dfrac{2\lambda_1+8}{4}\\
\dfrac{\lambda_1+2}{2},\dfrac{\lambda_1+3}{2},\dfrac{-\lambda_2+2}{2}
\end{array}\right.\right)\sum_{1\le\alpha\le 2|m|}\phi_{m,\alpha}^{(i)}(\cD_{n_1}^+)(n_Jb).
\]
The sum $\sum_{1\le\alpha\le 2m}\phi_{m,\alpha}^{(i)}(\cD_{n_1}^+)(n_Jb)$ is a skew-holomorphic Jacobi cusp form on $G_J$ since it is a sum of products of holomorphic cusp forms and anti-holomorphic Jacobi theta series.
\subsection*{(III)~The case of $P_J$-principal series representations with $\sigma=(\cD_{n}^+,\epsilon)$}
Let $\pi$ be an irreducible $P_J$-principal series representation with $\sigma=(\cD_n^+,\epsilon)$ and let $\tl$ be the corner $K$-type of $\pi$. 
From Theorems \ref{explicit-FJ-PJ-I},~\ref{explicit-FJ-PJ-II},~\ref{explicit-FJ-PJ-III} and \ref{explicit-FJ-PJ-IV} we recall that there is a possibility of $\dim\cJ_{\rho_{\pi_1,m},\pi}(\tl^*)^0=1$ only if
\begin{itemize}
\item $m>0$ and $\pi_1=\cP_s^{\tau}~(\tau=\pm\frac{1}{2},~s\in\sqrt{-1}\R)$, $\cC_s^{\tau}~(\tau=\pm\frac{1}{2},~0<s<\frac{1}{2})$ or $\cD_{\frac{1}{2}}^-$ with $n=1$,
\item $m>0$ and $\pi_1=\cD_{n_1}^+$ with $n_1\in\frac{1}{2}\Z_{\ge 3}\setminus\{0\}$ and $n_1\le n-\frac{1}{2}$,
\item $m<0$ and $\pi_1=\cD_{n_1}^+$ with $n_1\in\frac{1}{2}\Z_{\ge 3}\setminus\Z$ and $n_1>n-\frac{1}{2}$.
\end{itemize}
Throughout the case (III) we let 
\[
\delta:=
\begin{cases}
0&(\text{when the corner $K$-type is one dimensional})\\
1&(\text{when the corner $K$-type is two dimensional})
\end{cases}.
\] 
The examples below appear in the Fourier-Jacobi expansions of cusp forms generating $\pi$ of wight $\tl^*$~(cf.~Corollary \ref{FJ-exp-fourcases} (3)). 
\subsection*{(III-1) The case of $\pi_1=\cP_s^{\tau}$ and $m>0$}
As in the case of (II-1) $\cS_{\pi_1}(\tilde{SL}_2(\Z);\Phi_m)$ is the space of $\Phi_m$-valued Maass cusp forms of weight $\tau\in\{\pm\displaystyle\frac{1}{2}\}$. 
Suppose that $\tau=
\begin{cases}
-\displaystyle\frac{1}{2}&(\text{$n$:even})\\
\displaystyle\frac{1}{2}&(\text{$n$:odd})
\end{cases}$. We take $(j_0,k_0)=(\dfrac{1}{2},\delta)$. Then $l(j_0,k_0)=n-\dfrac{1}{2}$. 
Let $\{\sum_{1\le\alpha\le  2m}f_{\alpha}^{(i)}\theta_{\alpha}\}_{1\le i\le m(\cP_{s}^{\tau})}$ be a basis of $\cS_{\cP_{s}^{\tau}}(SL_2(\Z);\Phi_m)$ and recall that we have let $\{\phi_{\cP_{s}^{\tau}}^{(i)}\}_{1\le i\le m(\cP_{s}^{\tau})}$ be a corresponding basis of $\Hom_{G_J}(\rho_{m,\cP_{s}^{\tau}},\cH_m^0)$. 
We have 
\begin{align*}
c_{\frac{1}{2},\delta}^{(\cP_{s}^{1/2})}(a_J)\phi_{\cP_{s}^{1/2}}^{(i)}(w_{n-1/2}&\otimes h_{0})(n_Jb)=
e^{2\pi ma_1^2}G_{2,3}^{3,0}\left(4\pi ma_1^2\left|
\begin{array}{c}
\displaystyle\frac{z_0+\delta+5/2}{2},\displaystyle\frac{-z_0+\delta+5/2}{2}\\
\displaystyle\frac{n+\delta+1}{2},\displaystyle\frac{z+2}{2},\displaystyle\frac{-z+2}{2}
\end{array}
\right.\right)\\
&\times\sum_{1\le \alpha\le 2m}(\frac{1}{\prod_{1\le l\le \frac{n-1}{2}}(s+2l-1/2)}\cdot (2V^+)^{\frac{n-1}{2}}\cdot f_{\alpha}^{(i)})(b)\theta_{\alpha}(\omega_m(b)h_0)(n_J),
\end{align*}
when $n$ is odd and $\tau=\displaystyle\frac{1}{2}$, and 
\begin{align*}
c_{\frac{1}{2},\delta}^{(\cP_{s}^{-1/2})}(a_J)\phi_{\cP_s^{-1/2}}^{(i)}(w_{n-1/2}&\otimes h_{0})(n_Jb)=e^{2\pi ma_1^2}G_{2,3}^{3,0}\left(4\pi ma_1^2\left|
\begin{array}{c}
\displaystyle\frac{z_0+\delta+5/2}{2},\displaystyle\frac{-z_0+\delta+5/2}{2}\\
\displaystyle\frac{n+\delta+1}{2},\displaystyle\frac{z+2}{2},\displaystyle\frac{-z+2}{2}
\end{array}
\right.\right)\\
&\times\sum_{1\le \alpha\le 2m}(\frac{1}{\prod_{1\le l\le\frac{n}{2}}(s+2l-3/2)}\cdot (2V^+)^{n/2}\cdot f_{\alpha}^{(i)})(b)\theta_{\alpha}(\omega_m(b)h_0)(n_J),
\end{align*}
when $n$ is even and $\tau=-\displaystyle\frac{1}{2}$. 
Here we have remarked in (I-2) that the infinitesimal action by $V^+$~(cf.~(\ref{Infinitesimal-MaassOp})) is known as the Maass weight raising operator in the classical setting.

As in (II-1) one wants to understand the Jacobi cusp forms $\phi^{(i)}_{m,\alpha}(\cP_{s}^{\tau})$ explicitly also for this case. However, the higher derivatives of $f_{\alpha}^{(i)}$ with respect to $V^+$ are too complicated to write down explicitly. We do not thus 
go into further details. 
\subsection*{(III-2)~The case of $\pi_1=\cD_{n_1}^+$ and $m>0$}
As in (I-1) and (II-3) $\cS_{\pi_1}(\widetilde{SL}_2(\Z),\Phi_m)$ denotes the space of $\Phi_m$-valued holomorphic cusp forms of half integral weight $n_1$, where $n_1\le n-\displaystyle\frac{1}{2}$. 

As in (III-1) we put $(j_0,k_0)=(\displaystyle\frac{1}{2},\delta)$. Then $l(\frac{1}{2},\delta)=n-\displaystyle\frac{1}{2}$. 
Suppose $n_1+\displaystyle\frac{1}{2}\equiv n\mod 2$ and $n_1<n-\frac{1}{2}$. We have
\begin{align*}
c_{\frac{1}{2},\delta}^{(\cD_{n_1}^+)}(a_J)\phi_{\cD_{n_1}^+}^{(i)}(w_{n-\frac{1}{2}}&\otimes h_0)(n_Jb)=e^{2\pi ma_1^2}G_{2,3}^{3,0}\left(4\pi ma_1^2\left|~
\begin{array}{c}
\displaystyle\frac{n_1+\delta+3/2}{2},\displaystyle\frac{-n_1+\delta+7/2}{2}\\
\displaystyle\frac{n+\delta+1}{2},\displaystyle\frac{z+2}{2},\displaystyle\frac{-z+2}{2}
\end{array}\right.\right)\\
&\times\sum_{1\le \alpha\le 2m}(\displaystyle\frac{1}{\prod_{1\le l\le\frac{n-n_1-1/2}{2}}(2n_1+2(l-1))}(2V^+)^{\frac{n-n_1-1/2}{2}}\cdot f_{\alpha}^{(i)})(b)\theta_{\alpha}(\omega_m(b)h_0)(n_J).
\end{align*}
For $n_1=n-\displaystyle\frac{1}{2}$ we have $c_{\frac{1}{2},\delta}^{(\cD_{n-\frac{1}{2}}^+)}(a_J)\phi_{\cD_{n-\frac{1}{2}}}^{(i)}(w_{n-\frac{1}{2}}\otimes h_0)(n_Jb)=$
\[
e^{2\pi ma_1^2}G_{2,3}^{3,0}\left(4\pi ma_1^2\left|~
\begin{array}{c}
\displaystyle\frac{n+\delta+1}{2},\displaystyle\frac{-n+\delta+4}{2}\\
\displaystyle\frac{n+\delta+1}{2},\displaystyle\frac{z+2}{2},\displaystyle\frac{-z+2}{2}
\end{array}\right.\right)\sum_{1\le \alpha\le 2m}\phi_{m,\alpha}^{(i)}(\cD_{n-\frac{1}{2}}^+)(n_Jb).
\]
Here we put 
\[
\phi_{m,\alpha}^{(i)}(\cD_{n_1}^+)(n_Jb):=f_{\alpha}^{(i)}(b)\theta_{\alpha}(\omega_m(b)h_0)(n_J),
\]
which coincides with 
\[
\sum_{
\begin{subarray}{c}
t\in\Z\\
l\in\frac{1}{4m}\Z_{>0}
\end{subarray}}c_d(l)y^{\frac{n_1}{2}+\frac{1}{4}}e^{-2\pi my(t+\frac{\alpha}{2m})^2}\e(m(u_0^2z)+(2mt+\alpha)u_0z)\exp(-2\pi ly)\e(\Tr(T_{t,l}X),
\]
where  $T_{t,l}$ and $X$ are as in (I-1). 
The form $\sum_{1\le \alpha\le 2m}\phi_{m,\alpha}^{(i)}(\cD_{n-\frac{1}{2}}^+)(n_Jb)$ is a holomorphic Jacobi cusp form since this is a sum of products of holomorphic cusp forms and holomorphic Jacobi theta series.
\subsection*{(III-3)~The case of $\pi_1=\cD_{n_1}^+$ and $m<0$}
Let  $\cS_{\pi_1}(\widetilde{SL}_2(\Z),\Phi_m)$ and $\sum_{1\le \alpha\le 2|m|}f_{\alpha}^{(i)}\theta_{\alpha}$ be as in (III-2), and assume $n_1>n-\displaystyle\frac{1}{2}$. 

We take $(j_0,k_0)=(n-n_1,\delta)$, and then $l(n-n_1,\delta)=n_1$. We have
\begin{align*}
c_{n-n_1,\delta}^{(\cD_{n_1}^+)}(a_J)\phi_{\cD_{n_1}^+}^{(i)}(w_{n_1}\otimes h_{n_1-n-\frac{1}{2}})(n_Jb)&=e^{-2\pi ma_1^2}G_{2,3}^{3,0}\left(-4\pi ma_1^2\left|~
\begin{array}{c}
\displaystyle\frac{n_1+\frac{3}{2}}{2},\displaystyle\frac{n_1+\frac{5}{2}}{2}\\
\displaystyle\frac{n+\delta+1}{2},\displaystyle\frac{z+2}{2},\displaystyle\frac{-z+2}{2}
\end{array}\right.\right)\\
&\times\sum_{1\le \alpha\le 2m}f_{\alpha}^{(i)}(b)\theta_{\alpha}(\omega_m(b)h_{n_1-n-\frac{1}{2}})(n_J).
\end{align*}
We introduce a new kind of a Jacobi form 
\[
\phi_{m,\alpha}^{(i)}(\cD_{n_1}^+)(n_Jb):=f_{\alpha}^{(i)}(b)\theta_{\alpha}(\omega_m(b)h_{n_1-n+\frac{1}{2}})(n_J),
\]
which equals to
\begin{align*}
\sum_{
\begin{subarray}{c}
t\in\Z\\
l\in\frac{1}{4|m|}\Z_{>0}
\end{subarray}}c_{\alpha}^{(i)}(l)y^{\frac{2n_1+1}{4}}H_{n_1-n-\frac{1}{2}}(\sqrt{y}(u_0+t+\frac{\alpha}{2m}))e^{2\pi m(y+\frac{\alpha}{2m})^2}\\
\times\e(m(u_0^2\bar{z})+(2mt+\alpha)u_0\bar{z})\exp(-2\pi ly)\e(\Tr(T_{t,l}X)
\end{align*}
with $T_{t,l}$ and $X$ as in (I-1). We have 
\begin{align*}
&c_{n-n_1,\delta}^{(\cD_{n_1}^+)}(a_J)\phi_{\cD_{n_1}^+}^{(i)}(w_{n_1}\otimes h_{n_1-n-\frac{1}{2}})(n_Jb)=\\
&e^{-2\pi ma_1^2}G_{2,3}^{3,0}\left(-4\pi ma_1^2\left|~
\begin{array}{c}
\displaystyle\frac{n+\frac{3}{2}}{2},\displaystyle\frac{n_1+\frac{5}{2}}{2}\\
\displaystyle\frac{n+\delta+1}{2},\displaystyle\frac{z+2}{2},\displaystyle\frac{-z+2}{2}
\end{array}\right.\right)\sum_{1\le \alpha\le 2|m|}\phi_{m,\alpha}^{(i)}(\cD_{n_1}^+)(n_Jb).
\end{align*}
\subsection*{(IV)~The case of principal series representations}
In what follows, let $\pi$ be an even principal series representation with $\sigma=(1,1)$ or an odd principal series representation with $\sigma=(1,-1)$. The notion of even principal series includes spherical principal series while odd principal series are non-spherical. From Section \ref{PS-rep} we recall that $\pi$ has $\tau_{(0,0)}$~(respectively~$\tau_{(1,0)}$ and $\tau_{(0.-1)}$) as a minimal $K$-type when $\pi$ is former~(respectively~the latter). Note that the multiplicity one $K$-type of $\pi$ is of the form $\tau_{(n,n)}$ or $\tau_{(n,n-1)}$ with $n\in\Z$~(cf.~Section \ref{PS-rep}) and that $n=0$~(respectively~$n=1$) for $\tau_{(0,0)}$ or $\tau_{(0,-1)}$~(respectively~$\tau_{(1,0)}$). To illustrate the examples of Jacobi cusp forms uniformly for the two cases we use the notation 
\[
\delta:=
\begin{cases}
0 & (\text{when $\pi$ is even and $\sigma=(1,1)$})\\
1 & (\text{when $\pi$ is odd and $\sigma=(1,-1)$})
\end{cases}
\]
and 
\[
(\tilde{z},\tilde{z}'):=
\begin{cases}
(z_2,z_1) & (\text{when $\pi$ is odd, $\sigma=(1,-1)$ and has the minimal $K$-type $\tau_{(1,0)}$})\\
(z_1,z_2) & (\text{otherwise})
\end{cases}
\]
(recall that $(\tilde{z},\tilde{z}')$ is introduced just before Theorem \ref{explicit-FJ-P-II}).

From Theorems \ref{explicit-FJ-P-I} and \ref{explicit-FJ-P-II} we know that there is a possibility of $\dim\cJ_{\rho_{\pi_1,m},\pi}(\tl^*)^{00}=1$ only if
\begin{itemize}
\item $m>0$ and $\pi_1=\cP_s^{\tau}~(\tau=\pm\frac{1}{2},~s\in\sqrt{-1}\R)$, $\cC_s^{\tau}~(\tau=\pm\frac{1}{2},~0<s<\frac{1}{2})$ or $\cD_{n_1}^-$,
\item $m<0$ and $\pi_1=\cP_s^{\tau}~(\tau=\pm\frac{1}{2},~s\in\sqrt{-1}\R)$, $\cC_s^{\tau}~(\tau=\pm\frac{1}{2},~0<s<\frac{1}{2})$ or $\cD_{n_1}^+$
\end{itemize}
for the multiplicity one $K$-types $\tl$ of $\pi$. 
The following examples appear in the Fourier-Jacobi expansion of cusp forms generating $\pi$ whose weights are given by the contragredients of the minimal $K$-types just mentioned~(cf.~Corollary \ref{FJ-exp-fourcases} (4)).
\subsection*{(IV-1)~The case of $\pi_1=\cP_s^{\tau}$}
As in the case of (II-1) $\cS_{\pi_1}(\widetilde{SL}_2(\Z);\Phi_m)$ is the space of $\Phi_m$-valued Maass cusp forms of weight $\tau\in\{\pm\dfrac{1}{2}\}$.\\
(i)~Suppose first that $m>0$ and choose $\tau_{(0,-1)}$~(respectively~$\tau_{(0,0)}$) as a minimal $K$-type of $\pi$ when $\pi$ is odd~(respectively~even). We take $(j_0,k_0)=(\frac{1}{2},\delta)$ and then $l(j_0,k_0)=-\displaystyle\frac{1}{2}$. For this case we see that $\tau$ should be $-\dfrac{1}{2}$. 
For each $i$ with $1\le i\le m(\cP_s^{-1/2})$ we have introduced $\sum_{1\le\alpha\le 2m}f_{\alpha}^{(i)}\theta_{\alpha}\in\cS_{\cP_s^{-1/2}}(SL_2(\Z);\Phi_m)$ in (II-1). 

For each $i$ with $1\le i\le m(\cP_s^{-1/2})$ we see that $c_{\frac{1}{2},\delta}^{(\cP_s^{-1/2})}(a_J)\phi_{\cP_s^{\tau}}^{(i)}(w_{-\frac{1}{2}}\otimes h_0)(n_Jb)$
\[
=e^{2\pi ma_1^2}G_{4,0}^{3,4}\left(4\pi m a_1^2\left|
\begin{array}{c}
\dfrac{s+\delta+\frac{5}{2}}{2},\dfrac{-s+\delta+\frac{5}{2}}{2},\dfrac{3+\delta}{2}\\
\dfrac{z_1+2+\delta}{2},\dfrac{-z_1+2+\delta}{2},\dfrac{z_2+2}{2},\dfrac{-z_2+2}{2}
\end{array}\right.\right)\sum_{1\le \alpha\le 2m}f_{\alpha}^{(i)}(b)\theta_{\alpha}(\omega_m(b)h_0)(n_J)
\]
for $(n_J,b)\in N_J\times B$.

As in (II-1) let
\begin{align*}
&\phi_{m,\alpha}^{(i)}(\cP_s^{-1/2})(n_Jb):=\\
&\sum_{
\begin{subarray}{c}
t\in\Z\\
l\in\frac{1}{4m}\Z\setminus\{0\}
\end{subarray}}c_{\alpha}^{(i)}(l)y^{\frac{1}{4}}e^{-2\pi my(t+\frac{\alpha}{2m})^2}\e(m(u_0^2z)+(2tm+\alpha)u_0z)W_{-{\rm sign}(l)\frac{1}{4},s}(4\pi|l|y)\e(\Tr(T_{t,l}X)
\end{align*} 
for $i$ with $1\le i\le m(\cP_s^{-1/2})$. 
We then have 
\begin{align*}
&c_{\frac{1}{2},\delta}^{(\cP_s^{-1/2})}(a_J)\phi_{\cP_s^{-1/2}}^{(i)}(w_{-\frac{1}{2}}\otimes h_0)(n_Jb)=\\
&e^{2\pi ma_1^2}G_{4,0}^{3,4}\left(4\pi m a_1^2\left|
\begin{array}{c}
\dfrac{s+\delta+\frac{5}{2}}{2},\dfrac{-s+\delta+\frac{5}{2}}{2},\dfrac{3+\delta}{2}\\
\dfrac{z_1+\delta+2}{2},\dfrac{-z_1+\delta+2}{2},\dfrac{z_2+2}{2},\dfrac{-z_2+2}{2}
\end{array}\right.\right)\sum_{1\le \alpha\le 2m}\phi_{m,\alpha}^{(i)}(\cP_{s}^{-1/2})(n_Jb),
\end{align*}
where $\sum_{1\le \alpha\le 2m}\phi_{m,\alpha}^{(i)}(\cP_{s}^{-1/2})(n_Jb)$ is a Maass Jacobi cusp form.\\
(ii)~Suppose next that $m<0$ and choose $\tau_{(1,0)}$~(respectively~$\tau_{(0,0)}$) as a minimal $K$-type  of $\pi$ when $\pi$ is odd~(respectively~even). 
We take $(j_0,k_0)=(-\dfrac{1}{2},0)$ and then $l(j_0,k_0)=\dfrac{1}{2}$. We see that $\tau$ should be $\dfrac{1}{2}$. 
For each $i$ with $1\le i\le m(\cP_s^{1/2})$ let $\sum_{1\le \alpha\le 2|m|}f_{\alpha}^{(i)}\theta_{\alpha}\in\cS_{\cP_s^{1/2}}(SL_2(\Z),\Phi_m)$. From \cite[Section 2]{Sar} we know that each cusp form $f_{\alpha}^{(i)}$ has a Fourier expansion
\[
f_{\alpha}^{(i)}(b)=\sum_{l\not=0,~l\in\frac{1}{4m}\Z}c_{\alpha}^{(i)}(l)W_{{\rm sign}(l)\frac{1}{4},s}(4\pi|l|y)\e(lx)
\]
with Fourier coefficients $\{c_{\alpha}^{(i)}(l)\}$, 
where see \cite[Chapter 16]{W-W} and Theorem \ref{LargeDSWhittaker} for $W_{\kappa,\mu}$. 

For each $i$ with $1\le i\le m(\cP_s^{1/2})$ we see that $c_{-\frac{1}{2},0}^{(\cP_s^{1/2})}(a_J)\phi_{\cP_s^{\tau}}^{(i)}(w_{\frac{1}{2}}\otimes h_0)(n_Jb)=$
\[
e^{-2\pi ma_1^2}G_{4,0}^{3,4}\left(-4\pi m a_1^2\left|
\begin{array}{c}
\dfrac{s+\delta+\frac{5}{2}}{2},\dfrac{-s+\delta+\frac{5}{2}}{2},\dfrac{3+\delta}{2}\\
\dfrac{\tilde{z}+2}{2},\dfrac{-\tilde{z}+2}{2},\dfrac{\tilde{z}'+\delta+2}{2},\dfrac{-\tilde{z}'+\delta+2}{2}
\end{array}\right.\right)\sum_{1\le \alpha\le 2m}f_{\alpha}^{(i)}(b)\theta_{\alpha}(\omega_m(b)h_0)(n_J)
\]
for $(n_J,b)\in N_J\times B$. 
Now let 
\[
\phi_{m,\alpha}^{(i)}(\cP_s^{1/2})(n_Jb):=f_{\alpha}^{(i)}(b)\theta_{\alpha}(\omega_m(b)h_0)(n_J),
\]
which coincides with
\[
\sum_{
\begin{subarray}{c}
t\in\Z\\
l\in\frac{1}{4m}\Z\setminus\{0\}
\end{subarray}}c_{\alpha}^{(i)}(l)y^{\frac{1}{4}}e^{2\pi my(t+\frac{\alpha}{2m})^2}\e(m(u_0^2\bar{z})+(2tm+\alpha)u_0\bar{z})W_{{\rm sign}(l)\frac{1}{4},s}(4\pi|l|y)\e(\Tr(T_{t,l}X)
\]
for $i$ with $1\le i\le m(\cP_s^{1/2})$. 
We then have 
\begin{align*}
&c_{-\frac{1}{2},0}^{(\cP_s^{1/2})}(a_J)\phi_{\cP_s^{1/2}}^{(i)}(w_{\frac{1}{2}}\otimes h_0)(n_Jb)=\\
&e^{-2\pi ma_1^2}G_{4,0}^{3,4}\left(-4\pi m a_1^2\left|
\begin{array}{c}
\dfrac{s+\delta+\frac{5}{2}}{2},\dfrac{-s+\delta+\frac{5}{2}}{2},\dfrac{3+\delta}{2}\\
\dfrac{\tilde{z}+2}{2},\dfrac{-\tilde{z}+2}{2},\dfrac{\tilde{z}'+\delta+2}{2},\dfrac{-\tilde{z}'+\delta+2}{2}
\end{array}\right.\right)\sum_{1\le \alpha\le 2m}\phi_{m,\alpha}^{(i)}(\cP_{s}^{1/2})(n_Jb).
\end{align*}
The form $\sum_{1\le \alpha\le 2m}\phi_{m,\alpha}^{(i)}(\cP_{s}^{1/2})(n_Jb)$ is a Maass Jacobi cusp form.
\subsection*{(IV-2) The case of $\pi_1=\cD_{n_1}^-$ and $m>0$}
As in (II-2) $\cS_{\pi_1}(\widetilde{SL}_2(\Z),\Phi_m)$ is the space of $\Phi_m$-valued anti-holomorphic cusp forms of half  integral weight $-n_1$. Choose $\tau_{(1,0)}$~(respectively~$\tau_{(0,0)}$) as a minimal $K$-type of $\pi$ when $\pi$ is odd~(respectively~even). 

Suppose that $n_1>\displaystyle\frac{3}{2}-\delta$. Let us take $(j_0,k_0)=(n_1+\delta,0)$, and then  $l(j_0-\delta,k_0)=-n_1$, where we use the notation $j_0$ following Theorems \ref{explicit-FJ-P-I} and \ref{explicit-FJ-P-II}. 
For each $i$ with $1\le i\le m(\cD_{n_1}^-)$ we see that $c_{n_1,0}^{(\cD_{n_1}^-)}(a_J)\phi_{\cD_{n_1}^-}^{(i)}(w_{-n_1}\otimes h_{n_1-\frac{1}{2}})(n_Jb)=$
\[
e^{2\pi ma_1^2}G_{4,0}^{3,4}\left(4\pi m a_1^2\left|
\begin{array}{c}
\dfrac{n_1+\frac{5}{2}}{2},\dfrac{n_1+\frac{3}{2}}{2},\dfrac{3+\delta}{2}\\
\dfrac{\tilde{z}+2}{2},\dfrac{-\tilde{z}+2}{2},\dfrac{\tilde{z}'+\delta+2}{2},\dfrac{-\tilde{z}'+\delta+2}{2}
\end{array}\right.\right)
\sum_{1\le \alpha\le 2m}f_{\alpha}^{(i)}(b)\theta_{\alpha}(\omega_m(b)h_{n_1-\frac{1}{2}})(n_J)
\]
for $(n_J,b)\in N_J\times B$, where $\sum_{1\le\alpha\le 2m}f_{\alpha}^{(i)}\theta_{\alpha}\in\cS_{\cD_{n_1}^{-}}(\widetilde{SL}_2(\Z),\Phi_m)$ with anti-holomorphic cusp forms $f_{\alpha}^{(i)}$ with respect to $\widetilde{\Gamma(4m)}$ introduced in (II-2). 
We introduce 
\[
\phi_{m,\alpha}^{(i)}(\cD_{n_1}^-)(n_Jb):=f_{\alpha}^{(i)}(b)\theta_{\alpha}(\omega_m(b)h_{n_1-\frac{1}{2}})(n_J),
\]
which coincides with
\begin{align*}
&\sum_{
\begin{subarray}{c}
t\in\Z\\
l\in\frac{1}{4m}\Z_{<0}
\end{subarray}}c_{\alpha}(-l)y^{\frac{n_1}{2}+\frac{1}{4}}H_{n_1-\frac{1}{2}}(\sqrt{y}(u_0+t+\frac{\alpha}{2m}))e^{-2\pi my(t+\frac{\alpha}{2m})^2}\\
&\times\e(m(u_0^2z)+(2mt+\alpha)u_0z)\exp(2\pi ly)\e(\Tr(T_{t,l}X).
\end{align*}
We then have $c_{n_1,0}^{(\cD_{n_1}^-)}(a_J)\phi_{\cD_{n_1}^-}^{(i)}(w_{-n_1}\otimes\omega_m(b)h_{n_1-\frac{1}{2}})(n_Jb)=$
\[
e^{2\pi ma_1^2}G_{4,0}^{3,4}\left(4\pi m a_1^2\left|
\begin{array}{c}
\dfrac{n_1+\frac{5}{2}}{2},\dfrac{n_1+\frac{3}{2}}{2},\dfrac{3+\delta}{2}\\
\dfrac{\tilde{z}+2}{2},\dfrac{-\tilde{z}+2}{2},\dfrac{\tilde{z}'+\delta+2}{2},\dfrac{-\tilde{z}'+\delta+2}{2}
\end{array}\right.\right)\sum_{1\le \alpha\le 2m}\phi_{m,\alpha}^{(i)}(\cD_{n_1}^-)(n_Jb).
\]
\subsection*{(IV-3)~The case of $\pi_1=\cD_{n_1}^+$ and $m<0$}
For this case $\cS_{\pi_1}(\widetilde{SL}_2(\Z),\Phi_m)$ is the space of $\Phi_m$-valued holomorphic cusp forms of half integral weight $n_1$. Choose $\tau_{(1,0)}$~(respectively~$\tau_{(0,0)}$) as a minimal $K$-type of $\pi$ when $\pi$ is odd~(respectively~even).

Suppose that $n_1>\displaystyle\frac{3}{2}+\delta$. Take $(j_0,k_0)=(-n_1+\delta,\delta)$ and then $l(j_0,k_0)=n_1$. For each $i$ with $1\le i\le m(\cD_{n_1}^+)$ we see that $c_{-n_1+\delta,\delta}^{(\cD_{n_1}^+)}(a_J)\phi_{\cD_{n_1}^+}^{(i)}(w_{n_1}\otimes h_{n_1-\delta-\frac{1}{2}})(n_Jb)=$
\begin{align*}
&e^{-2\pi ma_1^2}G_{4,0}^{3,4}\left(-4\pi m a_1^2\left|
\begin{array}{c}
\dfrac{n_1+\frac{5}{2}}{2},\dfrac{n_1+\frac{3}{2}}{2},\dfrac{3}{2}\\
\dfrac{\tilde{z}+\delta+2}{2},\dfrac{-\tilde{z}+\delta+2}{2},\dfrac{\tilde{z}'+2}{2},\dfrac{-\tilde{z}'+2}{2}
\end{array}\right.\right)\\
&\times\sum_{1\le d\le 2|m|}f_{\alpha}^{(i)}(b)\theta_{\alpha}(\omega_m(b)h_{n_1-\delta-\frac{1}{2}})(n_J)
\end{align*}
for $(n_J,b)\in N_J\times B$, where $\sum_{1\le\alpha\le 2|m|}f_{\alpha}^{(i)}\theta_{\alpha}\in\cS_{\cD_{n_1}^{+}}(\widetilde{SL}_2(\Z),\Phi_m)$ with holomorphic cusp forms $f_{\alpha}^{(i)}$ with respect to $\widetilde{\Gamma(4m)}$ . Each $f_{\alpha}^{(i)}$ has the Fourier expansion as in (II-3). 
We then introduce 
\[
\phi_{m,\alpha}^{(i)}(\cD_{n_1}^+)(n_Jb):=f_{\alpha}^{(i)}(b)\theta_{\alpha}(\omega_m(b)h_{n_1-\delta-\frac{1}{2}})(n_J),
\]
which can be written as 
\begin{align*}
\sum_{
\begin{subarray}{c}
t\in\Z\\
l\in\frac{1}{4|m|}\Z_{>0}
\end{subarray}}&c_{\alpha}(l)y^{\frac{n_1}{2}+\frac{1}{4}}H_{n_1-\delta-\frac{1}{2}}(\sqrt{y}(u_0+t+\frac{\alpha}{2m}))e^{2\pi my(t+\frac{\alpha}{2m})^2}\\
&\times\e(m(u_0^2\bar{z})+(2mt+\alpha)u_0\bar{z})\exp(-2\pi ly)\e(\Tr(T_{t,l}X).
\end{align*}
We then have $c_{-n_1+\delta,\delta}^{(\cD_{n_1}^+)}(a_J)\phi_{\cD_{n_1}^+}^{(i)}(w_{n_1}\otimes h_{n_1-\delta-\frac{1}{2}})(n_Jb)=$
\[
e^{-2\pi ma_1^2}G_{4,0}^{3,4}\left(-4\pi m a_1^2\left|
\begin{array}{c}
\dfrac{n_1+\frac{5}{2}}{2},\dfrac{n_1+\frac{3}{2}}{2},\dfrac{3}{2}\\
\dfrac{\tilde{z}+\delta+2}{2},\dfrac{-\tilde{z}+\delta+2}{2},\dfrac{\tilde{z}'+2}{2},\dfrac{-\tilde{z}'+2}{2}
\end{array}\right.\right)\sum_{1\le \alpha\le 2|m|}\phi_{m,\alpha}^{(i)}(\cD_{n_1}^+)(n_Jb).
\]
\newpage
\section{Concluding remarks}
To conclude this paper we make a couple of remarks toward possible works in future.
\begin{enumerate}
\item As the discrete subgroup to define cusp forms on $Sp(2,\R)$ we have chosen the Siegel modular group $Sp(2,\Z)$. One naturally suggests the study for the cases of other discrete subgroups such as congruence subgroups of $Sp(2,\Z)$. We can say that it would be more or less possible. However, we cannot always say that the Fourier-Jacobi expansions for the cases of other discrete subgroups are deduced in a straightforward manner, following the $Sp(2,\Z)$-case.
 
For instance we use the left invariance with respect to $\xi\in Sp(2,\Z)$ to show Lemma \ref{Whittaker-F_m}, which is necessary to describe the $F_m^c$-term of the Fourier-Jacobi expansion for non-zero integers $m$. 
For discrete subgroups of $Sp(2,\R)$ without $\xi$ we need some modification of the argument. 
Another possible difficulty is that we will have to consider the representation theoretic Eichler-Zagier correspondence~(cf.~Theorem \ref{Eichler-ZagierCorresp}) for the case of a higher level, e.g.  if we study the Fourier-Jacobi expansion for the case of a congruence subgroup of a higher level. 
There are results on the Eichler-Zagier correspondences of higher level cases~(for instance see \cite{Ma-Ra}). 
In this paper we have not tried the representation theoretic reformulation of the correspondence for such cases. 
\item A natural question is whether we can consider the Fourier-Jacobi expansion for the groups other than $Sp(2,\R)$. Under suitable working assumptions we would be able to discuss the Fourier-Jacobi expansion in a more general setting. However, as is remarked already in the introduction, there seems no study of the Fourier-Jacobi type spherical functions other than Hirano's works \cite{Hi-1},~\cite{Hi-2} and \cite{Hi-3}. Only for the case of $Sp(2,\R)$ we can explicitly know the Fourier-Jacobi expansions of cusp forms including non-holomorphic ones. 

In addition, recall that we have to study the Whittaker functions for degenerate characters of the maximal unipotent subgroup to know the Fourier-Jacobi expansion of cusp forms. There seem to be quite a few papers taking up such Whittaker functions. Taku Ishii pointed me out that the paper \cite{Hz} by Hashizume studied the Whittaker functions on a general semi-simple real Lie group including the cases of degenerate characters. 
We also remark that Whittaker functions attached to degenerate characters are necessary to understand the Fourier expansion (not specialized to the Fourier-Jacobi one) for non-cuspidal automorphic forms, which we do not deal with in this paper.
\item It is possible to discuss the Fourier-Jacobi expansion in the adelic setting. Berndt-Schmidt \cite[Corollary 7.3.6]{Be-Sc} shows that the total space of adelic Jacobi cusp forms decomposes discretely into a sum of irreducible pieces with finite multiplicities. Berndt-Schmidt \cite[Theorem 7.3.3]{Be-Sc} can be viewed as the generalized Eichler-Zagier correspondence between the total spaces of adelic Jacobi cusp forms and $L^2(\bA)$-valued cusp forms on an adelic metaplectic group, where $L^2(\bA)$ denotes the  $L^2$-space on the adele ring $\bA$ of $\Q$. For this see the argument around the equation (6.8) in \cite{Ps-2}. 
The adelic treatment of Jacobi forms by Berndt-Schmidt \cite{Be-Sc} can be of great use for the adelic Fourier-Jacobi expansion, collaborating with the notion of a global Whittaker model and a global Fourier-Jacobi model. We also refer to Ikeda \cite{Ik-1} for an adelic Fourier-Jacobi expansion of Einsenstein series, which has an important application to a refinement of the regularized Siegel-Weil formula \cite{Kud-Ra} by Kudla-Rallis.

However, with the ingredients above we know the Fourier-Jacobi expansion only in an abstract manner. 
We are interested in an explicit description of the Fourier-Jacobi expansion in terms of spherical functions on the real group of $Sp(2,\R)$. We have thus followed the non-adelic formulation in this paper. 
\end{enumerate}

\newpage
\noindent
Hiro-aki Narita\\
Department of Mathematics\\
Faculty of Science and Engineering\\
Waseda University\\
3-4-1 Okubo, Shinjuku-ku, Tokyo 169-8555\\
JAPAN\\
E-mail:~hnarita@waseda.jp
\end{document}